\documentclass[12pt]{article}
\usepackage{algorithm}
\usepackage{algpseudocode}
\usepackage{amsthm}
\usepackage{amsmath}
\usepackage{amsfonts}
\usepackage{mathrsfs}
\usepackage{array}
\usepackage{amssymb}
\usepackage{units}
\usepackage{graphicx}
\usepackage{tikz}
\usetikzlibrary{cd}
\usetikzlibrary{fit,patterns}
\usepackage{nicefrac}
\usepackage{bbm}
\usepackage{color}
\usepackage{tensor}
\usepackage{tipa}
\usepackage{bussproofs}
\usepackage{ stmaryrd }
\usepackage{ textcomp }
\usepackage{leftidx}
\usepackage{afterpage}
\usepackage{varwidth}
\usepackage{tasks}
\usepackage{ cmll }
\usepackage{makecell}
\usepackage{MnSymbol}
\usepackage{quiver}
\usepackage{adjustbox}
\usepackage{multirow}
\usepackage{booktabs}
\usepackage{xparse}
\usepackage{calc}
\usepackage[all]{xy}

\definecolor{Myblue}{rgb}{0,0,0.6}
\usepackage[a4paper,colorlinks,citecolor=Myblue,linkcolor=Myblue,urlcolor=Myblue,pdfpagemode=None]{hyperref}

\ExplSyntaxOn
\makeatletter
\newcommand{\CMidRule}{\noalign\bgroup\@CMidRule{}}
\NewDocumentCommand{\@CMidRule}{
	m 
	O{0.0ex} 
	O{0.0ex} 
	m  
}{
	\peek_meaning_remove_ignore_spaces:NTF \CMidRule
	{ \@CMidRule { #1 \cmidrule[\cmidrulewidth](l{#2}r{#3}){#4} } }
	{ \egroup #1 \cmidrule[\cmidrulewidth](l{#2}r{#3}){#4} }
}
\makeatother
\ExplSyntaxOff

\graphicspath{ {images/} }

\SelectTips{cm}{}

\theoremstyle{plain}
\newtheorem{thm}{Theorem}[section] 
\newtheorem{proposition}[thm]{Proposition}
\newtheorem{lemma}[thm]{Lemma}

\newtheorem{cor}[thm]{Corollary}

\theoremstyle{definition}
\newtheorem{defn}[thm]{Definition} 

\newtheorem{remark}[thm]{Remark}

\newtheorem{example}[thm]{Example}

\newcommand{\scr}[1]{\mathscr{#1}}
\newcommand{\call}[1]{\mathcal{#1}}

\newcommand{\comment}[1]{}
\newcommand{\lto}{\longrightarrow}

\newcommand{\cut}{(\operatorname{cut})}
\newcommand{\ax}{(\operatorname{ax})}

\newcommand{\inc}{\iota}

\newcommand{\LM}{\operatorname{LM}}
\newcommand{\LT}{\operatorname{LT}}
\def\l{\,|\,}
\newcommand{\ediv}[2]{\operatorname{div}(#1,#2)}
\newcommand{\edives}[2]{\operatorname{div}_{es}(#1,#2)}
\newcommand{\multideg}{\operatorname{multideg}}

\DeclareRobustCommand{\diamondtimes}{%
	\mathbin{\text{\rotatebox[origin=c]{45}{$\boxplus$}}}%
}
\newcommand{\tagarray}{\mbox{}\refstepcounter{equation}$(\theequation)$}

\newcommand\showdiv[1]{\overline{\smash{)}#1}}

{\gdef\scalefactor{#1}\begin{center}\proofSkipAmount \leavevmode}%
	{\scalebox{\scalefactor}{\DisplayProof}\proofSkipAmount \end{center} }

\title{Elimination and cut-elimination in multiplicative linear logic}
\author{Daniel Murfet, William Troiani}
\date{\today}

\begin{document}
\maketitle

\begin{abstract}
We associate to every proof structure in multiplicative linear logic an ideal which represents the logical content of the proof as polynomial equations. We show how cut-elimination in multiplicative proof nets corresponds to instances of the Buchberger algorithm for computing Gröbner bases in elimination theory.
\end{abstract}

\tableofcontents
	
\section{Introduction}

The correspondences between logic and computation that can be found in the work  Curry-Howard \cite{howard} and Gentzen-Mints-Zucker \cite{gmz} have as their deep content the observation that intuitionistic proofs can be understood as \emph{patterns of equality} that relate differently named occurrences of the ``same'' variable. Such patterns can be enacted on other patterns in a dynamical process that we call computation. By now we have many useful ways of thinking about these patterns as static objects, but the mathematical theory of these patterns in motion, what Girard calls \emph{geometry of interaction} \cite{towards_goi}, is not yet as well-developed. 

This paper is the first of a series in which we bring a fresh perspective to this geometry of interaction, using algebraic geometry. Here we only consider the simplest kind of proof, multiplicative proof nets in linear logic, and their cut-elimination process. We assign to every multiplicative proof net $\pi$ (from here on out just \emph{proof net}) a polynomial ring $P_\pi$, in which the variables are occurrences of atomic propositions in the formulas labelling edges of $\pi$, and an ideal $I_\pi$ in this ring (Definition \ref{defn:ideal_proof}). The generating set $G_\pi$ for this ideal contains, in particular, for every $\ax$ or $\cut$ link in $\pi$ on an atomic proposition $A$
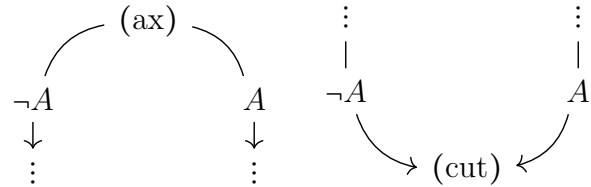
\begin{center}
			\begin{tabular}{c c}
				$
				\begin{tikzcd}[column sep = small, row sep = small]
					& \ax\arrow[dl,bend right, dash]\arrow[dr,bend left, dash]\\
					\neg A\arrow[d] & & A\arrow[d]\\
					\vdots & & \vdots
				\end{tikzcd}$
				&
				$
					\begin{tikzcd}[column sep = small, row sep = small]
						\vdots\arrow[d, dash] & & \vdots\arrow[d, dash]\\
						\neg A\arrow[dr,bend right] & & A\arrow[dl, bend left]\\
						& \cut
					\end{tikzcd}$
			\end{tabular}
		\end{center}
a polynomial $A' - A$ where $A, A'$ are two distinct copies of the same atomic proposition, one for each edge incident at the link (informally both links say $A = A$ and this equation is in $G_\pi$). For any reduction $\gamma$ of $\pi$ yielding another proof net $\pi'$ there is a natural inclusion $P_{\pi'} \subseteq P_\pi$ where the variables in $P_\pi$ not in $P_{\pi'}$ are atomic propositions on edges eliminated by $\gamma$. The relationship between $I_\pi$ and $I_{\pi'}$ is straightforward (Corollary \ref{cor:ideals_intersect})
\begin{equation}\label{eq:ideal_relation_intro}
I_{\pi'} = I_{\pi} \cap P_{\pi'}\,.
\end{equation}
If we think of $I_\pi$ as all consequences of the generating equations in $G_\pi$, then this says that the pattern of equalities between the occurrences of atomic propositions in $\pi'$ \emph{derived from} the equations in $G_\pi$ is the same as the pattern we would write down directly from the structure of $\pi'$. There is however an important subtlety here: the generating set $G_{\pi'}$ for $I_{\pi'}$ \emph{cannot} be derived from $G_\pi$ by simple intersection, as
\begin{equation}\label{eq:failure_G}
G_{\pi'} \neq G_{\pi} \cap P_{\pi'}\,.
\end{equation}
The difference between $G_\pi$ and $I_\pi$ is analogous to the difference between a set of axioms, and the set of all propositions which may be derived from those axioms. The reason for the lack of equality in \eqref{eq:failure_G} is easy to understand: if $\gamma$ eliminates a $\cut$ link as above then $A' - A \in G_\pi$, and if $A_1 - A', A - A_2 \in G_\pi$ are generators associated to other links in $\pi$ then
\[
A_1 - A_2 = (A_1 - A') + (A' - A) + (A - A_2) \in I_\pi
\]
and hence $A_1 - A_2 \in I_{\pi'} = I_\pi \cap P_\pi$. As we will see, this is one of the generators in $G_{\pi'}$ (since in the reduced proof net $A_1, A_2$ are incident at a single link). However, if we simply intersect $G_\pi$ with $P_\pi$ then we delete all three of the equations $A_1 - A', A' - A, A - A_2$ since each contains some eliminated variable, and the connection between $A_1$ and $A_2$ is lost.

In algebraic geometry, \emph{Elimination Theory} \cite[Ch. 3]{Grobner} provides a general framework for eliminating variables from systems of polynomial equations, in the sense that $A', A$ have been eliminated from $A_1 - A', A' - A, A - A_2$ to obtain $A_1 - A_2$. In our present circumstance the theory says the following:
\begin{itemize}
\item[(1)] Choose a lexicographic monomial order $<_\gamma$ for $P_\pi$ in which the variables eliminated by the reduction $\gamma$ (e.g. $A', A$) are large.
\item[(2)] Compute a Gröbner basis for $I_\pi$ from $G_\pi$ using this monomial order.
\item[(3)] Intersect the resulting Gröbner basis with $P_{\pi'}$.
\end{itemize}
The main theorem of Elimination Theory (recalled as Theorem \ref{thm:elimination} here) says that the intersection in (3) is a Gröbner basis for $I_{\pi'}$. We prove in Theorem \ref{thm:elimination_ours} that, if the monomial order is chosen carefully and a particular form of the Buchberger algorithm for computing Gröbner bases is used in (2), then (3) computes $G_{\pi'}$ (which is a Gröbner basis for $I_{\pi'}$ with respect to a monomial order determined by the ordering of atomic propositions along persistent paths). In this sense the Buchberger algorithm can be viewed as an algebraic interpretation of cut-elimination for multiplicative proof nets.
\\

The goals of this paper are to set up a basic dictionary between
\begin{itemize}
\item multiplicative proofs nets and ideals
\item reduction sequences and monomial orders
\item cut-elimination and elimination
\end{itemize}
While there is some small novelty in the form of the Buchberger algorithm that makes this connection clearest (we call it Buchberger with Early Stopping) there is nothing deep here logically, algebraically or geometrically. Indeed the ideals $I_\pi$ do not have a very interesting geometry: the associated affine variety is just an intersection of pairwise diagonals. In subsequent papers in this series we introduce, on top of the foundations laid here, more interesting algebra and geometry (see Section \ref{section:towards_geometry} for some hints).

\section{Motivation: the canonical detour}\label{section:intro_example}

The reader wishing to dive straight into the details should skip to Section \ref{section:ideal_of_proof}. In this section we recall some of the concepts behind proof nets, following \cite[\S 11.2.2]{bs}. The deduction rules for intuitionistic linear natural deduction include the introduction and elimination rule for linear implication:

\begin{center}
\begin{tabular}{ >{\centering}m{6cm} >{\centering}m{6cm} >{\centering}m{0.5cm}}
	\AxiomC{$[A]^{\bold{1}}$}
	\noLine
	\UnaryInfC{$\vdots$}
	\noLine
	\UnaryInfC{$B$}
	\RightLabel{$(\multimap I)^{\bold{1}}$}
	\UnaryInfC{$A \multimap B$}
	\DisplayProof
	&
	\AxiomC{}
	\noLine
	\UnaryInfC{$\vdots$}
	\noLine
	\UnaryInfC{$A$}
	\AxiomC{}
	\noLine
	\UnaryInfC{$\vdots$}
	\noLine
	\UnaryInfC{$A \multimap B$}
	\RightLabel{$(\multimap E)$}
	\BinaryInfC{$B$}
	\DisplayProof
	&
	\tagarray{\label{natural_deduction}}
\end{tabular}
\end{center}
Among other criticisms of this sytem of natural deduction, Girard points to the need for a non-local ``gimmick'' (the de Bruijn index $\bold{1}$) to ``physically'' link the hypothesis $A$ to the introduction rule \cite[\S 11.2.1]{bs}. This non-locality is resolved in proof nets, in two steps: firstly, by giving up active hypotheses and having only conclusions, and secondly by bringing these hypotheses-as-conclusions down to meet the introduction rule that employs them. The first step involves introducing the $\ax$ and $\cut$ links

\begin{center}
	\begin{tabular}{>{\centering}m{6cm} >{\centering}m{6cm}}
		$
		\begin{tikzcd}[column sep = small, row sep = small]
			& \ax\arrow[dl,bend right, dash]\arrow[dr,bend left, dash]\\
			\neg A & & A
		\end{tikzcd}$
		&
		$
			\begin{tikzcd}[column sep = small, row sep = small]
				\neg A\arrow[dr,bend right] & & A\arrow[dl, bend left]\\
				& \cut
			\end{tikzcd}$
	\end{tabular}
\end{center}
which ``replace a hypothesis by a conclusion and vice versa''. The second step, in which $\neg A$ is brought down to meet the introduction rule, is made clear by the proof net interpreting the introduction rule in \eqref{natural_deduction}
\begin{center}
\begin{tabular}{>{\centering}m{12cm} >{\centering}m{0.5cm}}
$
\begin{tikzcd}[column sep = small, row sep = small]
	& \ax\\
	\neg A && A \\
	&& B \\
	& \multimap \\
	& {A \multimap B} \\
	& \vdots
	\arrow[from=5-2, to=6-2]
	\arrow[curve={height=-12pt}, from=3-3, to=4-2]
	\arrow[curve={height=22pt}, from=2-1, to=4-2]
	\arrow[curve={height=12pt}, no head, from=1-2, to=2-1]
	\arrow[curve={height=-12pt}, no head, from=1-2, to=2-3]
	\arrow[no head, from=4-2, to=5-2]
	\arrow[dotted, no head, from=2-3, to=3-3]
\end{tikzcd}
$
&
\tagarray{\label{eq:intro_deduction_b}}
\end{tabular}
\end{center}
To explain Girard's ``vice versa'' note that we can read the elimination rule of \eqref{natural_deduction} as converting a proof of the \emph{conclusion} $A \multimap B$ into a proof of $B$ on \emph{hypothesis} $A$, and this is achieved in proof nets via $\cut$, as shown in the following diagram:
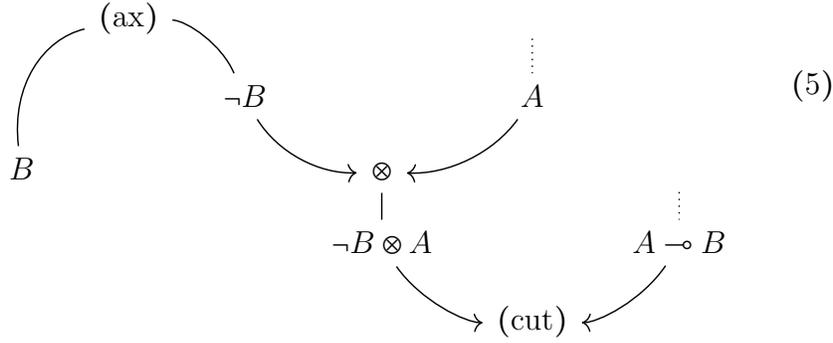
\begin{center}
\begin{tabular}{>{\centering}m{10cm} >{\centering}m{1cm}}
$
\begin{tikzcd}[column sep = small, row sep = small]
	& \ax & & & \,\\
	& & \neg B & & A & \,\\
	B &&& \otimes &&\, \\
	&&& \neg B \otimes A && A \multimap B\\
	&&&& \cut
	\arrow[dotted, no head, from=3-6, to=4-6]
	\arrow[curve={height=22pt}, no head, from=1-2, to=3-1]
	\arrow[curve={height=-12pt}, no head, from=1-2, to=2-3]
	\arrow[curve={height=12pt}, from=2-3,to=3-4]
	\arrow[curve={height=-12pt}, from=2-5,to=3-4]
	\arrow[no head, from=3-4,to=4-4]
	\arrow[curve={height=12pt},from=4-4,to=5-5]
	\arrow[curve={height=-12pt},from=4-6,to=5-5]
	\arrow[dotted, no head, from=1-5, to=2-5]
\end{tikzcd}
$
&
\tagarray{\label{eq:intro_deduction_c}}
\end{tabular}
\end{center}
The reduction step in natural deduction is the elimination of the ``detour'' created when the introduction rule in \eqref{natural_deduction} appears as the proof of $A \multimap B$ in the elimination rule; this detour is eliminated to obtain a more straightforward proof of $B$ by combining the given proof of $A$ and the proof of $B$ on hypothesis $A$. The canonical example of a detour is shown below:
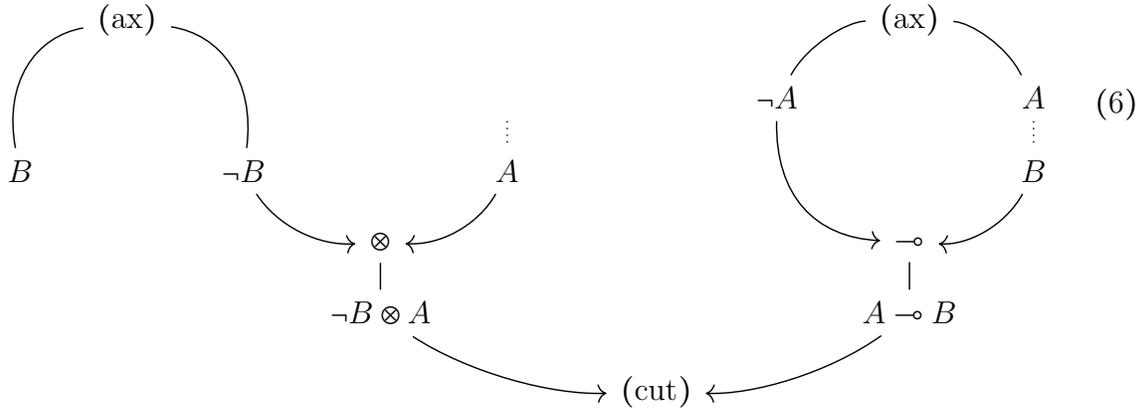
\begin{center}
\begin{tabular}{>{\centering}m{14cm} >{\centering}m{1cm}}
$
\begin{tikzcd}[column sep = small,row sep = small]
		& \ax &&&&&&& \ax \\
		&&&& \, &&& \neg A && A \\
		B && \neg B && A && & && B \\
		&&& \otimes &&& && \multimap \\
		&&& \neg B \otimes A &&&&& A \multimap B \\
		&&&&&& \cut
		\arrow[curve={height=25pt}, no head, from=1-2, to=3-1]
		\arrow[curve={height=-25pt}, no head, from=1-2, to=3-3]
		\arrow[dotted, no head, from=2-10, to=3-10]
		\arrow[curve={height=-12pt}, from=3-5, to=4-4]
		\arrow[curve={height=12pt}, from=3-3, to=4-4]
		\arrow[curve={height=12pt}, no head, from=1-9, to=2-8]
		\arrow[curve={height=-12pt}, no head, from=1-9, to=2-10]
		\arrow[curve={height=-12pt}, from=3-10, to=4-9]
		\arrow[curve={height=25pt}, from=2-8, to=4-9]
		\arrow[no head, from=4-9, to=5-9]
		\arrow[dotted, no head, from=2-5, to=3-5]
		\arrow[curve={height=-12pt}, from=5-9, to=6-7]
		\arrow[curve={height=12pt}, from=5-4, to=6-7]
		\arrow[no head, from=4-4, to=5-4]
	\end{tikzcd}
$
&
\tagarray{\label{eq:intro_deduction_d}}
\end{tabular}
\end{center}
In the language of (linear) lambda terms the proof net \eqref{eq:intro_deduction_b} is $\lambda x \,.\, M$ for some term $M$ determined by the proof of $B$ from $A$ in ellipsis. If the proof of $A$ in ellipsis in \eqref{eq:intro_deduction_c} has term $N$ then the cut \eqref{eq:intro_deduction_d} corresponds to $((\lambda x \,.\, M) N)$ reducing to $M[ x := N ]$ with the ``output'' routed through the $\cut$ and $\otimes$ links to the conclusion $B$ on the left. We can imagine a path (a more refined version of which we will later call a \emph{persistent path}) from the proof of $A$ down through the $\otimes$ link and $\cut$ links, up through the $\multimap$ link and around to the left through $\neg A$ to the $\ax$ link, then back down through the proof of $B$ from $A$ (the one with term $M$) through the $\multimap$ and $\cut$ links, then up through $\otimes$ and the $\ax$ for $B$ until it exits the conclusion $B$. Along this path ``nothing happens'' except when we pass through the proof of $B$ from $A$, and we will get the same result if we ``straighten out'' the rest of the proof net leaving just this interesting part \emph{with the proof of $A$ feeding into it}. This is the result of cut-elimination for this proof net.

Let us sketch the algebro-geometric point of view on this situation. We first distinguish every occurrence of the formulas $A,B$ in the proof net above by giving it an index\footnote{These indices are not a form of de Brujin index because they only matter locally, as we will explain.}
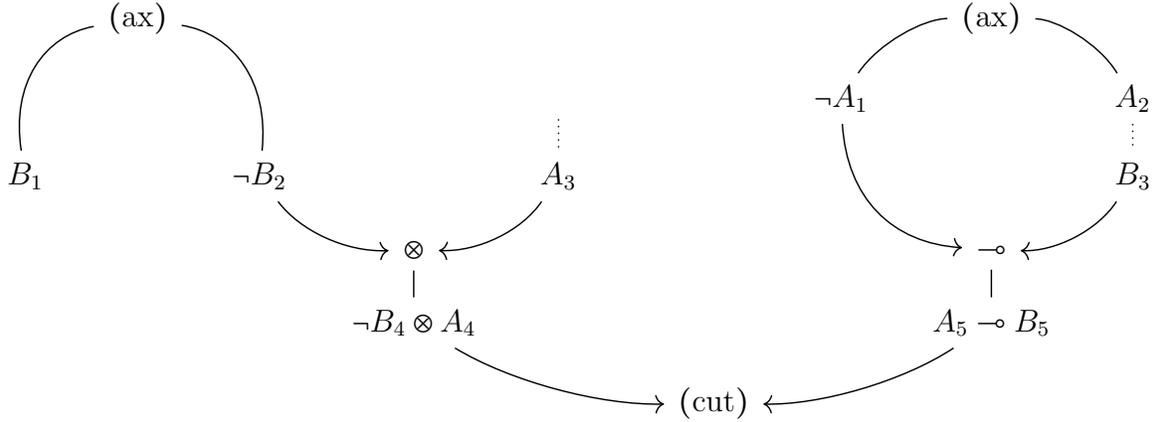
\begin{center}
$
\begin{tikzcd}[column sep = small,row sep = small]
		& \ax &&&&&&& \ax \\
		&&&& \, &&& \neg A_1 && A_2 \\
		B_1 && \neg B_2 && A_3 && & && B_3 \\
		&&& \otimes &&& && \multimap \\
		&&& \neg B_4 \otimes A_4 &&&&& A_5 \multimap B_5 \\
		&&&&&& \cut
		\arrow[curve={height=25pt}, no head, from=1-2, to=3-1]
		\arrow[curve={height=-25pt}, no head, from=1-2, to=3-3]
		\arrow[dotted, no head, from=2-10, to=3-10]
		\arrow[curve={height=-12pt}, from=3-5, to=4-4]
		\arrow[curve={height=12pt}, from=3-3, to=4-4]
		\arrow[curve={height=12pt}, no head, from=1-9, to=2-8]
		\arrow[curve={height=-12pt}, no head, from=1-9, to=2-10]
		\arrow[curve={height=-12pt}, from=3-10, to=4-9]
		\arrow[curve={height=25pt}, from=2-8, to=4-9]
		\arrow[no head, from=4-9, to=5-9]
		\arrow[dotted, no head, from=2-5, to=3-5]
		\arrow[curve={height=-12pt}, from=5-9, to=6-7]
		\arrow[curve={height=12pt}, from=5-4, to=6-7]
		\arrow[no head, from=4-4, to=5-4]
	\end{tikzcd}
$
\end{center}
Suppose the atomic propositions occurring in $A$ are $X_1,\ldots,X_n$ and label the occurrence of $X_j$ in $A_i$ by $X^{(i)}_j$. Similarly let $Y_1,\ldots,Y_m$ be the atomic propositions in $B$ and write $Y^{(i)}_j$ for the occurrence of $Y_j$ in $B_i$. We let $\bold{X}^{(i)}$ (resp. $\bold{Y}^{(i)}$) stand for all the $X$-variables (resp. $Y$-variables) with superscript $i$. Then associated to the path described above through the proof net are the identifications
\begin{equation}
\bold{X}^{(3)} = \bold{X}^{(4)} = \bold{X}^{(5)} = \bold{X}^{(1)} = \bold{X}^{(2)}
\end{equation}
leading to the ``input'' of the proof $B$ of $A$, and the identifications
\begin{equation}
\bold{Y}^{(3)} = \bold{Y}^{(5)} = \bold{Y}^{(4)} = \bold{Y}^{(2)} = \bold{Y}^{(1)}
\end{equation}
leading from the ``output'' back through the proof net to the conclusion $B$. Here $\bold{X}^{(i)} = \bold{X}^{(i')}$ stands for a sequence of equations $X^{(i)}_j = X^{(i')}_j$ for $1 \le j \le n$. We have grouped these equations by how they appear on the path from $A$ to $B$ through the proof net, but we can also group them by link
\begin{align}
\ax_B &\quad \bold{Y}^{(1)} = \bold{Y}^{(2)} \label{eq:equationset_1}\\
\ax_A &\quad \bold{X}^{(1)} = \bold{X}^{(2)} \label{eq:equationset_2}\\
(\otimes) &\quad \bold{X}^{(3)} = \bold{X}^{(4)}, \bold{Y}^{(2)} = \bold{Y}^{(4)} \label{eq:equationset_3}\\
(\multimap) &\quad \bold{Y}^{(3)} = \bold{Y}^{(5)}, \bold{X}^{(1)} = \bold{X}^{(5)} \label{eq:equationset_4}\\
\cut &\quad \bold{X}^{(4)} = \bold{X}^{(5)}, \bold{Y}^{(4)} = \bold{Y}^{(5)}\,.\label{eq:equationset_5}
\end{align}
For some base field $k$ we consider the polynomial ring $P$ generated by all the atomic propositions
\begin{equation}
P = k\big[ \bold{X}^{(1)}, \ldots, \bold{X}^{(5)}, \bold{Y}^{(1)}, \ldots, \bold{Y}^{(5)}]
\end{equation}
and to the sets of equations \eqref{eq:equationset_1}-\eqref{eq:equationset_5} we associate the ideals (one for each link)
\begin{align}
I_{\ax_B} &= \langle\bold{Y}^{(1)} - \bold{Y}^{(2)}\rangle = \langle Y^{(1)}_1 - Y^{(2)}_1, \ldots, Y^{(1)}_m - Y^{(2)}_m\rangle \label{eq:ideal_1}\\
I_{\ax_A} &= \langle \bold{X}^{(1)} - \bold{X}^{(2)} \rangle \label{eq:ideal_2}\\
I_{\otimes} &= \langle \bold{X}^{(3)} - \bold{X}^{(4)}, \bold{Y}^{(2)} - \bold{Y}^{(4)} \rangle \label{eq:ideal_3}\\
I_{\multimap} &= \langle \bold{Y}^{(3)} - \bold{Y}^{(5)}, \bold{X}^{(1)} - \bold{X}^{(5)} \rangle \label{eq:ideal_4}\\
I_{\cut} &= \langle \bold{X}^{(4)} - \bold{X}^{(5)}, \bold{Y}^{(4)} - \bold{Y}^{(5)} \rangle\,. \label{eq:ideal_5}
\end{align}
Assume that somehow ideals
\begin{equation}
\mathfrak{n} \subseteq k[\bold{X}^{(3)}]\,, \qquad \mathfrak{m} \subseteq k[\bold{X}^{(2)}, \bold{Y}^{(3)}]
\end{equation}
have been assigned to the proofs of $A$, and of $B$ from $A$, in ellipsis. The \emph{ideal of the proof net} \eqref{eq:intro_deduction_d} is (see Section \ref{section:ideal_of_proof} for the general definition)
\begin{equation}
I = I_{\ax_B} + I_{\ax_A} + I_{\otimes} + I_{\multimap} + I_{\cut} + \mathfrak{m} + \mathfrak{n} \subseteq P\,.
\end{equation}
This ideal represents the pattern of equality of the proof net in terms of the variables occurring on every edge. When it comes to the interaction of this proof net with other nets under cut-elimination, this pattern in the ``interior'' only matters insofar as it implicitly determines a pattern of equality between the variables in $\bold{X}^{(3)}$ and $\bold{Y}^{(1)}$, which live on the ``boundary'' of the proof net. Computation is the process of making this \emph{implicit} pattern in the interior \emph{explicit} on the boundary (Section \ref{section:main_theorem}).

The pattern of equalities on the boundary is the set of equations between the variables $\bold{X}^{(3)},\bold{Y}^{(1)}$ which are implied by the equations between the variables in the interior. These equations form an ideal $J \subseteq k[\bold{X}^{(3)},\bold{Y}^{(1)}]$ and to say that the equations in $J$ follow from those in $I$ is to say that under the inclusion
\[
\inc: k[\bold{X}^{(3)},\bold{Y}^{(1)}] \lto P
\]
we have $\inc(J) \subseteq I$. To say that \emph{every} equation between the variables $\bold{X}^{(3)},\bold{Y}^{(1)}$ implied by the equations in $I$ is in $J$ is to say further that
\begin{equation}
J = I \cap k[\bold{X}^{(3)},\bold{Y}^{(1)}]\,.
\end{equation}
So what is a generating set for this ideal? With a little computation using the above equations it is easy to see that $J = \mathfrak{n} + \mathfrak{m}'$ is the sum of $\mathfrak{n}$ with the ideal $\mathfrak{m}'$ obtained from $\mathfrak{m}$ by replacing every generator $f(\bold{X}^{(2)}, \bold{Y}^{(3)})$ of this ideal with the polynomial $f(\bold{X}^{(3)}, \bold{Y}^{(1)})$ obtained by substituting $X^{(3)}_i$ for $X^{(2)}_i$ in $f$ for all $1 \le i \le n$ and $Y^{(1)}_j$ for $Y^{(3)}_j$ for all $1 \le j \le m$. This parallels the substitution $M[ x := N ]$.

In this paper we treat such substitution processes systematically from the point of view of Elimination Theory, as instances of the Buchberger algorithm for computing a Gröbner basis of $J$ from $I$.



\section{The Ideal of a Proof}\label{section:ideal_of_proof}

To any proof structure $\pi$ we will associate a polynomial ring $P_\pi$. This is the universal algebraic object in which we can add and multiply the atomic propositions in the formulas adorning the edges of the proof structure (noting that each occurrence gets its own copy in the polynomial ring) and multiply them by scalars from some fixed base field $k$. The idea of making propositions into objects of arithmetic or algebra is due to Boole \cite{Boole}. The ring $P_\pi$ only reflects the set of occurrences of atomic propositions in $\pi$ not the structure of the proof; this structure is encoded in an ideal $I_\pi \subseteq P_\pi$ (Definition \ref{defn:ideal_proof}).

\subsection{Background}

We assume familiarity with linear logic, for an introduction see \cite{proofstypes, Laurent, Troiani, murfet_ll}.

\begin{defn}\label{def:oriented_atoms} There is an infinite set of \emph{unoriented atoms} $X,Y,Z,...$ and an \emph{oriented atom} (or \emph{atomic proposition}) is a pair $(X,+)$ or $(X,-)$ where $X$ is an unoriented atom. Let $\call{A}$ denote the set of oriented atoms.
\end{defn}

For $x \in \{ +, - \}$ we write $\bar{x}$ for the negation, so $\bar{+} = -, \bar{-} = +$. 

	\begin{defn}\label{def:formulas} The set of \emph{pre-formulas} is defined as follows:
		\begin{itemize}
			\item Any atomic proposition is a preformula.
			\item If $A,B$ are pre-formulas then so are $A \otimes B$, $A \parr B$.
			\item If $A$ is a pre-formula then so is $\neg A$.
		\end{itemize}
		The set of \emph{formulas} is the quotient of the set of pre-formulas by the equivalence relation $=$ generated, for arbitrary formulas $A,B$ and unoriented atom $X$, by
		\[
			\neg (A \otimes B) = \neg B \parr \neg A,\qquad \neg (A \parr B) = \neg B \otimes \neg A,\qquad \neg (X, x) = (X, \bar{x})
		\]
		 Let $\call{F}$ denote the set of formulas.
	\end{defn}

	The following easy lemma is left to the reader:
	
	\begin{lemma}
		For all formulas $A$ we have $\neg \neg A = A$.
	\end{lemma}
		
		
	\begin{defn}
		Let $\call{A}^\ast = \bigcup_{n \geq 0}\call{A}^n$ be the set of sequences of oriented atoms of length $\geq 0$. We define an involution $r$ on $\call{A}^\ast$ as follows:
		\begin{align*}
			r: \call{A}^\ast &\lto \call{A}^\ast\\
			\big((X_1,x_1),...,(X_n,x_n)\big) &\longmapsto \big((X_n,\bar{x}_n),...,(X_1,\bar{x}_1)\big)
		\end{align*}
For the empty string $\emptyset \in \call{A}^\ast$ we define $r(\emptyset) = \emptyset$.
	\end{defn}	
	
	The set $\call{A}^\ast$ is a monoid under concatenation $c: \call{A}^\ast \times \call{A}^\ast \lto \call{A}^\ast$ with identity $\emptyset$.
	
	\begin{defn}
		We denote by $\otimes: \call{F} \times \call{F} \lto \call{F}$ the function which maps a pair of formulas $(A,B)$ to the formula $A \otimes B$. Similarly, $\parr: \call{F} \times \call{F} \lto \call{F}$ denotes the function such that $\parr(A,B) = A \parr B$ and $\neg : \call{F} \lto \call{F}$ denotes the function such that $\neg(A) = \neg A$. We denote by $\operatorname{inc}: \call{A} \lto \call{F}$ the map which sends an oriented atom $(X,x)$ to itself $(X,x)$, and lastly we denote by $\operatorname{\iota}: \call{A} \lto \call{A}^\ast$ the function which maps an oriented atom $(X,x)$ to the sequence consisting only of $(X,x)$.
	\end{defn}
	
	\begin{lemma}
		There is a unique map $a: \call{F} \lto \call{A}^\ast$ making the following diagrams commute
		\begin{equation}
			\begin{tikzcd}
			\call{F} \times \call{F}\arrow[r,"{a \times a}"] \arrow[d,swap,"{\otimes}"]& \call{A}^\ast \times \call{A}^\ast\arrow[d,"{c}"]\\
			\call{F}\arrow[r,"{a}"] & \call{A}^\ast
			\end{tikzcd}
			\qquad
			\begin{tikzcd}
			\call{F} \times \call{F}\arrow[r,"{a \times a}"]\arrow[d,swap,"{\parr}"] & \call{A}^\ast \times \call{A}^\ast\arrow[d,"{c}"]\\
			\call{F}\arrow[r,"{a}"] & \call{A}^\ast
			\end{tikzcd}
		\end{equation}
		\begin{equation}
			\begin{tikzcd}
			\call{F}\arrow[r,"{a}"]\arrow[d,swap,"{\neg}"] & \call{A}^\ast\arrow[d,swap,"{r}"]\\
			\call{F}\arrow[r,"{a}"] & \call{A}^\ast
			\end{tikzcd}
			\qquad
			\begin{tikzcd}
			\call{A}\arrow[r,"{\operatorname{inc}}"]\arrow[dr,swap,"{\iota}"] & \call{F}\arrow[d,"{a}"]\\
			& \call{A}^\ast
			\end{tikzcd}
		\end{equation}
	\end{lemma}
	\begin{proof}
	Left to the reader.
	\end{proof}
	
	\begin{defn}\label{def:seq_set}
		Let $A$ be a formula. The \emph{sequence of oriented atoms} of $A$ is
		\[
		a(A) = (X_1,x_1),\ldots,(X_n,x_n)
		\]
		as defined by the previous lemma. The \emph{sequence of unoriented atoms} of $A$ is $X_1,...,X_n$ and the \emph{set of unoriented atoms} of $A$ is the disjoint union $U_A = \lbrace X_1 \rbrace \coprod \ldots \coprod \lbrace X_n \rbrace$.
	\end{defn}
	
	Our definition of \emph{multiplicative proof structures} follows that of \cite{Laurent} but with edges labelled by formulas in the sense of Definition \ref{def:formulas}. We remind the reader of the notion of a multigraph (a graph with multiple edges between vertices).
	
	\begin{defn}
		A \emph{directed multigraph} is a triple $(V,E,\varphi)$ where:
		\begin{itemize}
			\item $V$ is a set of \emph{vertices}, or \emph{nodes}.
			\item $E$ is a set of \emph{edges}.
			\item $\Gamma: E \lto \lbrace (x,y) \mid x, y \in V\rbrace$ is a function from the set of edges to the set of ordered pairs of vertices.
		\end{itemize}
		For $e \in E$, the first element of $\Gamma(e)$ is the \emph{source} and the second element is the \emph{target}.
	\end{defn}

	\begin{defn}
		A \emph{proof structure} is a directed multigraph with edges labelled by formulas and with vertices labelled by $\ax, \cut, \otimes, \parr$ or $\operatorname{c}$. The incoming edges of a vertex are called its \emph{premises}, the outgoing edges are its \emph{conclusions}. Proof structures are required to adhere to the following conditions:
		\begin{itemize}
			\item Each vertex labelled $\ax$ has exactly two conclusions and no premise, the conclusions are labelled $A$ and $\neg A$ for some $A$. We call this an \emph{axiom link}.
			\item Each vertex labelled $\cut$ has exactly two premises and no conclusion, where the premises are labelled $A$ and $\neg A$ for some $A$. We call this a \emph{cut link}.
			\item Each vertex labelled $\otimes$ has exactly two premises and one conclusion. The premises are ordered, the smallest one is called the \emph{left} premise of the vertex, the biggest one is called the \emph{right} premise. The left premise is labelled $A$, the right premise is labelled $B$ and the conclusion is labelled $A \otimes B$, for some $A,B$. We call this a \emph{tensor link}.
			\item Each vertex labelled $\parr$ has exactly two ordered premises and one conclusion. The left premise is labelled $A$, the right premise is labelled $B$ and the conclusion is labelled $A \parr B$, for some $A,B$. We call this a \emph{par link}.
			\item Each vertex labelled $\operatorname{c}$ has exactly one premise and no conclusion. Such a premise of a vertex labelled $\operatorname{c}$ is called a \emph{conclusion} of the proof structure.
		\end{itemize}
A \emph{link} is any of an axiom link, tensor link, par link, or cut link. Conclusions are not links. 
	\end{defn}
	
	
\begin{defn} A \emph{proof net} is a proof structure which is the image under the translation map of a sequent calculus proof in multiplicative linear logic, see \cite[Definition 1.0.12]{Troiani}.
\end{defn}

\begin{example}\label{example:smallest_pf}
	The smallest proof net with a single conclusion is
	
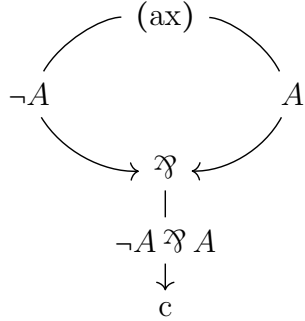
\[\begin{tikzcd}[column sep = small, row sep = small]
		& \ax \\
		{\neg A} && A \\
		& \parr \\
		& {\neg A \parr A} \\
		& {\operatorname{c}}
		\arrow[curve={height=-12pt}, no head, from=1-2, to=2-3]
		\arrow[curve={height=12pt}, no head, from=1-2, to=2-1]
		\arrow[curve={height=-12pt}, from=2-3, to=3-2]
		\arrow[curve={height=12pt}, from=2-1, to=3-2]
		\arrow[from=4-2, to=5-2]
		\arrow[no head, from=3-2, to=4-2]
	\end{tikzcd}\]
This is the translation of the following sequent calculus proof.
\begin{prooftree}
	\AxiomC{}
	\RightLabel{$\ax$}
	\UnaryInfC{$\vdash \neg A, A$}
	\RightLabel{$(\parr)$}
	\UnaryInfC{$\vdash \neg A \parr A$}
	\end{prooftree}
\end{example}

\begin{example}
	As in Section \ref{section:intro_example} the following proof net is the translation of the linear lambda term $(\lambda x. x)x$.
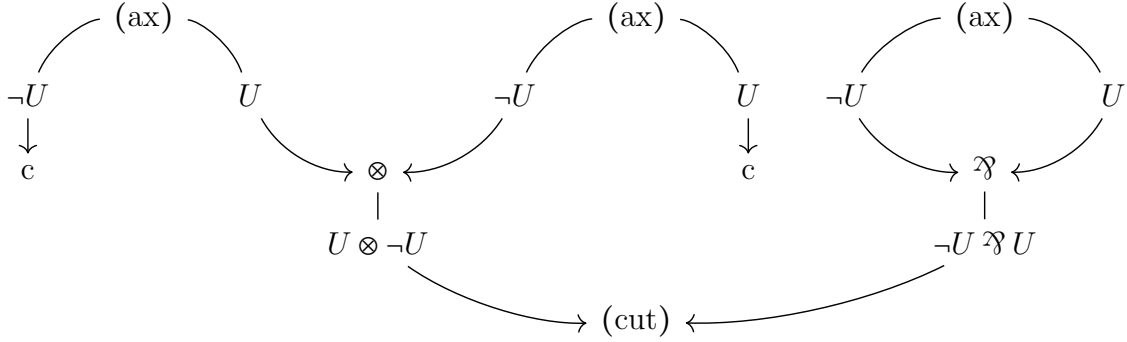
	\[\begin{tikzcd}[column sep = small, row sep = small]
		& \ax &&&& \ax &&& \ax \\
		{\neg U} && U && {\neg U} && U & {\neg U} && U \\
		{\operatorname{c}} &&& \otimes &&& {\operatorname{c}} && \parr \\
		&&& {U \otimes \neg U} &&&&& {\neg U \parr U} \\
		&&&&& \cut
		\arrow[curve={height=12pt}, from=4-4, to=5-6]
		\arrow[curve={height=-12pt}, from=4-9, to=5-6]
		\arrow[no head, from=3-9, to=4-9]
		\arrow[curve={height=-12pt}, from=2-10, to=3-9]
		\arrow[curve={height=12pt}, from=2-8, to=3-9]
		\arrow[curve={height=-12pt}, no head, from=1-9, to=2-10]
		\arrow[curve={height=12pt}, no head, from=1-9, to=2-8]
		\arrow[curve={height=-12pt}, no head, from=1-2, to=2-3]
		\arrow[curve={height=12pt}, no head, from=1-2, to=2-1]
		\arrow[from=2-1, to=3-1]
		\arrow[curve={height=12pt}, from=2-3, to=3-4]
		\arrow[curve={height=-12pt}, from=2-5, to=3-4]
		\arrow[no head, from=3-4, to=4-4]
		\arrow[curve={height=12pt}, no head, from=1-6, to=2-5]
		\arrow[curve={height=-12pt}, no head, from=1-6, to=2-7]
		\arrow[from=2-7, to=3-7]
	\end{tikzcd}\]
This is the translation of the following sequent style proof.
\begin{prooftree}
	\AxiomC{}
	\RightLabel{$\ax$}
	\UnaryInfC{$\vdash \neg U, U$}
	\AxiomC{}
	\RightLabel{$\ax$}
	\UnaryInfC{$\vdash \neg U, U$}
	\RightLabel{$(\otimes)$}
	\BinaryInfC{$\vdash \neg U, U \otimes \neg U, U$}
	\AxiomC{}
	\RightLabel{$\ax$}
	\UnaryInfC{$\vdash \neg U, U$}
	\RightLabel{$\parr$}
	\UnaryInfC{$\vdash \neg U \parr U$}
	\RightLabel{$\cut$}
	\BinaryInfC{$\vdash \neg U, U$}
	\end{prooftree}
\end{example}

\begin{example}\label{example:not_pf_net}
	The following is a proof structure which is \emph{not} a proof net.
	
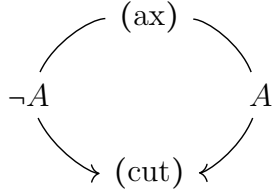
\[\begin{tikzcd}[column sep = small, row sep = small]
		& \ax \\
		{\neg A} && A \\
		& \cut
		\arrow[curve={height=12pt}, from=2-1, to=3-2]
		\arrow[curve={height=-12pt}, from=2-3, to=3-2]
		\arrow[curve={height=12pt}, no head, from=1-2, to=2-1]
		\arrow[curve={height=-12pt}, no head, from=1-2, to=2-3]
	\end{tikzcd}\]
Suppose the above proof structure was equal to $T(\pi)$ where $\pi$ is a sequent calculus proof. Then there exist sequent calculus proofs $\pi_1, \pi_2$ such that $\pi$ is given by a cut rule with $\pi_1$ on the left branch and $\pi_2$ on the right branch. Since sequent calculus proofs are binary trees with leaves labelled by instances of the Axiom rule, it follows that $\pi$ admits at least two instances of an Axiom rule, a contradiction.
\end{example}

\subsection{The polynomial ring}

	Let $k$ be a field.
	
	\begin{defn}
		The \emph{polynomial ring of a formula $A$} is the free commutative $k$-algebra $P_A$ on the set $U_A$ of unoriented atoms of $A$. 
	\end{defn}
	
	\begin{remark}\label{rmk:s'}
	Some technical remarks:
	\begin{itemize}
	\item The algebra $P_A$ does \emph{not} depend on an ordering of the unoriented atoms in $A$. For example, we can take as our concrete realisation of $P_A$ the symmetric algebra $\operatorname{Sym}(V_A)$ on the free vector space $V_A$ on the set $U_A$. 
	\item If we choose an ordering on $U_A$ so that $U_A = \{X_1,\ldots,X_n\}$ then $P_A \cong k[X_1,...,X_n]$. It will be convenient to present $P_A$ in this way in the following, but unless we explicitly introduce an ordering on $U_A$ no meaning should be associated to the indices $i$ in $X_i$.
	\item The oriented atom $(X,+)$ appears twice in $A = (X,+) \otimes (X,+)$, so the sequence of unoriented atoms of $A$ is $X,X$ and the \emph{set} of unoriented atoms is a disjoint union $\{ X, X' \} = \{ X \} \coprod \{ X \}$ inside which we can distinguish the two copies of $X$ (the precise nature of these copies is not important here). Hence $P_A \cong k[X, X']$.
	\end{itemize}
	\end{remark}
		
	\begin{defn} The \emph{set of unoriented atoms} of a proof structure $\pi$ is the disjoint union $U_\pi = \coprod_{e \in E} U_{A_e}$ where $E$ is the set of edges in $\pi$ and for each $e \in E$ we write $U_{A_e}$ for the set of unoriented atoms of the formula $A_e$ labelling $e$.
	\end{defn}
	
	\begin{defn}\label{def:polynomial_ring}
The \emph{polynomial ring of a proof structure $\pi$} is the free commutative $k$-algebra on the set $U_\pi$ of unoriented atoms of $\pi$.
	\end{defn}
	
The polynomial ring of a proof structure $\pi$ is the polynomial ring with a generator for each pair consisting of an edge $e$ in $\pi$ and an unoriented atom of the formula labelling $e$. If we choose an ordering on the edges $e$ of $\pi$ then we have an isomorphism
		\begin{equation}
			P_\pi = k[\coprod_{e \in E}U_{A_e}] \cong \bigotimes_{e \in E}P_{A_e}\,.
		\end{equation}
		

	
	\subsection{The ideal of a link}
	
	The ideal generated by a set $S$ is denoted $\langle S \rangle$. For each type of link $l$ we define the \emph{generating set} $G_l$ and \emph{link ideal} $I_l = \langle G_l \rangle$ in the polynomial ring generated by the set of unoriented atoms of formulas labelling edges incident at the link. Below $(X_1,x_1),...,(X_n,x_n)$ is the sequence of oriented atoms of a formula $A$, and $(Y_1,y_1),...,(Y_m,y_m)$ is the sequence of oriented atoms of $B$. 
	
	\begin{defn}[Cut and Axiom]\label{ideal_axiomcut} If $l$ is one of the links
		 
		\begin{center}
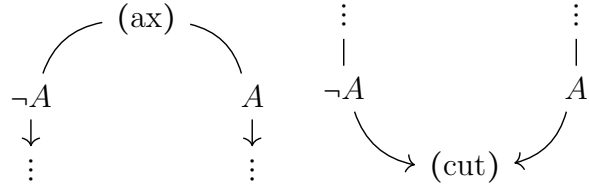
			\begin{tabular}{c c}
				$
				\begin{tikzcd}[column sep = small, row sep = small]
					& \ax\arrow[dl,bend right, dash]\arrow[dr,bend left, dash]\\
					\neg A\arrow[d] & & A\arrow[d]\\
					\vdots & & \vdots
				\end{tikzcd}$
				&
				$
					\begin{tikzcd}[column sep = small, row sep = small]
						\vdots\arrow[d, dash] & & \vdots\arrow[d, dash]\\
						\neg A\arrow[dr,bend right] & & A\arrow[dl, bend left]\\
						& \cut
					\end{tikzcd}$
			\end{tabular}
		\end{center}
then in the polynomial ring generated by the disjoint union of unoriented atoms $U_A \coprod U_{\neg A}$
\begin{equation}\label{eq:G_both_ways}
	G_l = \big\{ X_i - X_i'\big\}_{i = 1}^n 
	\end{equation}
where we write $X'_i$ for the elements of $U_{\neg A}$.
\end{defn}


\begin{defn}[Tensor and Par links]\label{ideal_axiompartensor} If $l$ is one of the links
		\begin{center}
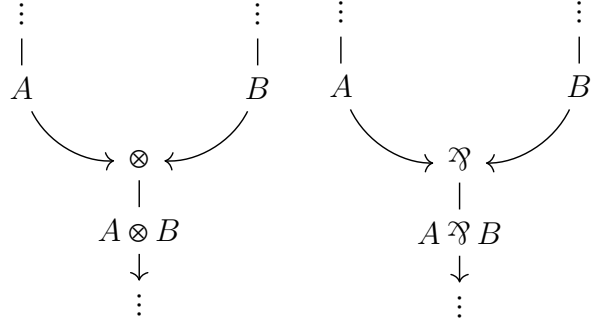
			\begin{tabular}{ c c }
				
					$\begin{tikzcd}[column sep = small, row sep = small]
						\vdots && \vdots \\
						A && B \\
						& \otimes \\
						& {A \otimes B} \\
						& \vdots
						\arrow[from=4-2, to=5-2]
						\arrow[curve={height=-12pt}, from=2-3, to=3-2]
						\arrow[curve={height=12pt}, from=2-1, to=3-2]
						\arrow[no head, from=3-2, to=4-2]
						\arrow[no head, from=1-3, to=2-3]
						\arrow[no head, from=1-1, to=2-1]
					\end{tikzcd}$
					&
					$\begin{tikzcd}[column sep = small, row sep = small]
						\vdots && \vdots \\
						A && B \\
						& \parr \\
						& {A \parr B} \\
						& \vdots
						\arrow[from=4-2, to=5-2]
						\arrow[curve={height=-12pt}, from=2-3, to=3-2]
						\arrow[curve={height=12pt}, from=2-1, to=3-2]
						\arrow[no head, from=3-2, to=4-2]
						\arrow[no head, from=1-3, to=2-3]
						\arrow[no head, from=1-1, to=2-1]
					\end{tikzcd}$
			\end{tabular}
		\end{center}
		then in polynomial ring generated by the disjoint union $U_A \coprod U_{B} \coprod U_{A \boxtimes B}$ (where $\boxtimes$ is either $\otimes$ or $\parr$)
		\begin{align*}
			G_l = \big\{ X_i - X_i'\big\}_{i = 1}^n \cup \big\{ Y_j - Y_j' \big\}_{j = 1}^m
		\end{align*}
		where $X'_i, Y'_j$ denote elements of $U_{A \boxtimes B}$.
		\end{defn}
	
	In Definition \ref{def:polynomial_ring} we defined the polynomial ring $P_\pi$ of a proof structure $\pi$. If $l$ is a link of $\pi$ then in the obvious way we can identify $G_l$ with a subset of $P_\pi$. 
	
\begin{defn} The \emph{ideal of a link $l$} in a proof structure $\pi$ is the ideal $I_l = \langle G_l \rangle \subseteq P_\pi$ generated by $G_l$ in $P_\pi$.
\end{defn}
		
	\begin{defn}\label{defn:ideal_proof} Let $\pi$ be a proof structure with set of links $\call{L}$. Then
	\begin{equation}
		G_\pi := \bigcup_{l \in \call{L}}G_l\,.
		\end{equation}
		The \emph{defining ideal} $I_\pi$ is the ideal in $P_\pi$ generated by $G_\pi$, or equivalently
		\begin{equation}
		I_\pi = \sum_{l \in \call{L}}I_l\,.
		\end{equation}
		The \emph{coordinate ring} of $\pi$ is the quotient $R_\pi = P_\pi/I_\pi$.
	\end{defn}
	




\begin{example} Consider the proof net $\pi$ of Example \ref{example:smallest_pf} in which we set $A$ to be the atomic proposition $A = (X,+)$
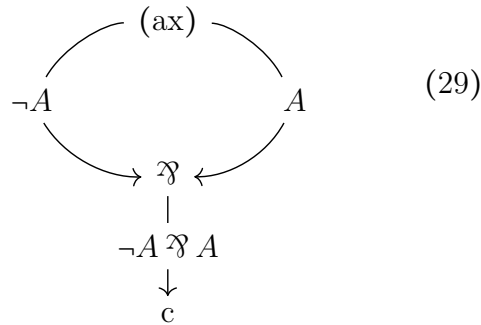
\begin{center}
\begin{tabular}{>{\centering}m{6cm} >{\centering}m{1cm}}
$
\begin{tikzcd}[column sep = small, row sep = small]
	& \ax\\
	\neg A && A \\
	& \parr \\
	& {\neg A \parr A} \\
	& \operatorname{c}
	\arrow[curve={height=-12pt}, from=2-3, to=3-2]
	\arrow[curve={height=12pt}, from=2-1, to=3-2]
	\arrow[curve={height=12pt}, no head, from=1-2, to=2-1]
	\arrow[curve={height=-12pt}, no head, from=1-2, to=2-3]
	\arrow[no head, from=3-2, to=4-2]
	\arrow[from=4-2, to=5-2]
\end{tikzcd}
$
&
\tagarray{\label{eq:proof_net_ident}}
\end{tabular}
\end{center}
The set of unoriented atoms of $\pi$ is $\{ X, X', X'', X'''\}$ where we associate $X$ to $A$, $X'$ to $\neg A$ and we let $X'', X'''$ be the sequence of unoriented atoms of $\neg A \parr A$. Then
\begin{align*}
G_{\ax} = \{ X - X' \}\, \qquad G_\parr = \{ X - X''', X' - X'' \}\,.
\end{align*}
Hence
\[
R_\pi = P_\pi / I_\pi = k[X, X', X'', X''']/(I_{\ax} + I_\parr) \cong k[X, X']/I_{\ax} \cong k[X]
\]
in which $X = X' = X'' = X'''$.
\end{example}

\begin{example}\label{ex:proof_net_not} Consider the proof structure $\rho$ of Example \ref{example:not_pf_net} (which is not a proof net) with $A = (X,+)$. Then $G_{\ax} = \{ X - X' \} = G_{\cut}$ so $R_\rho = k[X, X']/(X - X') \cong k[X]$.
\end{example}

\section{Reduction}

We have now associated to each proof structure $\pi$ a pair $(I_\pi, P_\pi)$ consisting of a polynomial ring $P_\pi$ generated by the unoriented atoms of $\pi$ and an ideal $I_\pi$ in this ring generated by differences $U - V$ where $U,V$ are ``the same'' variable on two different edges incident at a link in $\pi$. We now turn to the cut-elimination process for multiplicative proof nets and study how to think about a reduction $\pi \lto \pi'$ in terms of $(I_\pi, P_\pi), (I_{\pi'}, P_{\pi'})$.

	\begin{defn} An \emph{$a$-redex} in a proof structure $\pi$ is a subgraph of the form
		\begin{center}
			\begin{tabular}{ >{\centering}m{12cm} >{\centering}m{0.5cm}}
				$
				\begin{tikzcd}[column sep = small, row sep = small]
					\vdots\arrow[d,dash] &&& \ax\arrow[dl, bend right, dash]\arrow[dr,dash, bend left]\\
					A\arrow[dr, bend right] && \neg A\arrow[dl,bend left] && A\arrow[d]\\
					& \cut &&& \vdots
				\end{tikzcd}
				$
				&
				\tagarray{\label{eq:a_redex}}
			\end{tabular}
		\end{center}
	\end{defn}
	
	\begin{defn} An \emph{$m$-redex} in a proof structure $\pi$ is a subgraph of the form
	\begin{center}
		\begin{tabular}{ >{\centering}m{12cm} >{\centering}m{0.5cm}}
			$
			\begin{tikzcd}[column sep = small, row sep = small]
				\vdots && \vdots && \vdots && \vdots \\
				A && B && {\neg B} && {\neg A} \\
				& \otimes &&&& \parr \\
				& {A \otimes B} &&&& {\neg B \parr \neg A} \\
				&&& \cut
				\arrow[no head, from=1-7, to=2-7]
				\arrow[curve={height=-12pt}, from=2-7, to=3-6]
				\arrow[curve={height=12pt}, from=2-5, to=3-6]
				\arrow[no head, from=1-5, to=2-5]
				\arrow[no head, from=3-6, to=4-6]
				\arrow[curve={height=-12pt}, from=4-6, to=5-4]
				\arrow[curve={height=12pt}, from=4-2, to=5-4]
				\arrow[no head, from=3-2, to=4-2]
				\arrow[curve={height=-12pt}, from=2-3, to=3-2]
				\arrow[curve={height=12pt}, from=2-1, to=3-2]
				\arrow[no head, from=1-3, to=2-3]
				\arrow[no head, from=1-1, to=2-1]
			\end{tikzcd}
		$
		&
		\tagarray{\label{eq:m_redex}}
		\end{tabular}
	\end{center}
	\end{defn}


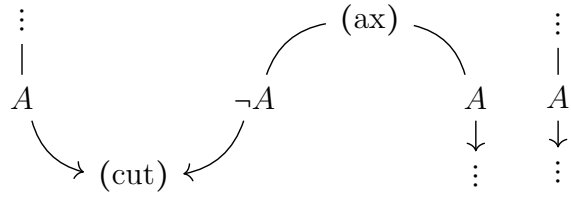
\begin{figure}
		\begin{center}
			\begin{tabular}{ c c }
			$
			\begin{tikzcd}[column sep = small, row sep = small]
			\vdots\arrow[d,dash] &&& \ax\arrow[dl, bend right, dash]\arrow[dr,dash, bend left]\\
			A\arrow[dr, bend right] && \neg A\arrow[dl,bend left] && A\arrow[d]\\
			& \cut &&& \vdots
			\end{tikzcd}
			$
			&
			$
			\begin{tikzcd}[column sep = small, row sep = small]
				\vdots\arrow[d,dash]\\
				A\arrow[d]\\
				\vdots
			\end{tikzcd}
			$
			\end{tabular}
		\end{center}
	\caption{Reduction of an $a$-redex.}
	\label{figure:a_redex_reduction}
\end{figure}

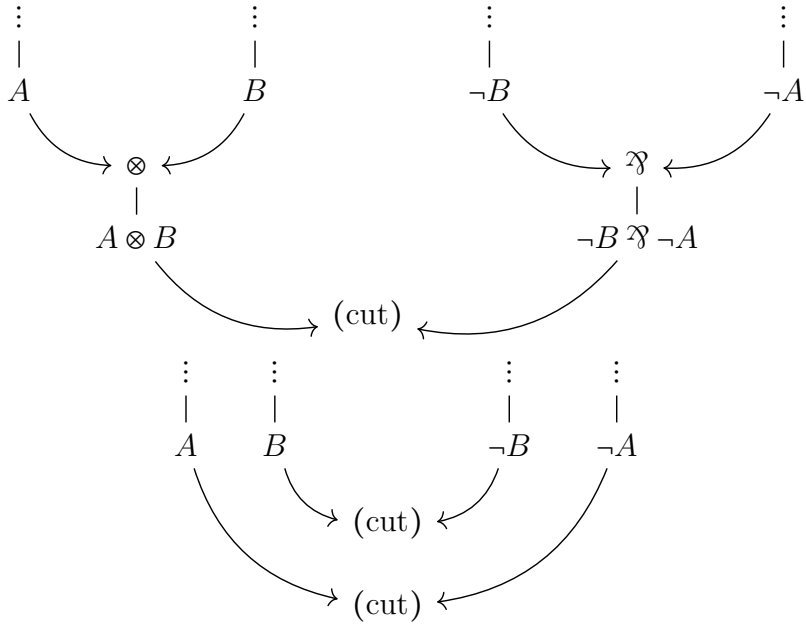
\begin{figure}
		\begin{center}
		\begin{tabular}{ >{\centering}m{12cm} }
		$\begin{tikzcd}[column sep = small, row sep = small]
			\vdots && \vdots && \vdots && \vdots \\
			A && B && {\neg B} && {\neg A} \\
			& \otimes &&&& \parr \\
			& {A \otimes B} &&&& {\neg B \parr \neg A} \\
			&&& \cut
			\arrow[from=2-5, to=3-6, bend right]
			\arrow[from=2-7, to=3-6, bend left]
			\arrow[from=3-6, to=4-6, dash]
			\arrow[from=3-2, to=4-2, dash]
			\arrow[from=2-1, to=3-2, bend right]
			\arrow[from=2-3, to=3-2, bend left]
			\arrow[from=1-1, to=2-1, dash]
			\arrow[from=1-3, to=2-3, dash]
			\arrow[from=1-5, to=2-5, dash]
			\arrow[from=1-7, to=2-7, dash]
			\arrow[from=4-2, to=5-4, bend right]
			\arrow[from=4-6, to=5-4, bend left]
		\end{tikzcd}$
	\\
	\begin{tikzcd}[column sep = small, row sep = small]
		\vdots & \vdots && \vdots & \vdots \\
		A & B && {\neg B} & {\neg A} \\
		&& \cut \\
		&& \cut
		\arrow[from=2-2, to=3-3, bend right]
		\arrow[from=2-4, to=3-3, bend left]
		\arrow[from=2-5, to=4-3, bend left]
		\arrow[from=2-1, to=4-3, bend right]
		\arrow[from=1-1, to=2-1, no head]
		\arrow[from=1-2, to=2-2, dash]
		\arrow[from=1-4, to=2-4, dash]
		\arrow[from=1-5, to=2-5, dash]
	\end{tikzcd}
	\end{tabular}
		\end{center}
	\caption{Reduction of a $m$-redex.}
	\label{figure:m_redex_reduction}
\end{figure}

	\begin{defn}\label{def:reduction} A \emph{reduction} $\gamma: \pi \lto \pi'$ between proof structures $\pi, \pi'$ is an $a$-redex or $m$-redex $\gamma$ in $\pi$ such that reducing $\gamma$ in $\pi$ yields $\pi'$, where
	\begin{itemize}
	\item if $\gamma$ is an $a$-redex \eqref{eq:a_redex}, the reduction is the proof structure obtained by replacing this subgraph by what is displayed on the right in Figure \ref{figure:a_redex_reduction}.
	\item if $\gamma$ is an $m$-redex \eqref{eq:m_redex}, the reduction is the proof structure obtained by replacing this subgraph by what is displayed on the bottom in Figure \ref{figure:m_redex_reduction}.
	\end{itemize} 
	\end{defn}

	
	\begin{defn}\label{def:morphisms} Let $\gamma: \pi \lto \pi'$ be a reduction of proof structures. We define morphisms $S_\gamma, T_\gamma$ of $k$-algebras
		\begin{equation}
			\begin{tikzcd}
				P_{\pi'}\arrow[rr,bend left, "{T_\gamma}"] & & P_{\pi}\arrow[ll, bend left, "{S_\gamma}"]
			\end{tikzcd}
		\end{equation}
		as follows: an unoriented atom in $\pi'$ (resp. $\pi$) not associated to an edge in the subgraph $\gamma$ is sent to itself in $\pi$ (resp. $\pi'$) by $T_\gamma$ (resp. $S_\gamma$), other unoriented atoms are mapped according to the schematic in Figure \ref{figure:tgamma} (resp. Figure \ref{figure:sgamma}).
\end{defn}

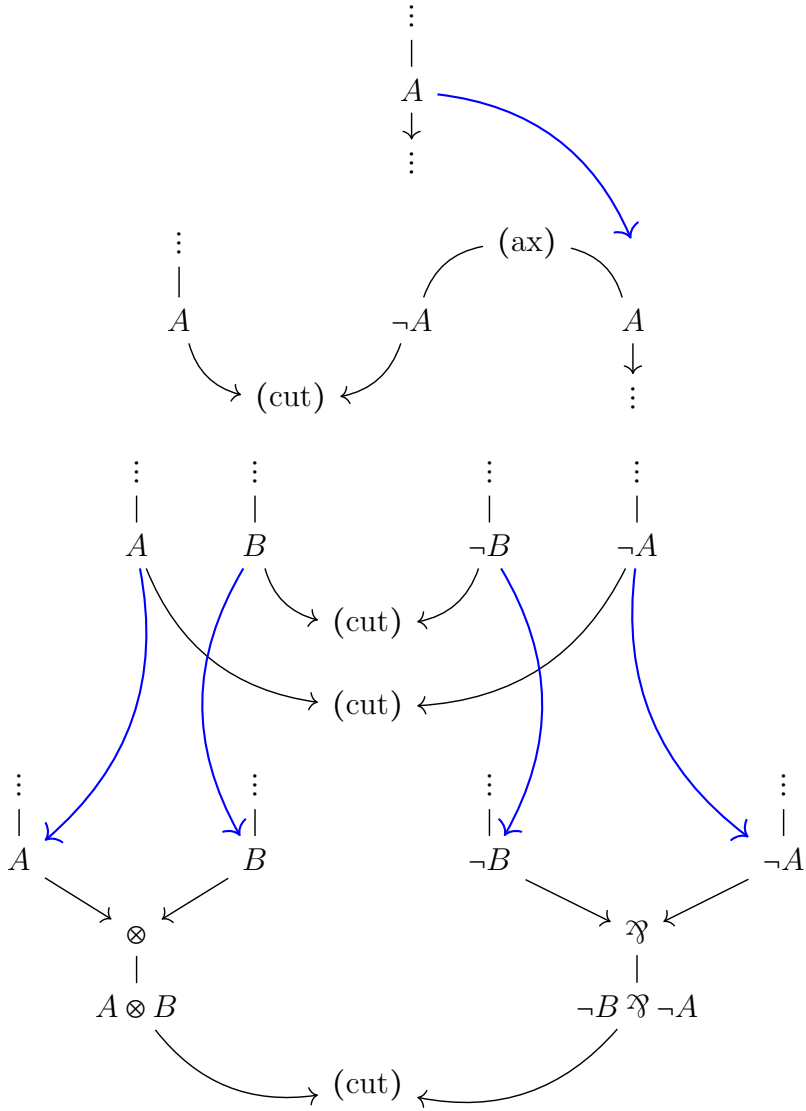
\begin{figure}
\[
		%
		\begin{tikzcd}[column sep = small, row sep = small]
			&& \vdots \\
			&& A \\
			&& \vdots \\
			\vdots\arrow[d,dash]&&& \ax & {} \\
			A && {\neg A} && A\arrow[d] \\
			& \cut &&& \vdots
			\arrow[from=4-4, to=5-5, bend left, dash]
			\arrow[from=4-4, to=5-3, bend right, dash]
			\arrow[from=5-1, to=6-2, bend right]
			\arrow[from=5-3, to=6-2, bend left]
			\arrow[from=2-3, to=4-5, thick, blue, bend left]
			\arrow[from=1-3, to=2-3, dash]
			\arrow[from=2-3, to=3-3]
		\end{tikzcd}
\]

\[
	\begin{tikzcd}[column sep = small, row sep = small]
		& \vdots & \vdots && \vdots & \vdots \\
		& A & B && {\neg B} & {\neg A} \\
		&&& \cut \\
		&&& \cut \\
		\vdots && \vdots && \vdots && \vdots \\
		A && B && {\neg B} && {\neg A} \\
		& \otimes\arrow[d,dash] &&&& \parr\arrow[d,dash] \\
		& A \otimes B &&&& \neg B \parr \neg A\\
		&&& \cut
		\arrow[from=2-3, to=3-4, bend right]
		\arrow[from=2-5, to=3-4, bend left]
		\arrow[from=2-6, to=4-4, bend left]
		\arrow[from=2-2, to=4-4, bend right]
		\arrow[from=1-2, to=2-2, no head]
		\arrow[from=1-3, to=2-3, dash]
		\arrow[from=1-5, to=2-5, dash]
		\arrow[from=1-6, to=2-6, dash]
		\arrow[from=5-3, to=6-3, dash]
		\arrow[from=5-1, to=6-1, dash]
		\arrow[from=5-5, to=6-5, dash]
		\arrow[from=5-7, to=6-7, dash]
		\arrow[from=8-2, to=9-4, bend right]
		\arrow[from=8-6, to=9-4, bend left]
		\arrow[from=6-5, to=7-6]
		\arrow[from=6-3, to=7-2]
		\arrow[from=6-1, to=7-2]
		\arrow[from=6-7, to=7-6]
		\arrow[from=2-2, to=6-1, thick, blue, bend left]
		\arrow[from=2-3, to=6-3, thick, blue, bend right]
		\arrow[from=2-5, to=6-5, thick, blue, bend left]
		\arrow[from=2-6, to=6-7, thick, blue, bend right]
	\end{tikzcd}
\]
\caption{Schematic for $T_\gamma$.}
\label{figure:tgamma}
\end{figure}

\begin{figure}
	\[
	\begin{tikzcd}[column sep = small, row sep = small]
		&& \vdots \\
		&& A \\
		&& \vdots \\
		\vdots &&& \ax \\
		A && {\neg A} && A \\
		& \cut &&& \vdots
		\arrow[curve={height=-12pt}, no head, from=4-4, to=5-5]
		\arrow[curve={height=12pt}, no head, from=4-4, to=5-3]
		\arrow[curve={height=12pt}, from=5-1, to=6-2]
		\arrow[curve={height=-12pt}, from=5-3, to=6-2]
		\arrow[no head, from=1-3, to=2-3]
		\arrow[from=2-3, to=3-3]
		\arrow[draw={rgb,255:red,153;green,92;blue,214}, curve={height=18pt}, from=5-5, to=2-3, thick]
		\arrow[draw={rgb,255:red,153;green,92;blue,214}, curve={height=-18pt}, from=5-3, to=2-3, thick]
		\arrow[draw={rgb,255:red,153;green,92;blue,214}, curve={height=-18pt}, from=5-1, to=2-3, thick]
		\arrow[from=5-5, to=6-5]
		\arrow[no head, from=4-1, to=5-1]
	\end{tikzcd}
	\]
	\[
	\begin{tikzcd}[column sep = small, row sep = small]
		\vdots &&& \vdots &&&& \vdots &&& \vdots \\
		A &&& B &&&& {\neg B} &&& {\neg A} \\
		&&&&& \cut \\
		&&&&& \cut \\
		&& \vdots && \vdots && \vdots && \vdots \\
		&& A && B && {\neg B} && {\neg A} \\
		&&& \otimes &&&& \parr \\
		\\
		&&& {A \otimes B} &&&& {\neg B \parr \neg A} \\
		&&&&& \cut
		\arrow[curve={height=12pt}, from=2-4, to=3-6]
		\arrow[curve={height=-12pt}, from=2-8, to=3-6]
		\arrow[curve={height=-12pt}, from=2-11, to=4-6]
		\arrow[curve={height=12pt}, from=2-1, to=4-6]
		\arrow[no head, from=1-1, to=2-1]
		\arrow[no head, from=1-4, to=2-4]
		\arrow[no head, from=1-8, to=2-8]
		\arrow[no head, from=1-11, to=2-11]
		\arrow[no head, from=5-5, to=6-5]
		\arrow[no head, from=5-3, to=6-3]
		\arrow[no head, from=5-7, to=6-7]
		\arrow[no head, from=5-9, to=6-9]
		\arrow[curve={height=-12pt}, from=6-5, to=7-4]
		\arrow[curve={height=12pt}, from=6-3, to=7-4]
		\arrow[curve={height=-12pt}, from=6-9, to=7-8]
		\arrow[draw={rgb,255:red,153;green,92;blue,214}, curve={height=6pt}, from=6-9, to=2-11]
		\arrow[no head, from=7-4, to=9-4]
		\arrow[curve={height=12pt}, from=9-4, to=10-6]
		\arrow[no head, from=7-8, to=9-8]
		\arrow[curve={height=-12pt}, from=9-8, to=10-6]
		\arrow[draw={rgb,255:red,153;green,92;blue,214}, curve={height=30pt}, from=9-8, to=2-11, thick]
		\arrow[draw={rgb,255:red,153;green,92;blue,214}, curve={height=-30pt}, from=9-4, to=2-1, thick]
		\arrow[draw={rgb,255:red,153;green,92;blue,214}, curve={height=-6pt}, from=6-3, to=2-1, thick]
		\arrow[draw={rgb,255:red,153;green,92;blue,214}, curve={height=6pt}, from=6-5, to=2-4, thick]
		\arrow[draw={rgb,255:red,153;green,92;blue,214}, curve={height=-6pt}, from=6-7, to=2-8, thick]
		\arrow[curve={height=12pt}, from=6-7, to=7-8]
		\arrow[draw={rgb,255:red,153;green,92;blue,214}, curve={height=-30pt}, from=9-8, to=2-8, thick]
		\arrow[draw={rgb,255:red,153;green,92;blue,214}, curve={height=30pt}, from=9-4, to=2-4, thick]
	\end{tikzcd}
	\]
\caption{Schematic for $S_\gamma$.}
\label{figure:sgamma}
\end{figure}
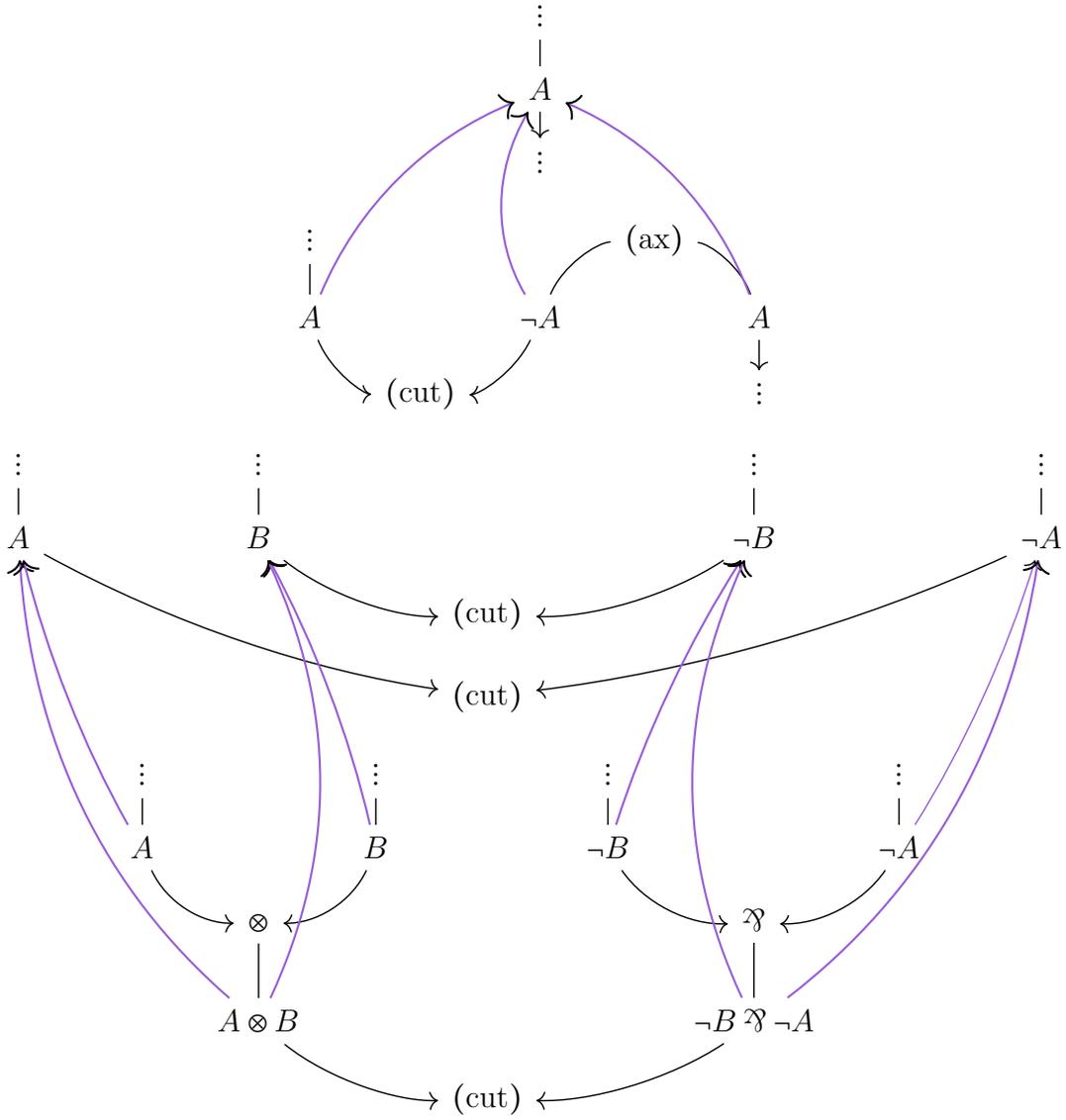

It is easily checked that

\begin{lemma}\label{lemma:section_ST} $S_\gamma T_\gamma = 1$.
\end{lemma}

The maps $T_\gamma, S_\gamma$ are morphisms of the pairs $(I_\pi, P_\pi), (I_{\pi'}, P_{\pi'})$ in the following sense:

	\begin{proposition}\label{prop:induced_maps}
		If $\gamma: \pi \lto \pi'$ is a reduction of proof structures then 
		\begin{equation}
		T_\gamma (I_{\pi'}) \subseteq I_\pi, \qquad S_\gamma(I_\pi) \subseteq I_{\pi'}
		\end{equation}
		and the induced morphisms of $k$-algebras $\overline{T}_\gamma, \overline{S}_\gamma$ making the squares in the following diagram commute
		\begin{equation}\label{eq:induced}
			\begin{tikzcd}[column sep = large, row sep = large]
				I_\pi\arrow[r] & P_\pi\arrow[r,"{p}"]\arrow[d, bend right,swap, "{S_\gamma}"] & R_\pi\arrow[d,bend right,swap, "{\overline{S}_\gamma}"]\\
				I_{\pi'}\arrow[r] & P_{\pi'}\arrow[r,"{p'}"]\arrow[u,bend right,swap,"{T_\gamma}"] & R_{\pi'}\arrow[u,bend right,swap, "{\overline{T}_\gamma}"]
			\end{tikzcd}
		\end{equation}
		are mutually inverse isomorphisms where $p, p'$ are the quotient maps.
	\end{proposition}
	\begin{proof}
		To prove $S_\gamma(I_\pi) \subseteq I_{\pi'}$ we first choose a link $l$ in $\pi$. If $l$ does not appear in the reduction (i.e, it is not the particular $\ax$ or $\cut$ involved in an $a$-redex nor the particular $\otimes, \parr,\cut$ in an $m$-redex) and no edge of $l$ is in the redex then every generator of $I_l$ is in $I_{\pi'}$. If $l$ does not appear in $\gamma$ but $l,\gamma$ share an edge, then $S_\gamma(I_l) \subseteq I_{\pi'}$ clearly holds in the case where $\gamma$ is an $m$-redex. In the case of $\ax$/$\cut$ (as for example the red $\ax$ link in Diagram \eqref{eq:adjacent} below)
		\begin{equation}\label{eq:adjacent}
			\begin{tikzcd}[column sep = small, row sep = small]
				& \textcolor{red}{\ax}\arrow[dr,bend left, dash, red]\arrow[dl,bend right, dash] &&&& \textcolor{red}{\ax}\arrow[dl,dash, bend right, red]\arrow[dr,bend left, dash, red]\\
				\neg A\arrow[d] && \textcolor{red}{A}\arrow[dr,bend right, red] && \textcolor{red}{\neg A}\arrow[dl,bend left, red] && \textcolor{red}{A}\arrow[d,red]\\
				\vdots&&&\textcolor{red}{\cut} &&& \textcolor{red}{\operatorname{c}}
			\end{tikzcd}
		\end{equation}
	 	the equations generating $I_l$ are of the form $X - X'$ where $X$ is an unoriented atom of the label of an edge in the redex and $X'$ is an unoriented atom of the label of an edge occuring in $l$. By inspection of Figure \ref{figure:a_redex_reduction} we see $S_\gamma(X - X') \in I_{\pi'}$.
	 	
	 	Now we consider the case where $l$ is in the redex and $\gamma$ reduces an $a$-redex. For all generators $X - X'$ of $I_l$ we have $S_\gamma(X) = S_\gamma(X')$ and so $S_\gamma(X - X') = 0 \in I_{\pi'}$. Now consider the case where $l$ is an $m$-redex. If $l$ is the displayed tensor or par link then any generator $X - X'$ of $I_{\pi'}$ associated to $l$ satisfies $S_\gamma(X) = S_\gamma(X')$ and so just as we had in the previous case we have $S_\gamma(X - X') = 0 \in I_{\pi'}$. The final case to consider is when $l$ is the displayed cut link. Let $X - X'$ be a generator for $I_l$. Let $l_1$ denote the cut link in $\pi$ with premises labelled $B, \neg B$ and $l_2$ denote the cut link with premises $A, \neg A$. We have that $S_\gamma(X) - S_\gamma(X') \in I_{l_1}$ if $X$ is an undirected atom of $A$ and $S_\gamma(X) - S_\gamma(X') \in I_{l_2}$ if $X$ is an undirected atom of $B$. In either case, we have that $S_\gamma(X) - S_\gamma(X') \in I_{\pi'}$.
		
		 Now we prove $T_\gamma(I_{\pi'}) \subseteq I_{\pi}$. A generator $X - X'$ of $I_{\pi'}$ is associated to a link $l$ of $\pi'$. If $\gamma$ is a reduction of an $a$-redex then $l$ is also present in $\pi$ and it is clear $T_{\gamma}(X -X') \in I_{\pi}$. If $\gamma$ is a reduction of an $m$-redex and $l$ is not one of the cut links shown in the bottom diagram of Figure \ref{figure:m_redex_reduction} then the same reasoning applies. If $l$ is one of these cut links the generator is either $X - X'$ for some unoriented atom $X$ in $A$ or $Y - Y'$ for some unoriented atom $Y$ in $B$, where we are using the notation of Figure \ref{figure:m_redex_reduction}. There exists unoriented atoms $X'',X'''$ of $A \otimes B$ and $\neg A \parr \neg B$ respectively which satisfy $X - X'', X'' - X''', X''' - X' \in I_{\pi}$. Hence $T_{\gamma}(X - X') = (X - X'') + (X'' - X''') + (X''' - X') \in I_{\pi}$, and similarly for $Y - Y'$. Hence $T_{\gamma}(I_{\pi'}) \subseteq I_{\pi}$.
		 
		 Hence $\overline{S}_\gamma, \overline{T}_\gamma$ exist. To prove they are mutually inverse it suffices to prove:
		 \begin{align}
		 	\overline{T}_\gamma \overline{S}_\gamma p &= p\\
		 	\overline{S}_\gamma \overline{T}_\gamma p' &= p'
		 \end{align}
	 as $p,p'$ are surjective. By commutativity of \eqref{eq:induced} this is equivalent to showing $p' S_\gamma T_\gamma = p', pT_\gamma S_\gamma = p$, or $p'(S_\gamma T_\gamma - 1) = 0, p(T_\gamma S_\gamma - 1)= 0$. It suffices to check this on generators, ie, on unoriented atoms. It is clear that $S_\gamma T_\gamma = 1$, however we have $T_\gamma S_\gamma \neq 1$. The circumstances where this is the case is indicated schematically in Figure \ref{fig:understanding_TS}.
	 
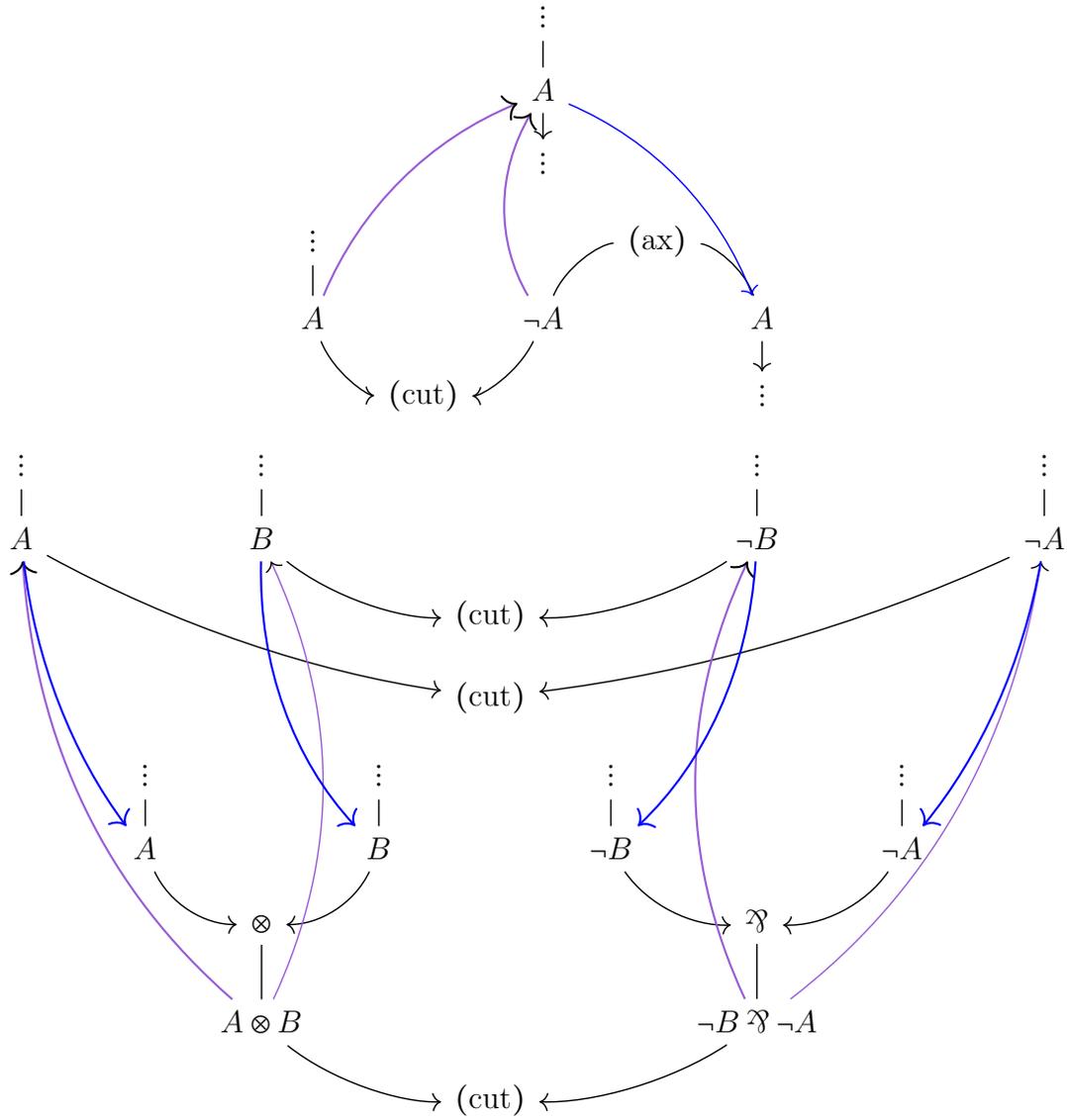
\begin{figure}
\[
 	\begin{tikzcd}[column sep = small, row sep = small]
 		&& \vdots \\
 		&& A \\
 		&& \vdots \\
 		\vdots &&& \ax \\
 		A && {\neg A} && A \\
 		& \cut &&& \vdots
 		\arrow[curve={height=-12pt}, no head, from=4-4, to=5-5]
 		\arrow[curve={height=12pt}, no head, from=4-4, to=5-3]
 		\arrow[curve={height=12pt}, from=5-1, to=6-2]
 		\arrow[curve={height=-12pt}, from=5-3, to=6-2]
 		\arrow[no head, from=1-3, to=2-3]
 		\arrow[from=2-3, to=3-3]
 		\arrow[draw={rgb,255:red,153;green,92;blue,214}, curve={height=-18pt}, from=5-3, to=2-3, thick]
 		\arrow[draw={rgb,255:red,153;green,92;blue,214}, curve={height=-18pt}, from=5-1, to=2-3, thick]
 		\arrow[from=5-5, to=6-5]
 		\arrow[no head, from=4-1, to=5-1]
 		\arrow[curve={height=-18pt}, from=2-3, to=5-5, blue]
 	\end{tikzcd}
\]
\[
 	\begin{tikzcd}[column sep = small, row sep = small]
 		\vdots &&& \vdots &&&& \vdots &&& \vdots \\
 		A &&& B &&&& {\neg B} &&& {\neg A} \\
 		&&&&& \cut \\
 		&&&&& \cut \\
 		&& \vdots && \vdots && \vdots && \vdots \\
 		&& A && B && {\neg B} && {\neg A} \\
 		&&& \otimes &&&& \parr \\
 		\\
 		&&& {A \otimes B} &&&& {\neg B \parr \neg A} \\
 		&&&&& \cut
 		\arrow[curve={height=12pt}, from=2-4, to=3-6]
 		\arrow[curve={height=-12pt}, from=2-8, to=3-6]
 		\arrow[curve={height=-12pt}, from=2-11, to=4-6]
 		\arrow[curve={height=12pt}, from=2-1, to=4-6]
 		\arrow[no head, from=1-1, to=2-1]
 		\arrow[no head, from=1-4, to=2-4]
 		\arrow[no head, from=1-8, to=2-8]
 		\arrow[no head, from=1-11, to=2-11]
 		\arrow[no head, from=5-5, to=6-5]
 		\arrow[no head, from=5-3, to=6-3]
 		\arrow[no head, from=5-7, to=6-7]
 		\arrow[no head, from=5-9, to=6-9]
 		\arrow[curve={height=-12pt}, from=6-5, to=7-4]
 		\arrow[curve={height=12pt}, from=6-3, to=7-4]
 		\arrow[curve={height=-12pt}, from=6-9, to=7-8]
 		\arrow[no head, from=7-4, to=9-4]
 		\arrow[curve={height=12pt}, from=9-4, to=10-6]
 		\arrow[no head, from=7-8, to=9-8]
 		\arrow[curve={height=-12pt}, from=9-8, to=10-6]
 		\arrow[draw={rgb,255:red,153;green,92;blue,214}, curve={height=30pt}, from=9-8, to=2-11]
 		\arrow[draw={rgb,255:red,153;green,92;blue,214}, curve={height=-30pt}, from=9-4, to=2-1, thick]
 		\arrow[curve={height=12pt}, from=6-7, to=7-8]
 		\arrow[draw={rgb,255:red,153;green,92;blue,214}, curve={height=-30pt}, from=9-8, to=2-8, thick]
 		\arrow[draw={rgb,255:red,153;green,92;blue,214}, curve={height=30pt}, from=9-4, to=2-4]
 		\arrow[curve={height=-12pt}, from=2-11, to=6-9, blue, thick]
 		\arrow[curve={height=12pt}, from=2-1, to=6-3, blue, thick]
 		\arrow[curve={height=18pt}, from=2-4, to=6-5, blue, thick]
 		\arrow[curve={height=-18pt}, from=2-8, to=6-7, blue, thick]
 	\end{tikzcd}
 \]
\caption{Understanding the composite $TS$.}
\label{fig:understanding_TS}
\end{figure}

First we consider the case in the first diagram of Figure \ref{fig:understanding_TS}. Let $l$ be the displayed $\ax$ link and let $X$ be an unoriented atom of $A$. Then by inspection of the figure we have $TS(X) - X \in I_l$. Now let $l'$ denote the displayed cut link. There exists an unoriented atom $X'$ such that $X - X' \in I_{l'}$ and $TS(X) - X' \in I_{l}$. Hence, $TS(X) - X = TS(X) - X' + X' - X \in I_{\pi}$. This shows that $p(T_\gamma S_\gamma - 1)(X)= 0$.

Now we consider the case in the second diagram of Figure \ref{fig:understanding_TS}. Let $X$ be an unoriented atom of $A \otimes B$ and denote the displayed tensor link by $l$. There exists an unoriented atom $X'$ either of $A$ or of $B$ such that $X - X' \in I_l$. In either case we have that $X - TS(X) = X - X' \in I_l$ and so $p(T_\gamma S_\gamma - 1)(X)= 0$. The remaining cases are similar.
	\end{proof}
	
\begin{cor}\label{cor:ideals_intersect} Let $\gamma: \pi \lto \pi'$ be a reduction of proof structures. Then $I_{\pi'} = T_\gamma^{-1}(I_\pi)$ or, identifying $P_{\pi'}$ as a subring of $P_\pi$ with inclusion $T_\gamma$,
\begin{equation}
I_{\pi'} = I_\pi \cap P_{\pi'}\,.
\end{equation}
\end{cor}
\begin{proof}
Proposition \ref{prop:induced_maps} shows that $T_\gamma(I_{\pi'}) \subseteq I_{\pi}$ so we have $I_{\pi'} \subseteq T_\gamma^{-1}(I_\pi)$. If $T_\gamma(f) \in I_\pi$ then by Lemma \ref{lemma:section_ST} $f = S_\gamma T_\gamma(f) \in S_\gamma( I_\pi ) \subseteq I_{\pi'}$ giving the reverse inclusion.
\end{proof}


It is not difficult to understand the effect of taking the quotient of $P_\pi$ by the ideal $I_\pi$. All those unoriented atoms connected across links (in the sense that $U - V$ appears in $G_\pi$) are identified in $R_\pi$ with each other. This common image can be thought of as the \emph{path} which connects all these differently named occurrences of the ``same'' variable; such paths are sometimes called persistent paths (see Definition \ref{defn:persistent_path}).

\begin{defn}\label{defn:simtilde} Let $\pi$ be a proof structure and $U, V \in U_\pi$. We write $U \sim V$ if $U - V \in G_\pi$ or $V - U \in G_\pi$. Let $\approx$ be the equivalence relation on $U_\pi$ generated by $\sim$.
\end{defn}


\begin{lemma}\label{lemma:simvsequal} If $U, V \in U_\pi$ then we have $U \approx V$ if and only if $U - V \in I_\pi$.
\end{lemma}
\begin{proof}
If $U \approx V$ then clearly $U - V \in I_\pi$. If $U - V \in I_\pi$ then we can write $U - V = \sum_j q_j ( U_j - V_j )$ with $U_j - V_j \in G_\pi$. Comparing the degrees of the left and right hand side, $q_j$ may be assumed to be in $k$, and in fact in $\{ 1, -1 \}$. From this sum we deduce a chain of $U_j \sim V_j$ relations leading from $V$ to $U$ as claimed.
\end{proof}

\begin{lemma}\label{lemma:sep_by_equiv} Let $\gamma: \pi \lto \pi'$ be a reduction of proof structures. Then if $U, V \in U_{\pi'}$ have $T_\gamma(U) \approx T_\gamma(V)$ in $U_\pi$ if and only if $U \approx V$ in $U_{\pi'}$.
\end{lemma}
\begin{proof}
Follows from Lemma \ref{lemma:simvsequal} and Corollary \ref{cor:ideals_intersect}.
\end{proof}

Recall that we call a formula $A$ labelling an edge incident at a conclusion of $\pi$ a \emph{conclusion} of $\pi$. If the sequence of oriented atoms of $A$ is $(U_1,u_1),...,(U_n,u_n)$ then we say that $U_i$ is \emph{positive} in $A$ if $u_i = +$ and otherwise it is \emph{negative}.

\begin{proposition}\label{prop:permutation}
	Let $\pi$ be a proof net with single conclusion $A$, and let
	\begin{equation}\label{eq:oriented_atoms}
		(Z_1,z_1),...,(Z_n,z_n)
	\end{equation}
	be the sequence of oriented atoms of $A$. Then $n = 2m$ is even, there are an equal number of positive and negative atoms, and if $\bold{U} = U_1,\ldots,U_m$ denotes the subsequence of positive unoriented atoms and $\bold{V} = V_1,\ldots,V_m$ the subsequence of negative unoriented atoms then
	\begin{itemize}
	\item[(i)] The inclusions $k[\bold{U}] \lto P_\pi$ and $k[\bold{V}] \lto P_\pi$ followed by the quotient $P_\pi \lto R_\pi$ are isomorphisms $\beta_+,\beta_-$ as in the diagram
	\begin{equation}\label{eq:diagram_U_V_P}
		\begin{tikzcd}
			k[\bold{U}]\arrow[d,dr,"{\beta_+}"]\arrow[d]\\
			P_\pi\arrow[r,twoheadrightarrow] & R_\pi\\
			k[\bold{V}]\arrow[u]\arrow[ur,swap,"{\beta_-}"]
		\end{tikzcd}
	\end{equation}
	\item[(ii)] The composite $\beta_-^{-1}\beta_+: k[\bold{U}] \lto k[\bold{V}]$ is 
\begin{equation}
	\beta_-^{-1}\beta_+(U_i) = V_{\sigma(i)},\quad 1 \leq i \leq m
\end{equation}
for some permutation $\sigma_\pi$ of $\lbrace 1,..., m \rbrace$.
\item[(iii)] Each equivalence class of the relation $\approx$ is the underlying set of a sequence 
\begin{equation}\label{eq:persistent_path_seq}
\mathscr{P} = ( Z_1, \ldots, Z_r )
\end{equation}
where for some $1 \le i \le m$ we have $Z_1 = V_{\sigma(i)}, Z_r = U_i$ and $Z_i \sim Z_{i+1}$ for $1 \le i < r$.
\end{itemize}
\end{proposition}
\begin{proof}
We define the \emph{weight} of a proof net $\pi$ to be the sum, over all $\cut$ links on some formula $A, \neg A$, of the weight $|A|$ of $A$ which is defined by recursively as follows: if $A$ is an atomic proposition then $|A| = 1$, if $A = B \diamondtimes C$ with $\diamondtimes \in \{ \otimes, \parr \}$ then $|A| = |B| + |C| + 1$, and $|\neg A| = |A|$. Note that the weight of a proof net strictly decreases under cut-elimination.

We prove the proposition by induction on the weight. All statements can be treated one persistent path at a time by Lemma \ref{lemma:sep_by_equiv}. In the base case a persistent path proceeds from a negative unoriented atom $V$ in the conclusion up through $\otimes, \parr$ links to an $\ax$ link and then downwards through $\parr, \otimes$ links to the positive unoriented atom $U$ in the conclusion. If the unoriented atoms involved are $V = Z_1, \ldots, Z_r = U$ then \eqref{eq:diagram_U_V_P} is
\begin{equation}
		\begin{tikzcd}
			k[U]\arrow[d,dr,"{\beta_+}"]\arrow[d]\\
			k[Z_1,\ldots,Z_r]\arrow[r,twoheadrightarrow] & k[\bold{Z}]/I\\
			k[V]\arrow[u]\arrow[ur,swap,"{\beta_-}"]
		\end{tikzcd}
	\end{equation}
where $I = (Z_2 - Z_1, \ldots, Z_r - Z_{r-1})$ so it is clear that $\beta_+, \beta_-$ are isomorphisms, and that $\beta_-^{-1} \beta_+(U) = V$ is the map sending the end of each persistent path to its beginning, a permutation of the set of unoriented atoms in the conclusion.

For the inductive hypothesis suppose the proposition is true of $\pi'$ and that there is a reduction $\gamma: \pi \lto \pi'$. Note $\pi'$ has the same conclusion as $\pi$. The claims follow from commutativity of the diagram
\begin{equation}
\xymatrix@C+1pc@R+1pc{
P_\pi \ar[rr] & & R_\pi\\
& k[\bold{U}] \ar[ul]\ar[dl] \ar[ur]^-{\beta_+} \ar[dr]_-{\beta_+}\\\
P_{\pi'} \ar[uu]^-{T_\gamma} \ar[rr] & & R_{\pi'} \ar[uu]^-{\cong}_-{\overline{T}_\gamma}
}
\end{equation}
and a similar one for $\bold{V}$.
\end{proof}

\begin{defn}\label{defn:persistent_path} Let $\pi$ be a proof net with single conclusion $A$. The sequences $\mathscr{P}$ of \eqref{eq:persistent_path_seq} whose underlying sets are the equivalence classes of $\approx$ are called \emph{persistent paths}.
\end{defn}

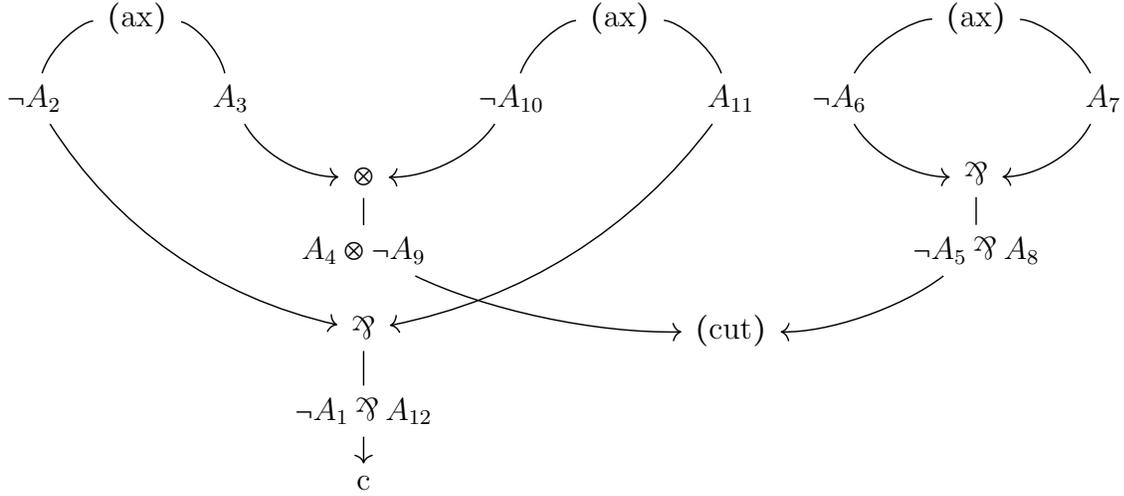
\begin{figure}
$
\begin{tikzcd}[column sep = tiny, row sep = small]
		& \ax &&&& \ax &&& \ax \\
		{\neg A_2} && A_3 && \neg A_{10} && A_{11} & \neg A_6 && A_7 \\
		&&& \otimes &&&&& \parr \\
		&&& A_4 \otimes \neg A_9 &&&&& \neg A_5 \parr A_8 \\
		&&& \parr &&& \cut\\
		&&& \neg A_1 \parr A_{12}\\
		&&& \operatorname{c}
		\arrow[curve={height=12pt}, no head, from=1-2, to=2-1]
		\arrow[curve={height=-12pt}, no head, from=1-2, to=2-3]
		\arrow[curve={height=22pt}, from=2-1, to=5-4]
		\arrow[curve={height=12pt}, no head, from=1-6, to=2-5]
		\arrow[curve={height=-12pt}, no head, from=1-6, to=2-7]
		\arrow[curve={height=-22pt}, from=2-7, to=5-4]
		\arrow[curve={height=-12pt}, from=2-5, to=3-4]
		\arrow[curve={height=12pt}, from=2-3, to=3-4]
		\arrow[curve={height=12pt}, no head, from=1-9, to=2-8]
		\arrow[curve={height=-12pt}, no head, from=1-9, to=2-10]
		\arrow[curve={height=-12pt}, from=2-10, to=3-9]
		\arrow[curve={height=12pt}, from=2-8, to=3-9]
		\arrow[no head, from=3-9, to=4-9]
		\arrow[curve={height=-12pt}, from=4-9, to=5-7]
		\arrow[curve={height=12pt}, from=4-4, to=5-7]
		\arrow[no head, from=3-4, to=4-4]
		\arrow[no head, from=5-4,to=6-4]
		\arrow[from=6-4,to=7-4]
	\end{tikzcd}
$
\caption{The canonical detour as a proof net, with $A_i = A = (X, +)$ for all $i$.}
\label{fig:detour_canonical}
\end{figure}

\begin{example}\label{example:canonical_detour} The proof net $\pi$ in Figure \ref{fig:detour_canonical} is the canonical detour \eqref{eq:intro_deduction_d} of the introduction with $B = A$, $\multimap$ written in terms of $\parr$ and the multiple conclusions combined with $\parr$ to a single conclusion. With $A_i = A = (X,+)$ for $1 \le i \le 12$ the set of unoriented atoms of $\pi$ is a disjoint union of twelve copies of $\{ X \}$ and we name the atoms $X_1,\ldots,X_{12}$. 

In the notation of the proposition, $\bold{U} = \{ X_{12} \}, \bold{V} = \{ X_1 \}$ and the single persistent path of $\pi$ is $X_1, X_2, \ldots, X_{12}$. Note that the persistent path goes ``through'' the $\cut$ link twice, once in each direction.
\end{example}

\begin{remark}\label{rmk:persistency}
By the definition of persistent paths, it follows from Proposition \ref{prop:induced_maps} that given any reduction $\gamma: \pi \lto \pi'$ and persistent path of a proof structure $\pi$, there is an associated persistent path $\scr{P}'$ in $\pi'$ given by removing the variables which occur in $\pi$ but not in $\pi'$ from $\scr{P}$. This will be used in the proofs of Theorems \ref{thm:elimination_ours}, \ref{thm:execution}.
\end{remark}

	
	

\begin{defn}\label{defn:reduction_sequence} A \emph{reduction sequence} $\Gamma: \pi \lto \pi'$ between proof structures $\pi, \pi'$ is a nonempty sequence of reductions (Definition \ref{def:reduction})
	\begin{equation}\label{eq:reduction_sequence_defn}
		\begin{tikzcd}[column sep = large]
			\pi = \pi_1\arrow[r,"{\gamma_1}"] & \cdots\arrow[r,"{\gamma_{n-1}}"] & \pi_n = \pi'\,.
			\end{tikzcd}
		\end{equation}
	This induces a sequence of $k$-algebra morphisms
	\begin{equation}
		\begin{tikzcd}[column sep = large]
			P_{\pi'}\arrow[r,"{T_{\gamma_{n-1}}}"] & \cdots\arrow[r,"{T_{\gamma_{1}}}"] & P_{\pi}
			\end{tikzcd}
		\end{equation}
	the composite of which we denote by $T_\Gamma: P_{\pi'} \lto P_{\pi}$.
	\end{defn}
	
\section{Elimination Theory}\label{section:elim_theory}



\subsection{Buchberger's algorithm}\label{section:buchberger_mod}


The Buchberger algorithm computes, given a generating sequence for an ideal $I$, a Gröbner basis for that ideal. There are a number of variations on this algorithm in the literature. We present here an optimisation of the algorithm of \cite[\S 2.10 Theorem 9]{Grobner} for the purposes of Elimination Theory, using a form of Euclidean Division with ``early stopping'' (Algorithm \ref{alg:division_adapted}). 

In this section $k$ is a field and $k[X_1,\ldots,X_n]$ a polynomial ring with a monomial order. We use the terminology of leading terms $\LT$, leading monomials $\LM$, leading coefficient $\operatorname{LC}$, and multidegree \cite[\S 2]{Grobner}. 

\begin{algorithm}
	\caption{Euclidean Division with Early Stopping}\label{alg:division_adapted}
	\begin{algorithmic}
		\Require $(f_1,\ldots,f_s),f$
		\State $p \gets f$
		\State $q_1,\ldots, q_s \gets 0,\ldots 0$
		\State $r \gets 0$
		\While{$p \neq 0$}
		\State $\text{DivOcc} \gets \texttt{False}$
		\State $i \gets 1$
		\While{$i \leq s \text{ and } \text{DivOcc} = \texttt{false}$}
		\If{$\operatorname{LT}f_i | \operatorname{LT}p$}
		\State $q_i \gets q_i + \operatorname{LT}p/\operatorname{LT}f_i$
		\State $p \gets p - (\operatorname{LT}p/\operatorname{LT}f_i)f_i$
		\State $\text{DivOcc} \gets \texttt{True}$
		\Else 
		\State $i \gets i + 1$
		\EndIf
		\EndWhile
		\If{$\operatorname{DivOcc} = \texttt{false}$}
		\State $r \gets p$
		\State $p \gets 0$
		\EndIf
		\EndWhile\\
		\Return{$(q_1,\ldots,q_s, r)$}
	\end{algorithmic}
\end{algorithm}

\begin{lemma}\label{lemma:division_property_alt} Let $G = (f_1,\ldots,f_s)$ be a sequence of elements of $k[X_1,\ldots,X_n]$. Given $f \in k[X_1,\ldots,X_n]$ let $(q_1,\ldots,q_s,r)$ be the output of Algorithm \ref{alg:division_adapted}. Then
\begin{equation}\label{eq:wosaj}
f = q_1 f_1 + \cdots + q_s f_s + r
\end{equation}
and if $r \neq 0$ the leading term of $r$ is not divisible by any $\LT(f_i)$ with $1 \le i \le s$. Furthermore, if $q_i f_i \neq 0$ then $\operatorname{multideg}(f) \ge \multideg(q_i f_i)$.
\end{lemma}
\begin{proof}
The expression $f = \sum_{i=1}^s q_i f_i + p + r$ holds at the end of every iteration of the outer while loop, and it terminates with $p = 0$ so \eqref{eq:wosaj} holds. If the algorithm returns $r \neq 0$ then it came from $r \gets p$ with $\operatorname{DivOcc} = \texttt{false}$ so the claim about divisibility of $\LT(r)$ follows. Termination and the other property follow as in the proof of the standard division algorithm \cite[\S 2.3]{Grobner}.
\end{proof}

\begin{defn} Let $G = (f_1,\ldots,f_s)$ be a sequence of polynomials and $f$ a polynomial.
We denote by $\ediv{f}{G}$ (resp. $\edives{f}{G}$) the remainder $r$ produced by the Euclidean division (resp. Algorithm \ref{alg:division_adapted}).
\end{defn}

To state our variation of Buchberger's algorithm we need some notation. For $i \neq j$ set
\[
[i,j] = \begin{cases} (i,j) &\text{if } i \le j\\ 
(j,i) &\text{if } i \ge j \end{cases}
\]
Recall that the \emph{$S$-polynomial} of $g,h \in k[X_1,\ldots,X_n]$ is
	\begin{equation}
		S(g,h) := \frac{X^\gamma}{\operatorname{LT}(g)}g - \frac{X^\gamma}{\operatorname{LT}(h)}h\,.
	\end{equation}
where $X^\gamma = \operatorname{LCM}(\LM(g),\LM(h))$. The only significant difference between Algorithm \ref{alg:elimination} below and the form of Buchberger's algorithm in \cite[\S 2.10 Theorem 9]{Grobner} is that we replace the line $S \gets \ediv{S(f_i,f_j)}{G}$ using normal Euclidean division with the line $S \gets \edives{S(f_i, f_j)}{G}$ using the division algorithm with early stopping. We also divide $S$ by its leading coefficient $\operatorname{LC}(S)$ before adding it to $G$, but this is insignificant.


\begin{algorithm}
	\caption{Buchberger with Early Stopping}\label{alg:elimination}
	\begin{algorithmic}
		\Require $F = (f_1,\ldots,f_s)$, returns a Gröbner basis for $\langle f_1,\ldots,f_s \rangle$.
		\State $B \gets \{ (i,j) \l 1 \le i < j \le s \}$
		\State $G \gets F$
		\State $t \gets s$
		\While{$B \neq \emptyset$}
		\State Let $(i,j) \in B$ be first in the lexicographic order.
		\If{$\operatorname{LCM}(\LM(f_i), \LM(f_j)) \neq \LM(f_i) \LM(f_j)$ \textbf{and} $\operatorname{Criterion}(f_i, f_j, B)$ \textbf{is false}}
		\State $S \gets \edives{S(f_i, f_j)}{G}$
		\If{$S \neq 0$}
		\State $t \gets t + 1$
		\State $f_t \gets \frac{1}{\operatorname{LC}(S)}S$
		\State $G \gets G \cup \{ f_t \}$
		\State $B \gets B \cup \{ (i,t) \l 1 \le i \le t - 1 \}$
		\EndIf
		\EndIf
		\State $B \gets B \setminus \{(i,j)\}$
		\EndWhile\\
	\Return{$G$}
	\end{algorithmic}
	where $\operatorname{Criterion}(f_i, f_j, B)$ is true provided that there is some $k \notin \{i,j\}$ for which the pairs $[i,k]$ and $[j,k]$ are \emph{not} in $B$ and $\LM(f_k)$ divides $\operatorname{LCM}(\LM(f_i),\LM(f_j))$.
\end{algorithm}

\begin{defn}
Given a set $G = \{ f_1, \ldots, f_s \}$ of polynomials we say that $f$ \emph{reduces to zero modulo $G$}, written $f \longrightarrow_G 0$ if $f$ can be written in the form $f = a_1 f_1 + \cdots + a_s f_s$ such that whenever $a_i f_i \neq 0$ we have $\operatorname{multideg}(f) \ge \operatorname{multideg}(a_if_i)$ (see \cite[\S 2.9 Definition 1]{Grobner}).
\end{defn}

\begin{proposition}
Algorithm \ref{alg:elimination} terminates and the output is a Gr\"{o}bner basis.
	\end{proposition}
\begin{proof}
Algorithm \ref{alg:elimination} is a variation on \cite[\S 2.10 Theorem 9]{Grobner} and the proof here is similar. Suppose that in some pass through the while loop, $G$ is enlarged to $G \cup \{ f_t \}$. By Lemma \ref{lemma:division_property_alt} we have that $\LT(f_t)$ is not divisible by $\LT(f)$ for any $f \in G$ and therefore $\langle \LT(G) \rangle$ is a proper subset of $\langle \LT(G \cup \{ f_t \}) \rangle$. Since $k[X_1,\ldots,X_n]$ is Noetherian this means that at some point $G$ stops growing, but then in each subsequent step the size of $B$ decreases by one. Hence the Algorithm terminates in finitely many steps.
	
To prove the output is a Gröbner basis, we first establish the following claim: at the end of every pass through the while loop, $B$ has the property that if $1 \le i < j \le t$ and $(i,j) \notin B$ then
\begin{equation}\label{eq:two_cond}
S(f_i, f_j) \longrightarrow_G 0 \quad \text{ or } \quad \operatorname{Criterion}(f_i, f_j, B) \text{ holds}\,.
\end{equation}
This is initially true since $B$ is the set of all possible pairs. To prove the claim we need to argue that if it holds for $B$ then it continues to hold when $B$ changes to $B'$.

To prove this assume $(i,j) \notin B'$. If $(i,j) \in B$ then $B' = B \setminus \{(i,j)\}$. In the step before $(i,j)$ is removed if either $\operatorname{LCM}(\LT(f_i),\LT(f_j)) = \LT(f_i) \LT(f_j)$ or $\operatorname{Criterion}(f_i,f_j,B)$ is true then one argues as in the original proof from \cite{Grobner}. If not then the algorithm computes $S \gets \edives{S(f_i,f_j), G}$. If $S = 0$ then $S(f_i,f_j) \longrightarrow_G 0$ by Lemma \ref{lemma:division_property_alt}. If $S \neq 0$ then we enlarge $G$ to $G' = G \cup \{ \frac{1}{\operatorname{LC}(S)} S \}$. Examining the division algorithm with early stopping it clear that $\edives{S(f_i, f_j)}{G'} = 0$ so $S(f_i,f_j) \longrightarrow_{G'} 0$ and \eqref{eq:two_cond} holds for $B',G'$.

It remains to consider the case $(i,j) \notin B$. Then \eqref{eq:two_cond} holds for $B$. Let $G, G'$ be the state of the sequence of polynomials in the two steps, so $G' = G$ or $G'$ has one additional element. If $S(f_i, f_j) \longrightarrow_G 0$ then $S(f_i, f_j) \longrightarrow_{G'} 0$. If $\operatorname{Criterion}(f_i, f_j, B)$ holds then there is some $k \in \{i,j\}$ with $k \le |G|$ such that $[i,k],[j,k]$ are not in $B$ and $\LT(f_k)$ divides $\operatorname{LCM}(\LT(f_i),\LT(f_j))$. Any new pairs in $B'$ involve $|G| + 1$ so $[i,k],[j,k] \notin B'$ and hence $\operatorname{Criterion}(f_i,f_j,B')$ holds.

The argument that $G$ is a Gröbner basis when $B = \emptyset$ goes through as before.
\end{proof}

\begin{defn} Given a sequence $F = (f_1,\ldots,f_s)$ of polynomials and a monomial order $<$ on $k[X_1,\ldots,X_n]$ we denote by $\mathbb{B}_{es}(F, <)$ the output of Algorithm \ref{alg:elimination}.
\end{defn}


We state the main theorem of Elimination Theory. Let $\bold{X} = \{X_1,\ldots,X_n\}$ and $\bold{Y} = \{Y_1,\ldots,Y_m\}$ be variables. We suppose given an ideal $I \subseteq k[\bold{X}, \bold{Y}]$ and we are interested to know the equations between the $\bold{X}$-variables that are implied by the equations in $I$. We call this set of equations the elimination ideal:

\begin{defn} The \emph{elimination ideal} of $I$ is the ideal $I \cap k[\bold{X}]$ in $k[\bold{X}]$.
\end{defn}

In the situation of Section \ref{section:intro_example} we can think of the $\bold{X}$-variables as living at the ``boundary'' of the proof net, and $\bold{Y}$ as being the ``interior'' variables.

\begin{thm}[The Elimination Theorem]\label{thm:elimination} Let $I \subseteq k[\bold{X}, \bold{Y}]$ be an ideal and $G$ a Gröbner basis of $I$ with respect to a lexicographic order where $X_i < Y_j$ for all $1 \le i \le n, 1 \le j \le m$. Then $G \cap k[\bold{X}]$ is a Gröbner basis for $I \cap k[\bold{X}]$.
\end{thm}
\begin{proof}
See \cite[\S 3.1 Theorem $2$]{Grobner}.
\end{proof}

\subsection{Buchberger and Falling Roofs}\label{section:falling_roofs}

We introduce graphical notation for understanding the operation of Algorithm \ref{alg:elimination} on a simple class of ideals. Fix a polynomial ring $k[X_1,\ldots,X_n]$, let $<$ be a total order on the set $\{X_1,\ldots,X_n\}$ and take the lexicographic monomial order determined by $<$. Let $\sigma$ be the permutation uniquely defined by
\[
X_{\sigma^{-1} 1} < X_{\sigma^{-1} 2} < \cdots < X_{\sigma^{-1} n}\,.
\]
The position of $X_i$ in this sequence is $\sigma(i)$. We view $\sigma$ as assigning a \emph{height} to variables:

\begin{defn} The \emph{realisation} of $<$ is the oriented graph $\mathscr{R}_<$ with vertices
\begin{equation}\label{eq:vertices_realisation}
\big\{(i, \sigma i) \l 1 \le i \le n \big\} \subseteq \mathbb{R}^2
\end{equation}
with an edge between $(i, \sigma i), (i+1, \sigma(i+1))$ for $i < n$, and $(i, \sigma i)$ decorated with $X_i$. The orientation of the edge is from $(i, \sigma i)$ to $(j, \sigma j)$ if $X_i < X_j$.
\end{defn}

\begin{defn} A \emph{$<$-graph} is an oriented graph on the vertex set \eqref{eq:vertices_realisation} with the property that if there is an edge from a vertex $(i, \sigma i)$ decorated with $X_i$ to a vertex decorated with $X_j$ then $X_i < X_j$.
\end{defn}

\begin{defn} A $<$-graph is \emph{linear} if every vertex has valence at most two.
\end{defn}

Clearly $\mathscr{R}_{<}$ is a linear $<$-graph. In this section we simply write \emph{graph} for $<$-graph, since the ordering $<$ on the variables is fixed. We write $e: X_i \lto X_j$ for an edge from a vertex decorated with $X_i$ to a vertex decorated with $X_j$.

\begin{example}\label{example_allowedgraph_test} Let $X_1, \ldots, X_6$ be ordered by $X_5 < X_1 < X_6 < X_3 < X_2 < X_4$. Then $\mathscr{R}_<$ is
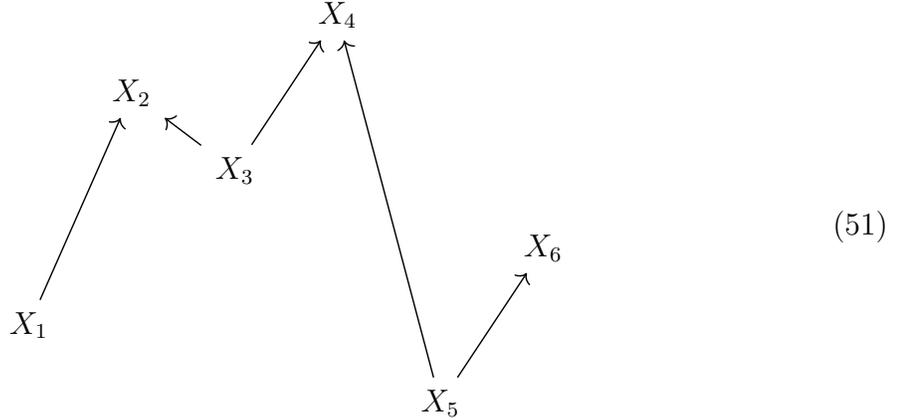
\begin{equation}\label{eq:realisation_1}
\begin{tikzcd}[row sep = small, column sep = small]
			& & & X_4 \\
			& X_2 \\
			& & X_3\\
			& & & & & X_6 \\
			X_1\\
			& & & & X_5\\
			\arrow[from=5-1, to=2-2]
			\arrow[from=3-3, to=2-2]
			\arrow[from=3-3, to=1-4]
			\arrow[from=6-5, to=1-4]
			\arrow[from=6-5, to=4-6]
		\end{tikzcd}
\end{equation}
\end{example}

\begin{remark}\label{remark:xaxis_vs_yaxis} The realisation involves two orderings on the variables: the order reading along the $x$-axis (the $X_1,\ldots,X_n$ order) and the order reading along the $y$-axis (the $<$-order). The realisation is a graphical presentation of the relation between these orders.
\end{remark}

\begin{defn} A \emph{roof} in a graph $\mathscr{S}$ is an ordered pair $(e,e')$ of edges $e: X_i \lto X_l, e': X_k \lto X_l$ with the same endpoint $X_l$ and $X_i < X_k$. We call $X_l$ the \emph{tip} of the roof.
\end{defn}



\begin{defn}\label{defn:gen_set_allowed} Given a graph $\mathscr{S}$ we define
\[
G_{\mathscr{S}} = \big\{ X_j - X_i \l e: X_i \lto X_j \text{ is an edge in } \mathscr{S} \big\}\,.
\]
\end{defn}

\begin{defn} Let $\mathscr{S}$ be a graph with edges $e: X_i \lto X_j, e': X_k \lto X_l$. Then $e < e'$ if the pair $(X_j, X_i)$ precedes $(X_l, X_k)$ lexicographically, that is, either $X_j < X_l$ or $X_j = X_l$ and $X_i < X_k$. This is a total order on the set of edges of $\mathscr{S}$.
\end{defn}

The map sending $e: X_i \lto X_j$ to $X_j - X_i$ is a bijection between the edge set of $\mathscr{S}$ and $G_{\mathscr{S}}$ and we give the latter set the total order induced by this bijection and the total order on the former set just defined.

\begin{defn} A roof $(e,e')$ precedes a roof $(d, d')$ if $e < d$ or $e = d$ and $e' < d'$.
\end{defn}

We will only consider roofs in linear graphs: since in a linear graph distinct roofs have distinct tips, $(e,e') < (d,d')$ simply means that $e < d$, or what is the same, the tip of the first roof is lower in the realisation than the tip of the second.


We now explain the operation of Algorithm \ref{alg:elimination} on input $(G_{\mathscr{S}}, <)$ for a linear graph $\mathscr{S}$ by giving an ``isomorphic'' algorithm (Algorithm \ref{alg:elimination_roofs_falling} below) which is organised around roofs and is easier to understand. The algorithm works by marking a subset of edges in a graph $\mathscr{N}$ derived from $\mathscr{S}$ as \emph{dead} (an edge that is not dead is \emph{live}). A roof is \emph{live} if both the edges in it are live. We indicate dead edges diagrammatically by dotted lines.

\begin{algorithm}
	\caption{Falling Roofs}\label{alg:elimination_roofs_falling}
	\begin{algorithmic}
		\Require Linear graph $\mathscr{S}$
		\State $\mathscr{N} \gets \mathscr{S}$
		\State Mark all edges in $\mathscr{N}$ as live
		\While{$\mathscr{N}$ contains a live roof}
		\State $(e,e') \gets$ the first live roof in $\mathscr{N}$
		\State Mark $e,e'$ as dead
		\State If it does not exist, add to $\mathscr{N}$ a live edge $d$ as shown below:
		\begin{equation}\label{eq:divisor_roof_0}
		\xymatrix{
			& \bullet\\
			\bullet \ar@{.>}[ur]^-{e}\ar[rr]_-{d} & & \bullet \ar@{.>}[ul]_-{e'}
		}
		\end{equation}
		\While{$d$ is part of a live roof in $\mathscr{N}$}
			\If{$(d,e'')$ is a live roof in $\mathscr{N}$}
				\State Mark $e''$ as dead
				\State If it does not exist, add to $\mathscr{N}$ a live edge $d'$ as shown below:
				\begin{equation}\label{eq:divisor_roof_1}
				\xymatrix{
				& \bullet\\
				\bullet \ar[ur]^-{d}\ar[rr]_-{d'} & & \bullet \ar@{.>}[ul]_-{e''}
				}
				\end{equation}
				\State Remove $d$ from $\mathscr{N}$
				\State $d \gets d'$
			\ElsIf{$(e'',d)$ is a live roof in $\mathscr{N}$}
				\State Mark $e''$ as dead
				\State If it does not exist, add to $\mathscr{N}$ a live edge $d'$ as shown below:
				\begin{equation}\label{eq:divisor_roof_2}
				\xymatrix{
				& \bullet\\
				\bullet \ar@{.>}[ur]^-{e''}\ar[rr]_-{d'} & & \bullet \ar[ul]_-{d}
				}
				\end{equation}
				\State Remove $d$ from $\mathscr{N}$
				\State $d \gets d'$
			\EndIf
		\EndWhile
		\EndWhile\\
	\Return{$\mathscr{N}$}
	\end{algorithmic}
\end{algorithm}

\begin{proposition}\label{prop:buchberger_equals_fallingroof} For any linear graph $\mathscr{S}$:
\begin{itemize}
\item[(i)] Algorithm \ref{alg:elimination_roofs_falling} terminates.
\item[(ii)] At every step $\mathscr{N}$ is a graph and the subgraph of live edges in $\mathscr{N}$ is linear.
\item[(iii)] If the output of of Algorithm \ref{alg:elimination} on input $(G_{\mathscr{S}}, <)$ is $G$ and the output of Algorithm \ref{alg:elimination_roofs_falling} on input $\mathscr{S}$ is $\mathscr{N}$ then $G = G_{\mathscr{N}}$.
\end{itemize}
\end{proposition}
\begin{proof}
Given an edge $e: X_i \lto X_j$ we write $f_e$ for $X_j - X_i$. We prove (iii) and along the way (i), (ii) by stepping through Algorithm \ref{alg:elimination} on input $G_{\mathscr{S}}$ (with intermediate variables $G,B,S$) and Algorithm \ref{alg:elimination_roofs_falling} on input $\mathscr{S}$ (with intermediate variable $\mathscr{N}$). At each step let $B_{rel} \subseteq B$ denote the set of pairs $(e,e')$ which are \emph{relevant} in the sense that
\[
\operatorname{LCM}(\LM(f_e), \LM(f_{e'})) \neq \LM(f_e) \LM(f_{e'}) \textbf{ and } \operatorname{Criterion}(f_e, f_{e'}, B) \textbf{ is false}\,.
\]
We will prove by induction on the number of steps through Algorithm \ref{alg:elimination} that at each step the following proposition $\textbf{IP}$ holds:
\begin{itemize}
\item[\textbf{IP}:]$\mathscr{N}$ is a graph, the live edges in $\mathscr{N}$ form a linear graph, no dead edge in $\mathscr{N}$ ends where a live edge begins, $G = G_{\mathscr{N}}$, and $B_{rel}$ is the set of live roofs in $\mathscr{N}$.
\end{itemize}

It is clear that these conditions are all initially true, except perhaps the last claim about $B_{rel}$. To examine this condition at the beginning of the algorithm, let $e: X_i \lto X_j$ and $e': X_k \lto X_l$ be edges in $\mathscr{N} = \mathscr{S}$. Then $\LT(f_{e}) = X_j, \LT(f_{e'}) = X_l$ and
\[
\operatorname{LCM}(\LM(f_e), \LM(f_{e'})) = \operatorname{LCM}(X_j, X_l) = \begin{cases} X_j X_l & X_j \neq X_l \\ X_j & X_j = X_l \end{cases}
\]
so $\operatorname{LCM}(\LM(f_e), \LM(f_{e'})) \neq \LM(f_e) \LM(f_{e'})$ if and only if $X_j = X_l$, that is, if and only if $e, e'$ have a common endpoint in $\mathscr{N}$. Suppose this is the case. Then $\operatorname{Criterion}(f_e, f_{e'}, B)$ is true if there exists some edge $e'' \notin \{ e, e'\}$ for which $[e,e''], [e',e'']$ are \emph{not} in $B$ and the edges $e, e', e''$ all have the same endpoint. But since $\mathscr{N}$ is linear this is impossible, so the criterion is false. Hence a pair $(e,e')$ is in $B_{rel}$ if and only if it is a roof in $\mathscr{S}$, and since initially all roofs are live this proves the claim.
\\

Suppose that \textbf{IP} is true up to some step of the outer \textbf{while} loops of the two algorithms. Let $\mathscr{N}_0$ denote the state of $\mathscr{N}$ at the beginning of this step of Algorithm \ref{alg:elimination_roofs_falling}. Algorithm \ref{alg:elimination} looks in order through pairs $(e,e')$ for one that is relevant, or what is the same by the inductive hypothesis, for a live roof in $\mathscr{N}$. Let $(e,e')$ be a live roof with $e: X_i \lto X_j, e': X_k \lto X_j$. Algorithm \ref{alg:elimination} computes the $S$-polynomial
\begin{equation}
S := S(f_e, f_{e'}) = f_e - f_{e'} = X_k - X_i
\end{equation}
which is $f_d$ for $d: X_i \lto X_k$ in Algorithm \ref{alg:elimination_roofs_falling}. The next step of Algorithm \ref{alg:elimination} is to compute $\edives{S}{G}$. By \textbf{IP} we have $G = G_{\mathscr{N}}$. A polynomial in $G$ with leading term $X_k$ is an edge in $\mathscr{N}$ ending at $X_k$, which must be live by \textbf{IP} as $e'$ was live. If there is no polynomial in $G$ with leading term $X_k$ then the division algorithm (with early stopping) returns $S$, $\frac{1}{\operatorname{LC}(S)} S$ is added to $G$, $d$ is added to $\mathscr{N}$, $e,e'$ are marked as dead and the step is complete. It is straightforward to check that \textbf{IP} still holds.

Otherwise there is a polynomial $f$ in $G$ with leading term $X_k$. If $f = S$ then $\edives{S}{G} = 0$ so nothing is added to $G$ by Algorithm \ref{alg:elimination}, nothing is added to $\mathscr{N}$ by Algorithm \ref{alg:elimination_roofs_falling} and the step is complete; again \textbf{IP} still holds. We are left with the case where $f \neq S$ exists in $G$ with leading term $X_k$. For this we analyse the inner \textbf{while} loop in Algorithm \ref{alg:elimination_roofs_falling}. The key observations are that at the beginning of each step of this inner loop:
\begin{itemize}
\item[(a)] $d$ is part of at most one live roof in $\mathscr{N}$.
\item[(b)] there is no dead edge in $\mathscr{N}$ ending where $d$ ends.
\end{itemize}
Both claims are true for the initial $d$ by \textbf{IP}. We call $d$ Type $1$ if it comes from $d'$ in \eqref{eq:divisor_roof_1} and Type $2$ if it comes from $d'$ in \eqref{eq:divisor_roof_2}. If $d$ is Type $1$ both claims (a),(b) follow from \textbf{IP}. In Type $2$, we have to analyse the previous steps; if the previous step was Type $2$ (or there was no previous step) then there is a live edge in $\mathscr{N}_0$ beginning where $d$ ends and from this we deduce (a),(b). Otherwise there are some Type $1$ steps intervening between the current step and the most recent Type $2$ step (or the original roof \eqref{eq:divisor_roof_0} if there are no previous Type $2$ steps). However, in Type $1$ steps the $d,d'$ edges start from the same vertex, so the most recent Type $2$ step (or original roof) provides a live edge beginning where $d$ ends, from which we deduce (a),(b) again from \textbf{IP}.
\\

We now prove that the inner \textbf{while} loop of Algorithm \ref{alg:elimination_roofs_falling} aligns with the division step of Algorithm \ref{alg:elimination}. By (a) there is a unique polynomial $f = f_{e''}$ in $G$ with the same leading term as $f_d$. The division algorithm replaces $S = f_d$ by the dividend
\begin{equation}\label{eq:dividend}
S \leftarrow S - (\LT(S)/\LT(f_{e''})) f_{e''} = f_d - f_{e''} = \pm f_{d'}
\end{equation}
There are two cases, corresponding to the two \textbf{if} statements in Algorithm \ref{alg:elimination_roofs_falling}. In the first case $f_{d'} = S$ and in the second $f_{d'} = -S$. 

If $f_{d'} \in G$ then by linearity of $\mathscr{N}_0$ and (b) there is no other polynomial in $G$ with the same leading term as $f_{d'}$ and so the division algorithm returns zero. Likewise, the inner \textbf{while} loop of Algorithm \ref{alg:elimination_roofs_falling} terminates because there is no live roof involving $d'$. To prove \textbf{IP} note that the only changes to $\mathscr{N}$ during the inner \textbf{while} loop are to mark some edges as dead. Marking the $e''$ as dead does not change the set of live roofs by (a).

If $f_{d'} \notin G$ then by (b) the leading term of $f_{d'}$ is the same as the leading term of some element of $G$ (that is, division continues) if and only if $d'$ is part of a live roof in $\mathscr{N}$ (that is, the inner \textbf{while} loop continues). If division and the \textbf{while} loop both stop, then $f_{d'} = \frac{1}{\operatorname{LC}S} S$ is added to $G$ and $d'$ is added to $\mathscr{N}$. As before one can check that \textbf{IP} is still true. Otherwise the division and \textbf{while} loop continue, and we repeat the above analysis. This completes the proof of the inductive step, showing that \textbf{IP} holds in the next iteration of the outer \textbf{while} loop and so (iii) follows by induction.

(i) Algorithm \ref{alg:elimination} terminates when $B$ is empty, at which step $B_{rel}$ must be empty and so there are no live roofs in $\mathscr{N}$ so Algorithm \ref{alg:elimination_roofs_falling} also terminates.
\end{proof}

Figure \ref{figure:falling_roofs} shows the Falling Roofs algorithm applied to the graph of Example \ref{example_allowedgraph_test}.

\begin{cor}\label{cor:falling_roofs_computes_grobner} If $\mathscr{S}$ is a linear graph and $\mathscr{N}$ is the output of Algorithm \ref{alg:elimination_roofs_falling} on $\mathscr{S}$ then $G_{\mathscr{N}}$ is a Gröbner basis for the ideal generated by $G_{\mathscr{S}}$.
\end{cor}

Examining the algorithm it is easy to verify the following:

\begin{lemma}\label{lemma:valence_one_pres} Let $\mathscr{S}$ be a linear graph and $\mathscr{N}$ the output of Algorithm \ref{alg:elimination_roofs_falling} on $\mathscr{S}$. If a vertex has valence one in $\mathscr{S}$ then it is incident with precisely one live edge in $\mathscr{N}$.
\end{lemma}

\begin{figure}
		\begin{center}
			\begin{tabular}{ c c }
			\begin{tikzcd}[row sep = small, column sep = small]
			& & & X_4 \\
			& X_2 \\
			& & X_3\\
			& & & & & X_6 \\
			X_1\\
			& & & & X_5\\
			\arrow[from=5-1, to=2-2]
			\arrow[from=3-3, to=2-2]
			\arrow[from=3-3, to=1-4]
			\arrow[from=6-5, to=1-4]
			\arrow[from=6-5, to=4-6]
		\end{tikzcd}
		&
		\begin{tikzcd}[row sep = small, column sep = small]
			& & & X_4 \\
			& X_2 \\
			& & X_3\\
			& & & & & X_6 \\
			X_1\\
			& & & & X_5\\
			\arrow[dotted, from=5-1, to=2-2]
			\arrow[dotted, from=3-3, to=2-2]
			\arrow[from=5-1, to=3-3]
			\arrow[from=3-3, to=1-4]
			\arrow[from=6-5, to=1-4]
			\arrow[from=6-5, to=4-6]
		\end{tikzcd}
		\\
		\begin{tikzcd}[row sep = small, column sep = small]
			& & & X_4 \\
			& X_2 \\
			& & X_3\\
			& & & & & X_6 \\
			X_1\\
			& & & & X_5\\
			\arrow[dotted, from=5-1, to=2-2]
			\arrow[dotted, from=3-3, to=2-2]
			\arrow[from=5-1, to=3-3]
			\arrow[dotted, from=3-3, to=1-4]
			\arrow[dotted, from=6-5, to=1-4]
			\arrow[from=6-5,to=3-3]
			\arrow[from=6-5, to=4-6]
		\end{tikzcd}
		&
		\begin{tikzcd}[row sep = small, column sep = small]
			& & & X_4 \\
			& X_2 \\
			& & X_3\\
			& & & & & X_6 \\
			X_1\\
			& & & & X_5\\
			\arrow[dotted, from=5-1, to=2-2]
			\arrow[dotted, from=3-3, to=2-2]
			\arrow[dotted, from=5-1, to=3-3]
			\arrow[dotted, from=3-3, to=1-4]
			\arrow[dotted, from=6-5, to=1-4]
			\arrow[from=6-5,to=5-1]
			\arrow[from=6-5, to=4-6]
		\end{tikzcd}
			\end{tabular}
		\end{center}
\caption{Algorithm \ref{alg:elimination_roofs_falling} applied to the graph of Example \ref{example_allowedgraph_test}, reading from left to right and top to bottom.}
\label{figure:falling_roofs}
\end{figure}
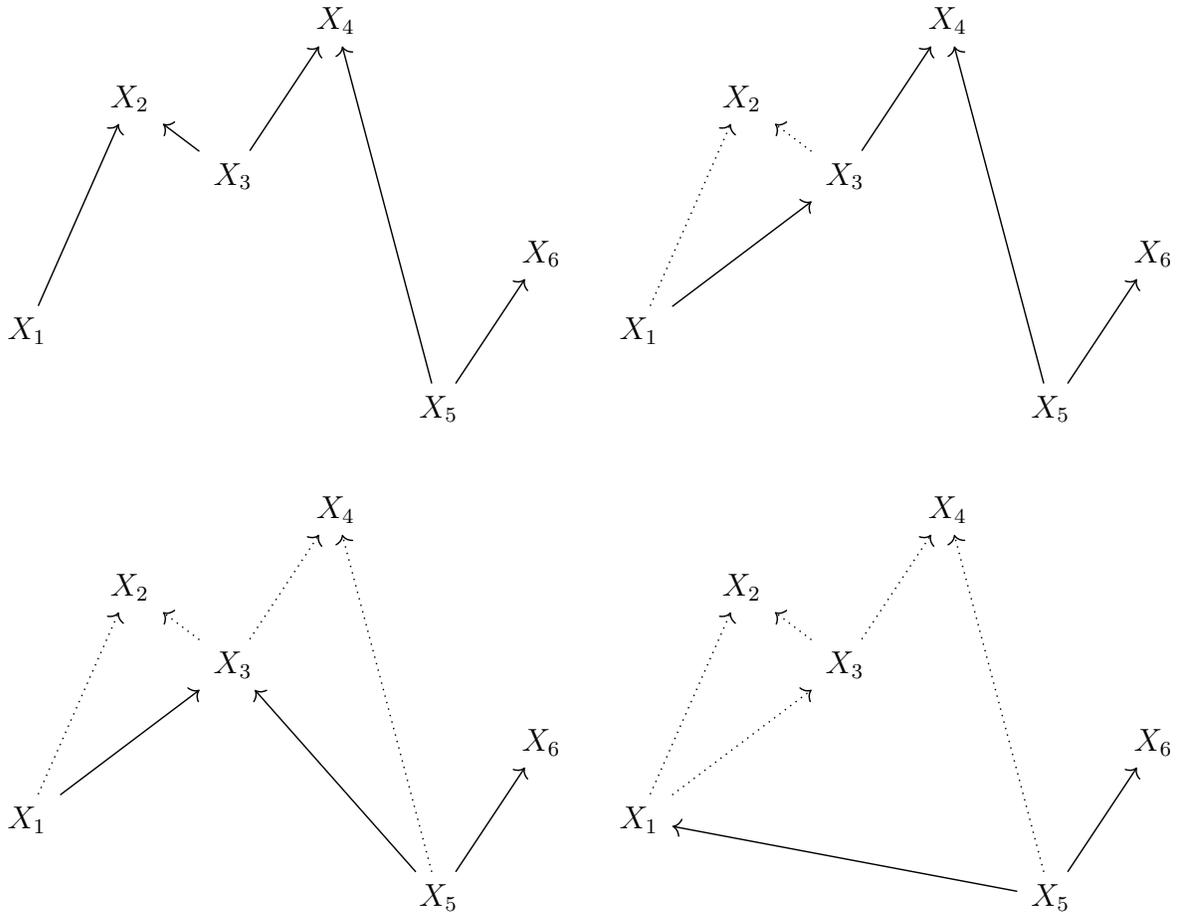

\subsection{Falling Roofs and Elimination}\label{section:falling_roofs_elimination}

We consider the Falling Roofs algorithm (Algorithm \ref{alg:elimination_roofs_falling}) in the context of Elimination Theory. Let $\bold{Z} = \{Z_1,\ldots,Z_r\}$ be a set of variables and $<$ a total order on $\bold{Z}$. We denote the associated lexicographic monomial order also  by $<$. We assume given a partition $\bold{Z} = \bold{X} \cup \bold{Y}$ such that if $Z_i \in \bold{X}$ and $Z_j \in \bold{Y}$ then $Z_i < Z_j$. That is, the $\bold{X}$-variables are ``low'' and the $\bold{Y}$-variables are ``high''. Let $\mathscr{S}$ be a linear $<$-graph with associated ideal
\[
I = \langle G_{\mathscr{S}}\rangle \subseteq k[\bold{X},\bold{Y}]\,.
\]
Let $\mathscr{N}$ be the output of the Falling Roofs algorithm on $\mathscr{S}$. Then by Corollary \ref{cor:falling_roofs_computes_grobner} $G_{\mathscr{N}}$ is a Gröbner basis for $I$ and hence by the Elimination Theorem (Theorem \ref{thm:elimination})
\[
G_{\mathscr{N}} \cap k[\bold{X}] \text{ is a Gröbner basis for } I \cap k[\bold{X}]\,.
\]
If we write $X_1,\ldots,X_n$ for the variables in $\bold{X}$ and similarly $Y_1,\ldots,Y_m$ for the variables in $\bold{Y}$ then $G_{\mathscr{N}} \cap k[\bold{X}]$ is simply all the polynomials $f_e = X_j - X_i$ where $e: X_i \lto X_j$ is an edge in $\mathscr{N}$ between $\bold{X}$-variables, that is, we delete any edges involving $\bold{Y}$-variables. This has a natural interpretation as ``cutting off the top'' of the realisation: if we draw a line through the realisation at the height of the highest $\bold{X}$-variable, intersecting with $k[\bold{X}]$ means deleting any vertices above the line and all edges incident with such vertices.

\begin{example}\label{example:cut_the_realisation} In Example \ref{example_allowedgraph_test} take $\bold{X} = \{ X_1, X_5, X_6 \}$ and $\bold{Y} = \{ X_2, X_3, X_4 \}$. Then the set $G_{\mathscr{N}} \cap k[\bold{X}]$ is a Gröbner basis for $I \cap k[\bold{X}]$ where $G_{\mathscr{N}} \cap k[\bold{X}] = \{ X_6 - X_5, X_1 - X_5 \}$ is determined by the subgraph in the box below:
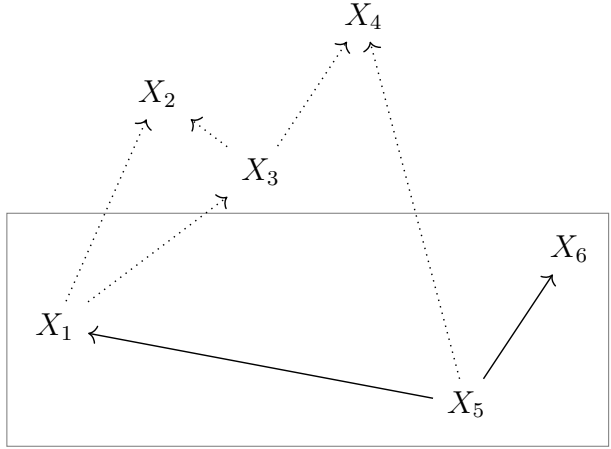
\[
\tikz[
overlay]{
    \filldraw[fill=white,draw=blue!50!yellow] (-0.2,-2.8) rectangle (7.8,0.3);
}
\begin{tikzcd}[row sep = small, column sep = small]
			& & & X_4 \\
			& X_2 \\
			& & X_3\\
			& & & & & X_6 \\
			X_1\\
			& & & & X_5\\
			\arrow[dotted, from=5-1, to=2-2]
			\arrow[dotted, from=3-3, to=2-2]
			\arrow[dotted, from=5-1, to=3-3]
			\arrow[dotted, from=3-3, to=1-4]
			\arrow[dotted, from=6-5, to=1-4]
			\arrow[from=6-5,to=5-1]
			\arrow[from=6-5, to=4-6]
		\end{tikzcd}
\]
Note that we follow the notational convention introduced just prior to Proposition \ref{prop:buchberger_equals_fallingroof}, of denoting some edges of $\mathscr{N}$ with solid lines (the \emph{live} edges) and some with dotted lines (the \emph{dead}) edges. Shown is the state of these labels at the end of the algorithm. These are all edges of $\mathscr{N}$ and so for example $X_3 - X_1 \in G_{\mathscr{N}}$.
\end{example}

We assume $X_1,\ldots,X_n$ is the order these variables appear in the list $Z_1,\ldots,Z_r$, that is, the order the $\bold{X}$-variables are enumerated is their order of appearance from left to right in the realisation (which is \emph{not} necessarily the $<$ order). 

\begin{proposition}\label{prop:cutting_is_living} Suppose that $\mathscr{S}$ is a subgraph of the realisation $\mathscr{R}_{<}$ satisfying the following conditions
\begin{itemize}
\item[(i)] Every vertex in $\mathscr{S}$ of valence one is labelled with an element from $\bold{X}$.
\item[(ii)] We have $X_1 < X_2 < \cdots < X_n$.
\end{itemize}
Let $\mathscr{N}^+$ denote the subgraph of live edges in $\mathscr{N}$. Then $\mathscr{N}^+$ is a linear graph, every vertex of nonzero valence in $\mathscr{N}^+$ is labelled with an element of $\bold{X}$, and $G_{\mathscr{N}^+} = G_{\mathscr{N}} \cap k[\bold{X}]$ is a Gröbner basis for $I \cap k[\bold{X}]$.
\end{proposition}
\begin{proof}
By Proposition \ref{prop:buchberger_equals_fallingroof}(ii) the graph $\mathscr{N}^+$ is linear and by construction it contains no roofs. Suppose there were a vertex with nonzero valence in $\mathscr{N}^+$ labelled by $Y \in \bold{Y}$. By examination of the Falling Roofs algorithmn it is easy to see that $Y$ must have had nonzero valence in $\mathscr{S}$. It belongs to a connected component of $\mathscr{S}$ of the form
\begin{equation}\label{eq:segment_Y}
\xymatrix{
Z_{i_1} \ar@{-}[r] & \cdots \ar@{-}[r] & Y \ar@{-}[r] & \cdots \ar@{-}[r] & Z_{i_l}
}
\end{equation}
where $Z_{i_1}, Z_{i_l} \in \bold{X}$ by (i). By Lemma \ref{lemma:valence_one_pres} the vertices $Z_{i_1}, Z_{i_l}$ have valence one in $\mathscr{N}^+$ and again by examining the Algorithm we can see that $Y$ is still connected to these vertices in $\mathscr{N}^+$ (one proves by induction that if two vertices connected by an unoriented path in $\mathscr{S}$ have nonzero valence in $\mathscr{N}^+$ then they are still connected by an unoriented path in $\mathscr{N}^+$). That is, we have also in $\mathscr{N}^+$ a connected component that looks like \eqref{eq:segment_Y}.

Now we invoke a discrete version of Rolle's theorem: since \eqref{eq:segment_Y} starts low (at $Z_{i_1}$) ends low (at $Z_{i_l}$) and is high somewhere in between (at $Y$) it must have a local maxima (that is, a roof) as is easily proven by induction. But this contradicts the non-existence of live roofs in $\mathscr{N}$. Hence all vertices with nonzero valence in $\mathscr{N}^+$ are labelled by $\bold{X}$. 

For the last statement, since $G_{\mathscr{N}} \cap k[\bold{X}]$ is a Gröbner basis for $I \cap k[\bold{X}]$ it suffices for the last claim to prove that 
\begin{equation}\label{eq:equal_as_sets_gs}
G_{\mathscr{N}^+} = G_{\mathscr{N}} \cap k[\bold{X}]\,.
\end{equation}
Suppose $f \in G_{\mathscr{N}} \cap k[\bold{X}]$ is $f = f_e$ for an edge $e: X_i \lto X_j$ which is not live. In the Falling Roofs algorithm the only way $e$ can be marked as dead is if at some point during the algorithm there are two live edges with target $X_j$. Since by hypothesis every $\bold{Y}$ variable is greater in the $<$-order than every $\bold{X}$ variable, the source of both of these edges with target $X_j$ must be in $\bold{X}$. But this is impossible by (ii). This contradiction proves \eqref{eq:equal_as_sets_gs}.
\end{proof}

\section{Monomial Orders}\label{sec:monomial_orders_before_cut}

Let $\pi$ be a proof net with single conclusion $A$. We learned in Proposition \ref{prop:permutation} that the variables of $P_\pi$ may be organised into sequences called persistent paths. It is natural to use these sequences to define a monomial order on $P_\pi$. Let
\begin{equation}\label{eq:persistent_paths0}
\mathscr{P}_1,\ldots,\mathscr{P}_m
\end{equation}
be the persistent paths of $\pi$ ordered so that if $U_i$ is the last unoriented atom in $\mathscr{P}_i$ then $U_1, \ldots, U_m$ is the order that these atoms appear in $A$. Let us name the variables $X_i$ so that \eqref{eq:persistent_paths0} is $X_1, \ldots, X_n$ and $P_\pi = k[X_1,\ldots,X_n]$.

\begin{defn}\label{defn:monomial_order0} We write $U <_0 V$ if $U$ is before $V$ in \eqref{eq:persistent_paths0}. The monomial order $<_0$ on $P_\pi$ is the lexicographic order determined by $<_0$ on the variables.
\end{defn}

In this section the monomial order on $P_\pi$ is $<_0$.


\begin{lemma}\label{lemma:LMdiff} If $f,g$ are distinct elements of $G_\pi$ then $\LM(f) \neq \LM(g)$, where leading monomials are defined with respect to $<_0$.
\end{lemma}

Recall the relation $\sim$ on the set $U_\pi$ of unoriented atoms of $\pi$, which generates the equivalence relation $\approx$ whose equivalence classes are persistent paths. We have $U \sim V$ if $U,V$ appear next to each other on a persistent path.

\begin{defn}\label{defn:s_0} Let $\mathscr{S}_0$ be the oriented graph with vertex set $U_\pi$ where two variables $U, V \in U_\pi$ are connected by an edge $e: U \lto V$ if $U \sim V$ and $U <_0 V$.
\end{defn}

In the terminology of Section \ref{section:falling_roofs} this is a linear $<_0$-graph. This is a subgraph of the realisation $\mathscr{R}_{<_0}$.

\begin{defn}\label{defn:sequence_gpi0} We denote by $G^{(0)}_\pi$ the ordered set of polynomials $G_{\mathscr{S}_0}$ of Definition \ref{defn:gen_set_allowed},
\begin{equation}
G^{(0)}_\pi = \big\{ V - U \l e: U \lto V \text{ is an edge in } \mathscr{S}_0 \big\}\,.
\end{equation}
\end{defn}

Up to signs, $G^{(0)}_\pi$ is just $G_\pi$.

\begin{lemma}\label{lemma:edivesG0} If $f,g$ are distinct elements of $G = G^{(0)}_\pi$ then $S(f,g) \longrightarrow_G 0$.
\end{lemma}
\begin{proof}
	Suppose $f = X - X', g = Y - Y'$. By Lemma \ref{lemma:LMdiff} we may assume $X >_0 Y$. Then
	\begin{align*}
		S(f,g) &= \frac{XY}{X}(X - X') - \frac{XY}{Y}(Y - Y')\\
		&= Y(X - X') - X(Y - Y')\\
		&= XY' - X'Y
	\end{align*}
	Hence $\LT(S(f,g)) = XY'$. Rearranging
	\begin{align*}
		&= Y'( X - X' ) - X'( Y - Y' )\\
		&= Y'f - X' g
	\end{align*}
	and $\LT(Y'f) = XY'$, $\LT(X'g) = X'Y < XY'$ so $S(f,g) \longrightarrow_G 0$.
\end{proof}

\begin{proposition}\label{prop:minimal_grob} $G^{(0)}_\pi$ is a minimal Gröbner basis for $I_\pi$ with respect to $<_0$.
\end{proposition}
\begin{proof}
To prove $G = G^{(0)}_\pi$ is a Gröbner basis we need to (by \cite[\S 2.9 Theorem 3]{Grobner}) prove that $S(f,g) \longrightarrow_G 0$ for distinct elements $f,g \in G$. But this is Lemma \ref{lemma:edivesG0}. To prove $G$ is a minimal Gröbner basis we can treat each persistent path $\mathscr{P} = X_1, \ldots, X_r$ separately. The part of $G$ associated to $\mathscr{P}$ is $\{ X_{i+1} - X_i \}_{i=1}^{r-1}$. With $f = X_{i+1} - X_i$ it is clear that $\operatorname{LC}(f) = 1$ and $\langle \operatorname{LT}( G \setminus \{ f \}) \rangle = \langle X_2, \ldots, X_{i}, X_{i+2}, \ldots, X_r \rangle$ which does not contain $\LT(f)$.
\end{proof}

\begin{cor} $G_\pi$ is a Gröbner basis for $I_\pi$.
\end{cor}

\begin{remark} While $G^{(0)}_\pi$ is minimal, it is not the \emph{reduced} Gröbner basis of $I_\pi$ as soon as there is a persistent path with more than two variables in $\pi$. The reduced Gröbner basis for $I_\pi$ is the union, over all persistent paths $\mathscr{P}$ in the above notation, of $X_2 - X_1, \ldots, X_i - X_1, \ldots, X_r - X_1$.
\end{remark}

\subsection{Reduction and monomial orders}\label{sec:monomial_order}

In this section $\Gamma: \pi \lto \pi'$ is a reduction sequence between proof nets with single conclusion $A$. Let $\mathscr{P}_1,\ldots,\mathscr{P}_m$ denote the persistent paths of $\pi$ ordered as in \eqref{eq:persistent_paths0} and let us name the variables $X_i$ so that $\mathscr{P}_1,\ldots,\mathscr{P}_m$ is $X_1, \ldots, X_n$ and $P_\pi = k[X_1,\ldots,X_n]$. That is, we name the variables according to their $<_0$-order.

Let $\mathscr{Q}_i$ be the subsequence of $\mathscr{P}_i$ consisting just of those unoriented atoms in $\pi'$ (those in the image of $T_\Gamma$) so that $\mathscr{Q}_1, \ldots, \mathscr{Q}_m$ is a partition of the unoriented atoms of $\pi'$. Let $\mathscr{P}_i \setminus \mathscr{Q}_i$ denote the complement of the subsequence $\mathscr{Q}_i$. Then
\begin{equation}\label{eq:order_gamma}
\mathscr{Q}_1, \ldots, \mathscr{Q}_m, \mathscr{P}_1 \setminus \mathscr{Q}_1,\ldots,\mathscr{P}_m \setminus \mathscr{Q}_m
\end{equation}
is the set of unoriented atoms of $\pi$ arranged in an order that depends on the reduction $\Gamma$. Note that $\mathscr{P}_i \setminus \mathscr{Q}_i$ are the variables in $\mathscr{P}_i$ eliminated during the reduction sequence.

\begin{defn}\label{defn:monomial_order} We write $U <_\Gamma V$ if $U$ is before $V$ in \eqref{eq:order_gamma}, reading from left to right. The monomial order $<_\Gamma$ on $P_\pi$ is the lexicographic order determined by $<_\Gamma$ on the variables.
\end{defn}


From the point of view of the realisation $\mathscr{R}_{<_\Gamma}$ of Section \ref{section:falling_roofs} the order $<_\Gamma$ ``lifts'' the variables in $P_\pi$ that are to be eliminated above those that persist in $P_{\pi'}$. Observe that this realisation refers to the relation between $<_0$ and $<_\Gamma$ (Remark \ref{remark:xaxis_vs_yaxis}).

\begin{defn}\label{defn:s_gamma} Let $\mathscr{S}_\Gamma$ be the oriented graph with vertex set $U_\pi$, where two variables $U, V \in U_\pi$ are connected by an edge $e: U \lto V$ if $U \sim V$ and $U <_\Gamma V$.
\end{defn}

In the terminology of Section \ref{section:falling_roofs} this is a linear $<_\Gamma$-graph which is a subgraph of the realisation $\mathscr{R}_{<_\Gamma}$. Note that the connected components of $\mathscr{S}_\Gamma$ correspond to persistent paths, since unoriented atoms in different persistent paths are not related by $\sim$.

\begin{defn}\label{defn:sequence_gpi} Let $G^{(\Gamma)}_\pi$ be the ordered set of polynomials $G_{\mathscr{S}_\Gamma}$ of Definition \ref{defn:gen_set_allowed},
\begin{equation}
G^{(\Gamma)}_\pi = \big\{ V - U \l e: U \lto V \text{ is an edge in } \mathscr{S}_\Gamma \big\}\,.
\end{equation}
\end{defn}

\begin{example}\label{example:canonical_detour_graph} Let $\gamma$ be the reduction of the $m$-redex of the canonical detour $\pi$ in Example \ref{example:canonical_detour}. The sequence \eqref{eq:order_gamma} is
\[
\mathscr{Q}, \underline{\mathscr{P} \setminus \mathscr{Q}} = X_1, X_2, X_3, X_6, X_7, X_{10}, X_{11}, X_{12}, \underline{X_4, X_5, X_8, X_9}
\]
where we underline the atoms in $\mathscr{P} \setminus \mathscr{Q}$ for the reader's convenience. The graph $\mathscr{S}_\gamma$ is shown in Figure \ref{figure:s_gamma_deotour}.

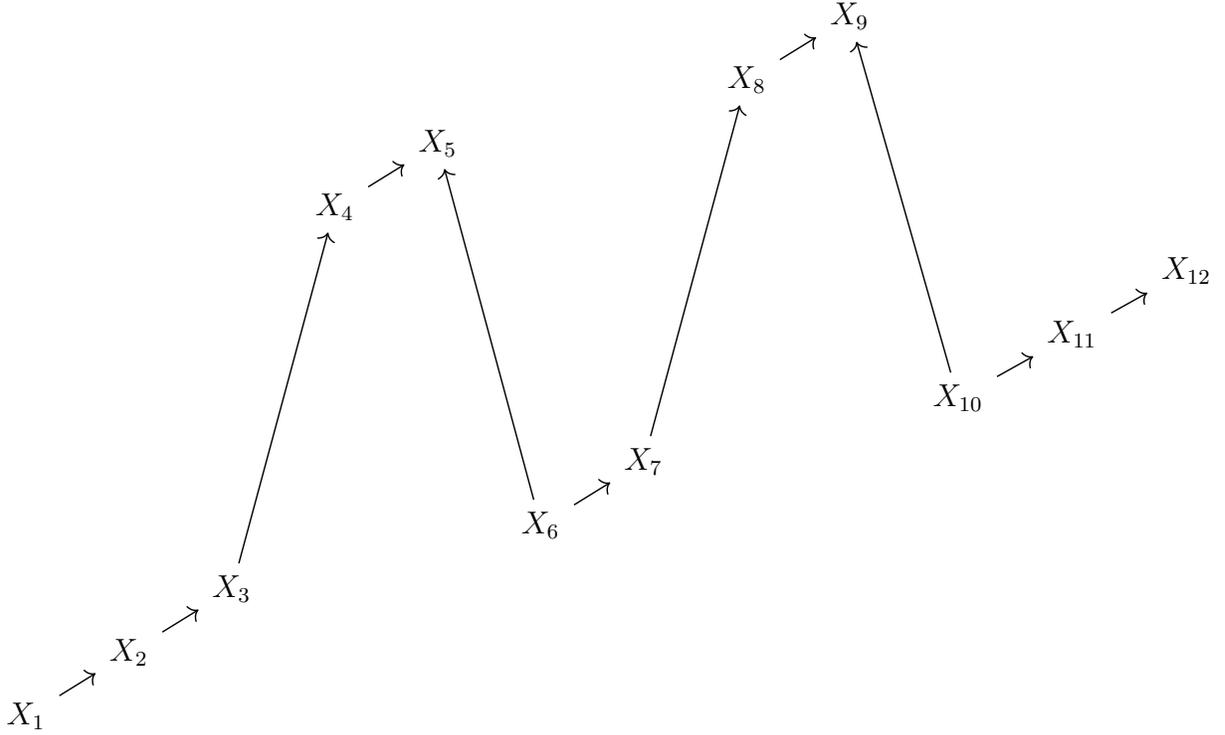
\begin{figure}[h]
\begin{center}
\begin{tikzcd}[row sep = tiny, column sep = small]
     &     &     &     &     &     &     &     & X_9 &                 \\
     &     &     &     &     &     &     & X_8                         \\
     &     &     &     & X_5                                           \\
     &     &     & X_4                                                 \\
     &     &     &     &     &     &     &     &     &     &     & X_{12}\\
     &     &     &     &     &     &     &     &     &     & X_{11}    \\
     &     &     &     &     &     &     &     &     & X_{10}          \\
     &     &     &     &     &     & X_7                               \\
     &     &     &     &     & X_6                                     \\
     &     & X_3\\
     & X_2\\
X_1
	\arrow[from=2-8,to=1-9]
	\arrow[from=4-4,to=3-5]
	\arrow[from=6-11,to=5-12]
	\arrow[from=7-10,to=6-11]
	\arrow[from=7-10,to=1-9]
	\arrow[from=8-7,to=2-8]
	\arrow[from=9-6,to=8-7]
	\arrow[from=9-6,to=3-5]
	\arrow[from=10-3,to=4-4]
	\arrow[from=11-2,to=10-3]
	\arrow[from=12-1,to=11-2]
\end{tikzcd}
\end{center}
\caption{The graph $\mathscr{S}_\gamma$ for the canonical detour of Example \ref{example:canonical_detour}.}
\label{figure:s_gamma_deotour}
\end{figure}
The sequence $G^{(\gamma)}_\pi$ is
\begin{align*}
& X_2 - X_1, X_3 - X_2, X_7 - X_6, X_{11} - X_{10}, X_{12} - X_{11},\\
&\qquad X_4 - X_3, X_5 - X_6, X_5 - X_4, X_8 - X_7, X_9 - X_{10}, X_9 - X_8\,.
\end{align*}
Note that the persistent path $\mathscr{P}$ passes through the $\cut$ link twice, which we can see clearly if we underline the atoms in $\mathscr{P} \setminus \mathscr{Q}$:
\begin{equation}\label{eq:persistent_path_good_ex}
X_1, X_2, X_3, \underline{X_4}, \underline{X_5}, X_6, X_7, \underline{X_8}, \underline{X_9}, X_{10}, X_{11}, X_{12}\,.
\end{equation}
Each pass contributes two consecutive unoriented atoms in $\mathscr{P} \setminus \mathscr{Q}$. Note that the sequence \eqref{eq:persistent_path_good_ex}, in which the atoms are ordered by appearance in the persistent path, is $<_\gamma$ increasing except at two points of ``defect'' where the path exits an edge incident at a $\cut$ link:
\begin{equation}
X_1 <_\gamma X_2 <_\gamma X_3 <_\gamma \underline{X_4} <_\gamma \underline{X_5} >_\gamma X_6 <_\gamma X_7 <_\gamma \underline{X_8} <_\gamma \underline{X_9} >_\gamma X_{10} <_\gamma X_{11} <_\gamma X_{12}\,.
\end{equation}
These are of course the roofs in $\mathscr{S}_{\gamma}$.
\end{example}


\begin{lemma} $G^{(\Gamma)}_\pi$ is not a Gröbner basis for $I_\pi$ with respect to $<_\Gamma$.
\end{lemma}

From this it is also clear that $G_\pi$ is not a Gröbner basis with respect to $<_\Gamma$.

\begin{remark}\label{remark:order_on_persistent_paths} The definition of the monomial orders $<_0, <_\Gamma$ involves ordering the persistent paths by the order of appearance of their last unoriented atom in $A$. This order could be chosen arbitrarily without changing our results; some choice must be made and this one appears relatively natural. The ordering on the persistent paths matters insofar as the Falling Roofs Algorithm chooses roofs in earlier persistent paths before later ones.
\end{remark}

\section{Main Theorems}\label{section:main_theorem}


Let $\pi$ be a proof net with single conclusion $A$. We wrote informally in Section \ref{section:intro_example} about how the structure of $\pi$ implicitly defines a relationship between unoriented atoms on the ``boundary'' of the proof net, and how computation is the process of making this implicit relationship explicit. Let us now make these comments more precise.

Let $\bold{U}$ be the subsequence of positive unoriented atoms in $A$, and $\bold{V}$ the subsequence of negative unoriented atoms in $A$. The composite
\begin{equation}\label{eq:quotient_map_kernel}
k[\bold{U},\bold{V}] \lto P_\pi \lto P_\pi/I_\pi
\end{equation}
is by Proposition \ref{prop:permutation} a surjective morphism of $k$-algebras, and we denote by $I_{\partial \pi}$ the ideal which is its kernel. There is a commutative diagram
\begin{equation}
\xymatrix@C+2pc@R+2pc{
k[\bold{U}, \bold{V}] \ar[d]\ar[r]^-{\operatorname{inc}} & P_\pi \ar[d]\\
k[\bold{U}, \bold{V}]/I_{\partial \pi} \ar[r]_-{\cong} & P_\pi/I_\pi
}
\end{equation}
By Proposition \ref{prop:permutation}(iii)
\begin{equation}
I_{\partial \pi} = \big\langle \{ U_i - V_{\sigma i} \}_i \big\rangle
\end{equation}
for a permutation $\sigma$ which contains the structural information of the normal form of $\pi$ (see Appendix \ref{section:mult_goi0}). The structure of $\pi$, as represented say by the generating set $G_\pi$ for the ideal $I_\pi$, contains the same information but in an implicit form; the computational process of reduction of $\pi$ to normal form may thus be identified with a process of discovering the generating set $\{ U_i - V_{\sigma i} \}_i$ for $I_{\partial \pi}$ from the generating set $G_\pi$ for $I_\pi$.

There are many such processes: indeed, we can now sketch a machine for generating them. Let $<$ be a total order on the set of unoriented atoms in $\pi$
\[
U_\pi = \bold{U} \sqcup \bold{V} \sqcup \bold{W} \text{ with } \bold{V} < \bold{U} < \bold{W}
\]
in the sense that for any $i,j,k$ we have $V_i < U_j < W_k$. Let $\mathscr{S} = \mathscr{S}_{<}$ denote the graph (see Section \ref{sec:monomial_orders_before_cut}) with an edge $X \lto X'$ if $X \sim X'$ and $X < X'$. Then $G_{\mathscr{S}}$ is, up to signs, the generating set $G_\pi$ of $I_\pi$. As our computational process we take the Buchberger Algorithm with Early Stopping applied to $G_{\mathscr{S}}$, or what is the same, the Falling Roofs Algorithm applied to $\mathscr{S}$. These algorithms return $G_{\mathscr{N}}$ and $\mathscr{N}$ respectively (Proposition \ref{prop:buchberger_equals_fallingroof}). Let $\mathscr{N}^+$ be the subgraph of live edges in $\mathscr{N}$.

\begin{thm} $G_{\mathscr{N}^+} = \big\{ U_i - V_{\sigma i} \big\}_i$ is a Gröbner basis for $I_{\partial \pi}$.
\end{thm}
\begin{proof}
The graph $\mathscr{S}$ satisfies the hypotheses of Proposition \ref{prop:cutting_is_living} with $\bold{X} = \bold{V} \cup \bold{U}$, so the subgraph $\mathscr{N}^+$ of live edges in $\mathscr{N}$ satisfies that $G_{\mathscr{N}^+} = G_{\mathscr{N}} \cap k[\bold{U}, \bold{V}]$ is a Gröbner basis of $I_\pi \cap k[\bold{U},\bold{V}] = I_{\partial \pi}$. For each $i$ there is a persistent path beginning at $V_{\sigma i}$ and ending at $U_i$, and these endpoints maintain valence one in the graph of live edges throughout the Falling Roofs algorithm; since only vertices within the same persistent path may be connected by an edge in $\mathscr{N}$ this means that only edges in $\mathscr{N}^+$ are of the form $V_{\sigma i} \lto U_i$, as claimed.
\end{proof}

In this section we answer two further questions:
\begin{itemize}
\item Is there a choice of total order on $U_\pi$ under which the Gröbner basis algorithms match the familiar process of cut-elimination as represented by the reduction rules of Definition \ref{def:reduction}? Theorem \ref{thm:elimination_ours} answers this in the affirmative.

\item The normal form of $\pi$ contains more unoriented atoms than just $\bold{U}, \bold{V}$, so is there a choice of total order on $U_\pi$ for which the Gröbner basis algorithms produce the normal proof net $\widetilde{\pi}$ with its ideal $I_{\widetilde{\pi}}$?	
\end{itemize}

\subsection{Elimination vs Cut-Elimination}
	
Let $\Gamma: \pi \lto \pi'$ be a reduction sequence between proof nets with single conclusion $A$. We view $P_{\pi'}$ as a subring of $P_\pi$ via the algebra homomorphism $T_\Gamma$ of Definition \ref{defn:reduction_sequence}. Recall the following notation:
\begin{itemize}
\item $<_\Gamma$ is a monomial order; variables eliminated by $\Gamma$ are ``large'' (Definition \ref{defn:monomial_order}).
\item $G_{\pi}^{(\Gamma)}$ is a generating set for $I_\pi$ (Definition \ref{defn:sequence_gpi}). Up to signs it is $G_\pi$.
\item $\mathbb{B}_{es}(G_{\pi}^{(\Gamma)}, <_\Gamma)$ is a Gröbner basis for $I_\pi$ obtained by applying the Buchberger with Early Stopping algorithm (Algorithm \ref{alg:elimination}) to $G_{\pi}^{(\Gamma)}$.
\item $G_{\pi'}^{(0)}$ is a (minimal) Gröbner basis for $I_{\pi'}$ (Definition \ref{defn:sequence_gpi0}).
\end{itemize}

\begin{thm}\label{thm:elimination_ours}
	There is an equality of sets
	\begin{equation}\label{eq:thm_elimination}
		G_{\pi'}^{(0)} = \mathbb{B}_{es}(G_{\pi}^{(\Gamma)}, <_\Gamma) \cap P_{\pi'}\,.
	\end{equation}
\end{thm}

\begin{proof}
By definition $G^{(\Gamma)}_\pi = G_{\mathscr{S}}$ where $\mathscr{S}$ is the graph of Definition \ref{defn:s_gamma}. If $\mathscr{N}$ is the output of the Falling Roofs algorithm applied to $\mathscr{S}$ then by Proposition \ref{prop:buchberger_equals_fallingroof}
\begin{equation}
\mathbb{B}_{es}(G_{\pi}^{(\Gamma)}, <_\Gamma) = G_{\mathscr{N}}\,.
\end{equation}
We are in the situation of Section \ref{section:falling_roofs_elimination} where the $\bold{X}$-variables are those in $P_{\pi'}$. The monomial order $<_\Gamma$ satisfies the conditions of Proposition \ref{prop:cutting_is_living}, from which we deduce that $G_{\mathscr{N}^+} = G_{\mathscr{N}} \cap P_{\pi'}$ is a Gröbner basis for $I_{\pi'}$. Since the $\bold{X}$ variables are ordered by $<_\Gamma$ in $P_{\pi}$ as they are by $<_0$ in $P_{\pi'}$ it is clear that $\mathscr{N}^+$ is just the graph $\mathscr{S}_0$ of Definition \ref{defn:s_0} and hence $G_{\mathscr{N}^+} = G^{(0)}_{\pi'}$.
\end{proof}

The core idea of the theorem can be seen in the following example.

\begin{example} Let $\pi$ denote the canonical detour as in Example \ref{example:canonical_detour}. We continue Example \ref{example:canonical_detour_graph} by presenting the Falling Roofs algorithm on the graph of Figure \ref{figure:s_gamma_deotour}. The result of the algorithm is the graph $\mathscr{N}$ in Figure \ref{figure:s_gamma_deotour2}. Note that in the context of Section \ref{section:falling_roofs_elimination} if we partition the variables of our polynomial ring
\[
P_\pi = k[X_1, X_2, X_3, \underline{X_4}, \underline{X_5}, X_6, X_7, \underline{X_8}, \underline{X_9}, X_{10}, X_{11}, X_{12}]
\]
into the set $\bold{X}$ of those variables \emph{without} underlines and the set $\bold{Y}$ of those variables \emph{with} underlines, then ``cutting off the top'' of the graph $\mathscr{N}$ means cutting it at a height with all the underlined variables above and all the other variables below. Recall from Proposition \ref{prop:buchberger_equals_fallingroof} that the Buchberger algorithm with Early Stopping on $G^{(\gamma)}_\pi$ computes $\mathscr{N}$:
\[
\mathbb{B}_{es}(G^{(\gamma)}_\pi, <_\gamma) = G_{\mathscr{N}}\,.
\]
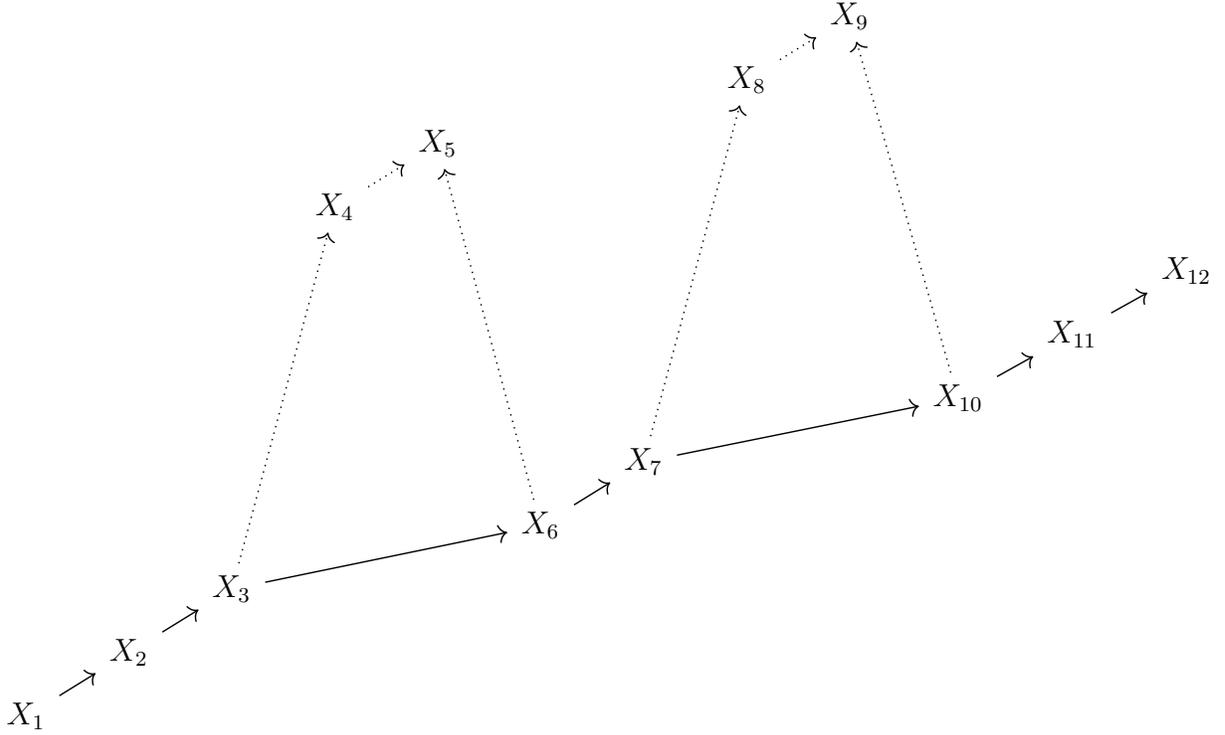
\begin{figure}
\begin{center}
\begin{tikzcd}[row sep = tiny, column sep = small]
     &     &     &     &     &     &     &     & X_9 &                 \\
     &     &     &     &     &     &     & X_8                         \\
     &     &     &     & X_5                                           \\
     &     &     & X_4                                                 \\
     &     &     &     &     &     &     &     &     &     &     & X_{12}\\
     &     &     &     &     &     &     &     &     &     & X_{11}    \\
     &     &     &     &     &     &     &     &     & X_{10}          \\
     &     &     &     &     &     & X_7                               \\
     &     &     &     &     & X_6                                     \\
     &     & X_3\\
     & X_2\\
X_1
	\arrow[dotted, from=2-8,to=1-9]
	\arrow[dotted, from=4-4,to=3-5]
	\arrow[from=6-11,to=5-12]
	\arrow[from=7-10,to=6-11]
	\arrow[dotted, from=7-10,to=1-9]
	\arrow[dotted, from=8-7,to=2-8]
	\arrow[from=9-6,to=8-7]
	\arrow[dotted, from=9-6,to=3-5]
	\arrow[dotted, from=10-3,to=4-4]
	\arrow[from=11-2,to=10-3]
	\arrow[from=12-1,to=11-2]
	\arrow[from=10-3,to=9-6]
	\arrow[from=8-7,to=7-10]
\end{tikzcd}
\end{center}
\caption{The Falling Roofs algorithm on the canonical detour.}
\label{figure:s_gamma_deotour2}
\end{figure}
The Elimination Theorem (Theorem \ref{thm:elimination}) says that
\begin{align*}
G_{\mathscr{N}} \cap k[\bold{X}] &= \big\{ X_2 - X_1, X_3 - X_2, X_6 - X_3, X_7 - X_6,\\
&\qquad X_{10} - X_7, X_{11} - X_{10}, X_{12} - X_{11} \big\}
\end{align*}
is a Gröbner basis for $I_\pi \cap k[\bold{X}]$. Note that eliminating the cut in the canonical detour leads to the proof net $\pi'$, the main part of which is shown in Figure \ref{figure:cut_reduction_detour}. By construction $T_\gamma: P_{\pi'} \lto P_{\pi}$ is the inclusion of
\[
P_{\pi'} = k[\bold{X}] = k[X_1, X_2, X_3, X_6, X_7, X_{10}, X_{11}, X_{12}]
\]
into $P_\pi$ (note that $X_1, X_{12}$ are not shown in the Figure for compactness). 
\end{example}

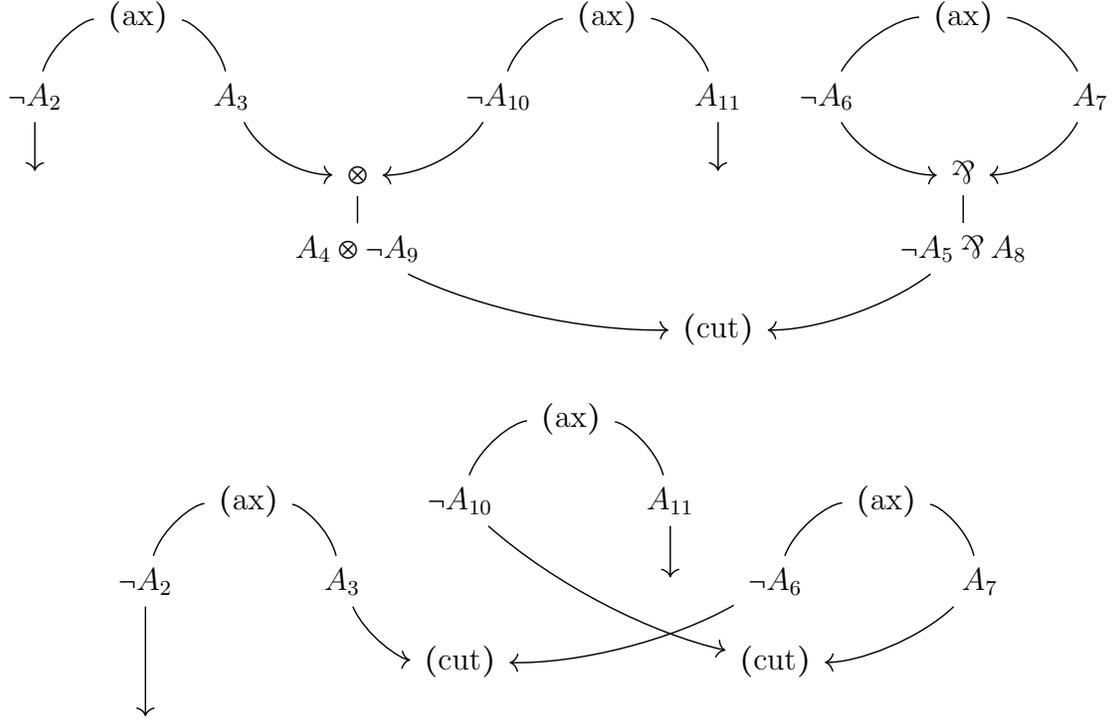
\begin{figure}
\begin{center}
	\[\begin{tikzcd}[column sep = tiny, row sep = small]
		& \ax &&&& \ax &&& \ax \\
		{\neg A_2} && A_3 && \neg A_{10} && A_{11} & \neg A_6 && A_7 \\
		\, &&& \otimes &&& \, && \parr \\
		&&& A_4 \otimes \neg A_9 &&&&& \neg A_5 \parr A_8 \\
		&&&&&& \cut
		\arrow[curve={height=12pt}, no head, from=1-2, to=2-1]
		\arrow[curve={height=-12pt}, no head, from=1-2, to=2-3]
		\arrow[from=2-1, to=3-1]
		\arrow[curve={height=12pt}, no head, from=1-6, to=2-5]
		\arrow[curve={height=-12pt}, no head, from=1-6, to=2-7]
		\arrow[from=2-7, to=3-7]
		\arrow[curve={height=-12pt}, from=2-5, to=3-4]
		\arrow[curve={height=12pt}, from=2-3, to=3-4]
		\arrow[curve={height=12pt}, no head, from=1-9, to=2-8]
		\arrow[curve={height=-12pt}, no head, from=1-9, to=2-10]
		\arrow[curve={height=-12pt}, from=2-10, to=3-9]
		\arrow[curve={height=12pt}, from=2-8, to=3-9]
		\arrow[no head, from=3-9, to=4-9]
		\arrow[curve={height=-12pt}, from=4-9, to=5-7]
		\arrow[curve={height=12pt}, from=4-4, to=5-7]
		\arrow[no head, from=3-4, to=4-4]
	\end{tikzcd}\]
\[\begin{tikzcd}[column sep = tiny, row sep = small]
	&&&&& \ax \\
	& \ax &&& \neg A_{10} && A_{11} && \ax \\
	{\neg A_2} && A_3 &&&& \, & {\neg A_6} && A_7 \\
	&&&& \cut &&& \cut \\
	\,
	\arrow[curve={height=12pt}, no head, from=2-2, to=3-1]
	\arrow[curve={height=-12pt}, no head, from=2-2, to=3-3]
	\arrow[from=3-1, to=5-1]
	\arrow[curve={height=12pt}, no head, from=1-6, to=2-5]
	\arrow[curve={height=-12pt}, no head, from=1-6, to=2-7]
	\arrow[from=2-7, to=3-7]
	\arrow[curve={height=12pt}, no head, from=2-9, to=3-8]
	\arrow[curve={height=-12pt}, no head, from=2-9, to=3-10]
	\arrow[curve={height=12pt}, from=3-3, to=4-5]
	\arrow[curve={height=-12pt}, from=3-8, to=4-5]
	\arrow[curve={height=12pt}, from=2-5, to=4-8]
	\arrow[curve={height=-12pt}, from=3-10, to=4-8]
\end{tikzcd}\]
\end{center}
\caption{Cut reduction on the canonical detour.}
\label{figure:cut_reduction_detour}
\end{figure}

\subsection{Normal Buchberger and Normal Forms}

In Theorem \ref{thm:elimination_ours} we made a connection between our modification to the Buchberger algorithm (Algorithm \ref{alg:elimination}) and single step cut reduction. We now examine the ``standard'' form of Buchberger's Algorithm \cite[\S 2.10 Theorem 9]{Grobner} and show that it relates to \emph{normalisation}, that is, a maximal sequence of cut reductions yielding a normal form. In this section $\pi$ denotes a proof net with single conclusion $A$. 

\begin{lemma} If $\Gamma, \Omega: \pi \lto \pi'$ are reduction sequences between proof nets with $\pi'$ cut-free then $T_\Gamma = T_\Omega$.
	\end{lemma}
\begin{proof}
Let the reduction sequences be
		
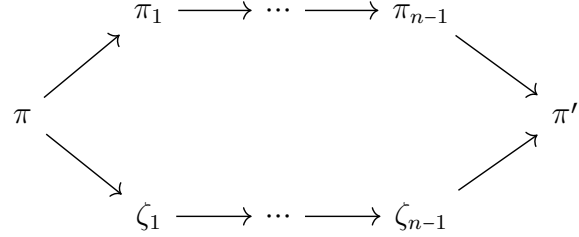
\[\begin{tikzcd}
			& {\pi_1} & \cdots & {\pi_{n-1}} \\
			{\pi} &&&& {\pi'} \\
			& {\zeta_1} & \cdots & {\zeta_{n-1}}
			\arrow[from=3-4, to=2-5]
			\arrow[from=2-1, to=3-2]
			\arrow[from=1-2, to=1-3]
			\arrow[from=1-3, to=1-4]
			\arrow[from=1-4, to=2-5]
			\arrow[from=3-2, to=3-3]
			\arrow[from=3-3, to=3-4]
			\arrow[from=2-1, to=1-2]
		\end{tikzcd}\]
	with $\Gamma$ on top and $\Omega$ on the bottom. Recall that any two reduction paths have the same length \cite[Corollary 3.0.7]{Troiani}. Cut reduction is confluent \cite[Proposition 3.0.3]{Troiani}, so it suffices to show that any commutative diagram of reductions
	\begin{equation}
		\begin{tikzcd}
			\zeta_1\arrow[r,"{\gamma_1}"]\arrow[d,swap,"{\gamma_2}"] & \zeta_2\arrow[d,"{\gamma_3}"]\\
			\zeta_3\arrow[r,swap,"{\gamma_4}"] & \zeta_4
			\end{tikzcd}
		\end{equation}
	induces a commutative diagram
	\begin{equation}\label{eq:induced_confluence}
		\begin{tikzcd}[column sep = large, row sep = large]
			P_{\zeta_1} & P_{\zeta_2}\arrow[l,swap,"{T_{\gamma_1}}"]\\
			P_{\zeta_3}\arrow[u,"{T_{\gamma_2}}"] & P_{\zeta_4}\arrow[u,swap,"{T_{\gamma_3}}"]\arrow[l,"{T_{\gamma_4}}"]
			\end{tikzcd}
		\end{equation}
	For any reduction $\gamma: \pi \lto \pi'$, the schematics used to define $T_\gamma$ indicate a mapping from the edges of $\pi$ to the edges of $\pi'$. Commutativity of \eqref{eq:induced_confluence} amounts to confluence of this induced map. This follows from inspection of Figure \ref{figure:tgamma}.
	\end{proof}
	
Let $\Gamma$ denote any reduction path from $\pi$ to its normal form $\widetilde{\pi}$ (the unique cut-free multiplicative proof net equivalent under cut reduction to $\pi$). By the previous lemma the $k$-algebra morphism
\[
T = T_\Gamma: T_{\widetilde{\pi}} \lto T_\pi
\]
is independent of the choice of $\Gamma$ and in this section we simply denote it by $T$. There is a simple characterisation of the variables of $P_\pi$ in the image of $T$. Recall that a proof net is an oriented multigraph, and a \emph{directed path} in $\pi$ means a sequence of edges traversed in the same direction as the orientation.

\begin{lemma}\label{lem:elimination_characterisation}
An unoriented atom $U$ of $\pi$ is in the image of $T$ if and only if it is \emph{above a conclusion} in the sense that the unique maximal directed path in $\pi$ starting at the edge associated to $U$ ends at a conclusion.
	\end{lemma}
\begin{proof}
Let $\rho$ be the maximal directed path whose first edge is the one associated to $U$. First we prove that if $\rho$ ends at a conclusion, then every unoriented atom of every formula labelling an edge in $\rho$ is in the image of $T$.
	
	Write $\rho = (e_1,\ldots, e_n)$ and notice first that none of the edges $e_2,\ldots, e_n$ can appear in any redex. So we only need to consider $e_1$, which is labelled by $A$. If $e_1$ is part of a redex, it is necessarily part of an $a$-redex. Inspection of \eqref{eq:a_redex} shows that reducing this $a$-redex results in a proof with the edge labelled $A$ still present. This resulting edge now either may be part of an $a$-redex or not. Reducing any other $a$-redex not involving the edge labelled $A$ does not remove this edge labelled $A$. Reducing any $m$-redex also does not remove this edge labelled $A$. It follows that the edge labelled $A$ appears in the normal form $\pi'$. That is, $U$ is in the image of $T$.
	
	Now we prove that if $\rho$ ends at a cut link then none of the unoriented atoms of $A$ are in the image of $T$. We proceed by induction on the length $n$ of $\rho$. Say $n = 1$, if the edge labelled $A$ is premise to a cut link, then it is part of an $a$-redex which when reduced, removes the edge labelled $A$. Thus $U$ cannot be in the image of $T$.
	
	Now assume that $n > 1$ and the result has been proved in the setting where $\rho$ is of length $n-1$. If $A$ is above a cut link, then reducing this cut link results in a proof net where $A$ is still above a cut link, the result follows directly from the inductive hypothesis.
	\end{proof}
	
Note that an unoriented atom of $\pi$ which is not above a conclusion must be \emph{above a cut} in the sense that the unique maximal directed path in $\pi$ starting at the edge associated to the atom ends at a cut link. These are the variables eliminated by $\Gamma$, or what is the same, the variables \emph{not} in the image of $T$.

\begin{lemma}\label{lem:r1_r2}
	Let $\scr{P} = (Z_1,\ldots, Z_r)$ be a persistent path in $\pi$ which traverses a cut link. Then there exists a (necessarily unique) pair of indices $1 < t_1, t_2 < r$ such that $Z_i$ is above a cut if and only if $t_1 \le i \le t_2$.
	\end{lemma}
\begin{proof}
	Let $Z$ be a variable above a conclusion. By Lemma \ref{lem:elimination_characterisation} every unoriented atom of every formula labelling an edge in the directed path starting at the edge of $Z$ and ending at a conclusion is also above a conclusion. Thus, either $(Z_1,\ldots, Z)$ or $(Z, \ldots, Z_r)$ is a subsequence of $\scr{P}$ consisting only of variables above a conclusion. Thus, $t_1$ can be taken to be the maximal integer such that $(Z_1,\ldots, Z_{r_1-1})$ is a subsequence of $\scr{P}$ consisting only of variables above a conclusion, and $t_2$ can be taken to be the maximal integer such that $(Z_{t_2+1},\ldots, Z_r)$ is a subsequence of $\scr{P}$ consisting only of variables above a conclusion.
	\end{proof}

Let $\scr{P}_1,\ldots, \scr{P}_m$ denote the persistent paths of $\pi$, ordered in the way described at the beginning of Section \ref{sec:monomial_order}. Let $\mathscr{Q}_i$ denote the subsequence of $\mathscr{P}_i$ consisting of those unoriented atoms in the normalisation $\widetilde{\pi}$, that is, those atoms in the image of $T$. In light of Lemma \ref{lem:r1_r2} we can divide the sequence $\mathscr{P}_i$ into three parts:
\begin{equation}
\mathscr{P}_i = \mathscr{Q}_i^1, \mathscr{P}_i \setminus \mathscr{Q}_i, \mathscr{Q}_i^2
\end{equation}
where $\mathscr{Q}_i^1$ is the subsequence of $\mathscr{P}_i$ consisting of those unoriented atoms in the persistent path coming \emph{before} all the atoms above a cut, and $\mathscr{Q}_i^2$ is the subsequence of atoms coming \emph{after} all those above a cut. By definition $\mathscr{Q}_i = \mathscr{Q}_i^1 \sqcup \mathscr{Q}_i^2$.

Consider the following permutation of $\mathscr{P}_i$
\begin{equation}\label{eq:normal_order_part}
\widetilde{\mathscr{P}}_i = (\mathscr{Q}_i^1)^{\textrm{op}}, \mathscr{Q}_i^2, (\mathscr{P}_i \setminus \mathscr{Q}_i)^{\textrm{op}}
\end{equation}
where $\mathscr{R}^{\textrm{op}}$ means the reversal of the sequence $\mathscr{R}$. Combining these re-ordered persistent paths we obtain a total order on the set of unoriented atoms of $\pi$:
\begin{equation}\label{eq:normal_order}
\widetilde{\scr{P}}_1, \ldots, \widetilde{\scr{P}}_m\,.
\end{equation}


\begin{defn}\label{def:normal_order}
	We write $U <_n V$ if $U$ is before $V$ in \eqref{eq:normal_order}, reading from left to right. The monomial order $<_n$ on $P_\pi$ is the lexicographic order determined by $<_n$ on the variables.
	\end{defn}

\begin{defn} Let $\mathscr{S}_n$ be the oriented graph with vertex set $U_\pi$, where two variables $U,V \in U_\pi$ are connected by an edge $e: U \lto V$ if $U \sim V$ and $U <_n V$.
\end{defn}

\begin{defn}\label{def:callGpi} Let $G^{(n)}_\pi$ be the ordered set of polynomials $G_{\mathscr{S}_n}$ of Definition \ref{defn:gen_set_allowed},
\begin{equation}
G^{(n)}_\pi = \big\{ V - U \l e: U \lto V \text{ is an edge in } \mathscr{S}_n \big\}\,.
\end{equation}
\end{defn}

\begin{example}\label{example:n_canonical_detour} Let $\Gamma$ be the normalisation of the canonical detour $\pi$ in Example \ref{example:canonical_detour}. Then $X_1, X_2, X_{11}, X_{12}$ are above a conclusion and the rest are above a cut, so $\mathscr{P}$ is
\[
X_1, X_2, \underline{X_3, X_4, X_5, X_6, X_7, X_8, X_9, X_{10}}, X_{11}, X_{12}
\]
where we underline the atoms in $\mathscr{P} \setminus \mathscr{Q}$. Hence $\mathscr{Q}^1$ is $X_1,X_2$, $\mathscr{Q}^2$ is $X_{11}, X_{12}$ and \eqref{eq:normal_order_part} is
\[
(\mathscr{Q}^1)^{\textrm{op}}, \mathscr{Q}^2, (\mathscr{P} \setminus \mathscr{Q})^{\textrm{op}} = X_2, X_1, X_{11}, X_{12}, \underline{X_{10}, X_9, X_8, X_7, X_6, X_5, X_4, X_3}\,.
\]
The graph $\mathscr{S}_n$ is shown in Figure \ref{figure:s_n_deotour}
\end{example}


\begin{figure}[h]
\begin{center}
\begin{tikzcd}[row sep = tiny, column sep = small]
     &     & X_3\\
     &     &     & X_4\\
     &     &     &     & X_5\\
     &     &     &     &     & X_6  \\
     &     &     &     &     &     & X_7                                 \\
     &     &     &     &     &     &     & X_8                          \\
     &     &     &     &     &     &     &     & X_9                     \\
     &     &     &     &     &     &     &     & & X_{10}                \\
     &     &     &     &     &     &     &     & & & & X_{12}                \\
          &     &     &     &     &     &     &     & & & X_{11}                \\
X_1\\
     & X_2
	\arrow[to=6-8,from=7-9]
	\arrow[to=2-4,from=3-5]
	\arrow[from=10-11,to=9-12]
	\arrow[to=8-10,from=10-11]
	\arrow[from=8-10,to=7-9]
	\arrow[to=5-7,from=6-8]
	\arrow[to=4-6,from=5-7]
	\arrow[from=4-6,to=3-5]
	\arrow[to=1-3,from=2-4]
	\arrow[from=12-2,to=1-3]
	\arrow[from=12-2,to=11-1]
\end{tikzcd}
\end{center}
\caption{The graph $\mathscr{S}_n$ for the canonical detour of Example \ref{example:n_canonical_detour}.}
\label{figure:s_n_deotour}
\end{figure}

We identify $P_{\widetilde{\pi}}$ as a subring of $P_\pi$ using $T$.

\begin{thm}\label{thm:execution}
	Let $\pi$ be a proof net and $\widetilde{\pi}$ the normal form. Let $\mathbb{B}(G^{(n)}_\pi, <_n)$ denote the result of applying the Buchberger algorithm. Then there is an equality of sets
	\begin{equation}\label{eq:execution}
		G^{(n)}_{\widetilde{\pi}} = \mathbb{B}(G^{(n)}_\pi, <_n) \cap P_{\widetilde{\pi}}\,.
		\end{equation}
	\end{thm}
\begin{proof}
We write $G^{(n)}_\pi = \{ f_1, \ldots, f_s \}$. Buchberger's Algorithm \cite[\S 2.10 Theorem 9]{Grobner} begins by choosing a pair $(i,j)$ with $i < j$ and considering the $S$-polynomial $S(f_i, f_j)$. The only interesting case is when $f_i, f_j$ share a leading term, so we assume this below.
	
The polynomials $f_i, f_j$ must lie on a common persistent path $\scr{P}$ which traverses a cut. We have that $f_i = Z_{r_1+1} - Z_{r_1 + 2}$ and $f_j = Z_{r_1+1} - Z_{r_1}$ for some unoriented atoms in the persistent path $\mathscr{P} = Z_1, \ldots, Z_l$ and
	\begin{align*}
		S(f_i, f_j) &= (Z_{r_1+1} - Z_{r_1 + 2}) - (Z_{r_1+1} - Z_{r_1})\\
		&= Z_{r_1 } - Z_{r_1 + 2}\,.
		\end{align*}
	We now divide by the sequence $(f_1,\ldots, f_s)$.
	\begin{equation}
		\begin{array}{rll}
			(f_1,...,f_s) & \showdiv{Z_{r_1} - Z_{r_1 + 2}}\\
			&\hspace{0.45em} Z_{r_1 + 3} - Z_{r_1 + 2}\\
			\CMidRule[0.0ex][18.0ex]{2-2}
			&\hspace{0.45em} \hphantom{Y' - {}}  Z_{r_1 + 3} - Z_{r_1 }\\
			& \hspace{0.45em} \hphantom{Y' - {}} Z_{r_1 + 3} - Z_{r_1 + 4}\\
			\CMidRule[6.0ex][10.0ex]{2-2}
			& \hspace{0.45em} \hphantom{Y' - Z_{r_1 + 4} - {}} Z_{r_1 + 4} - Z_{r_1}\\
			& \hspace{0.45em} \hphantom{Y' - Z_{r_1 + 4} - Z_{r_1 + 4}} \vdots\\
			\CMidRule[15.0ex][6.0ex]{2-2}
			& \hspace{0.45em} \hphantom{Y' - Z_{r_1 + 4}  + Z_{r_1 + 4}} Z_{r_2-1} - Z_{r_1}\\
			& \hspace{0.45em} \hphantom{Y' - Z_{r_1 + 4}  + Z_{r_1 + 4}} Z_{r_2-1} - Z_{r_2}\\
			\CMidRule[18.0ex][0.0ex]{2-2}
			& \hspace{0.45em} \hphantom{Y' - Z_{r_1 + 4} + Z_{r_1 + 4} + Z_{r_2-1}} Z_{r_2} - Z_{r_1 }  
		\end{array}
	\end{equation}
	The division algorithm necessarily terminates here as there is no polynomial in $G^{(n)}_\pi$ with leading term either $Z_{r_2}$ or $Z_{r_1}$. It remains to show that $Z_{r_2} - Z_{r_1}$ appears in $G^{(n)}_{\pi'}$.
	
	The sequence $(Z_{r_1}, Z_{r_1+1}, \ldots, Z_{r_2-1}, Z_{r_2})$ is a subsequence of $\scr{P}$ and there is an associated persistent path $\scr{P}'$ in $\pi'$ for which $(Z_{r_1}, Z_{r_2})$ is a subsequence. By the way that $<_n$ is defined, we have that $Z_{r_2} - Z_{r_1} \in G^{(n)}_{\pi'}$. This proves the containment $\supseteq$ in \eqref{eq:execution}.
	
	Next we show the reverse inclusion. Let $f \in G^{(n)}_{\pi'}$ be given. If $f \in G^{(n)}_\pi$ then there is nothing to show, as $G^{(n)}_\pi \subseteq \mathbb{B}(G^{(n)}_\pi, <_n)$. So suppose $f \notin G^{(n)}_\pi$. Then $f = X' - X$ for some persistent path in $\pi$ with subsequence $(X, W_1, \ldots, W_t, X')$ where $W_1,\ldots, W_t$ are variables not in $P_{\pi'}$. Let $g = W_1 - W_2$ and $g' = W_1 - X$. Then $g, g' \in G^{(n)}_\pi$ and the first half of this proof shows that division of $S(g,g')$ by $G^{(n)}_{\pi}$ yields $f$, so $f \in \mathbb{B}(G^{(n)}_\pi, <_n)$.
	\end{proof}

\section{Conclusion}

Let us summarise the dictionary between proofs and ideals developed in this paper. We have shown how to associate to a proof net $\pi$ in multiplicative linear logic a triple
\begin{equation}\label{eq:proof_to_ideal}
\pi \longmapsto (P_\pi, G_\pi, <_0)
\end{equation}
consisting of a polynomial ring $P_\pi$, a generating set $G_\pi$ for an ideal $I_\pi$ in this ring, and a monomial order $<_0$ on $P_\pi$. We have shown that $G_\pi$ is a Gröbner basis for $I_\pi$ with respect to this monomial order (Proposition \ref{prop:minimal_grob}). 

There are many other monomial orders on $P_\pi$. With respect to some monomial orders $G_\pi$ will be ``stable'' in the sense that it is a Gröbner basis, while it will be ``unstable'' (not a Gröbner basis) with respect to others and running the Buchberger algorithm makes nontrivial changes. The correspondence in \eqref{eq:proof_to_ideal} maps proof nets to stable triples. However, if $\pi$ contains $\cut$ links then we can destabilise it by using a monomial order $<_\Gamma$ derived from some reduction sequence. The ``restabilisation'' of $(P_\pi, G_\pi, <_\Gamma)$ is the process of running the Buchberger algorithm and intersecting with the subring $P_{\pi'}$ of ``low'' variables with respect to $<_\Gamma$. Our main Theorem says that this process ends with the image of the reduction $\pi'$ under the map \eqref{eq:proof_to_ideal}, or in other words the following diagram commutes:
\begin{equation}\label{eq:summary_comm_dia}
\xymatrix@C+1pc@R+1pc{
\pi \ar[r]^-{\eqref{eq:proof_to_ideal}}\ar[ddd]_-{\gamma} & (P_\pi, G_\pi, <_0) \ar[r]^-{\text{destabilise}} & (P_\pi, G_\pi, <_\Gamma) \ar[d]^-{\text{Buchberger}}\\
& & (P_\pi, G_{\mathscr{N}}, <_\Gamma) \ar[d]^-{\text{intersection}}\\
& & (P_{\pi'}, G_{\mathscr{N}} \cap P_{\pi'}, <_0) \ar@{=}[d]^-{\text{Theorem } \ref{thm:elimination_ours}}\\
\pi' \ar[rr]_-{\eqref{eq:proof_to_ideal}} & & (P_{\pi'}, G_{\pi'}, <_0)
}
\end{equation}
The language of stability is best understood in the context of the realisations: the combinatorial $1$-manifolds together with a discrete Morse function (``height on the page'') that we have associated to any pair $(G_\pi, <)$ where $<$ is a lexicographic monomial order. If the realisation contains no local maximum (that is, a roof) then by Corollary \ref{cor:falling_roofs_computes_grobner}, $G_\pi$ is a Gröbner basis with respect to $<$. Otherwise the Buchberger algorithm (or equivalently Falling Roofs) proceeds to ``straighten out'' the pair consisting of the $1$-manifold with its Morse function, until the subgraph of live edges contains no such local maxima. That is, the Falling Roofs algorithm transforms a segment looking like
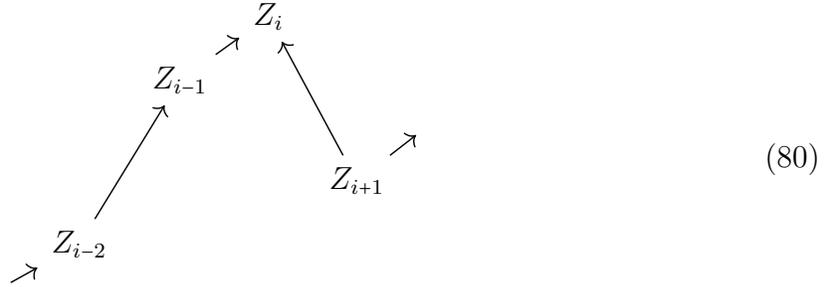
\begin{equation}\label{eq:defect_peak1}
\begin{tikzcd}[row sep = tiny, column sep = tiny]
        &     &     & Z_i                                           \\
        &     & Z_{i-1}                                                 \\
        &     &     &     &     & \,                               \\
        &     &     &     & Z_{i+1}                                     \\
        & Z_{i-2}\\
     \,\\
	\arrow[from=2-3,to=1-4]
	\arrow[from=4-5,to=3-6]
	\arrow[from=4-5,to=1-4]
	\arrow[from=5-2,to=2-3]
	\arrow[from=6-1,to=5-2]
\end{tikzcd}
\end{equation}
into a segment looking like
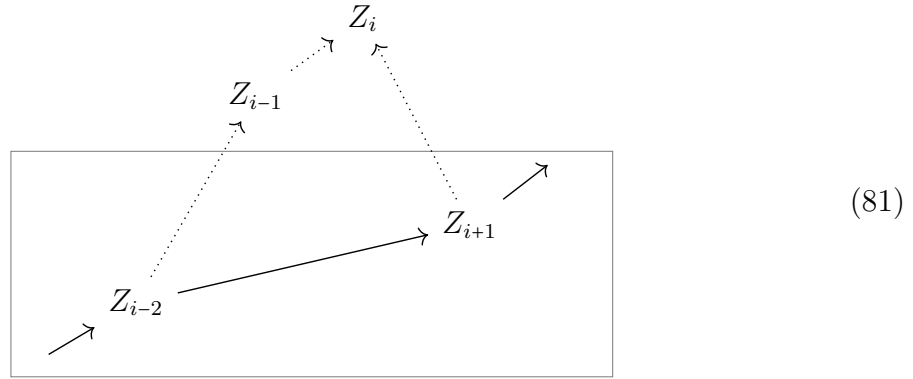
\begin{equation}\label{eq:defect_peak2}
\tikz[
overlay]{
    \filldraw[fill=white,draw=blue!50!yellow] (-0.2,-2.2) rectangle (7.8,0.8);
}
\begin{tikzcd}[row sep = small, column sep = small]
        &     &     & Z_i                                           \\
        &     & Z_{i-1}                                                 \\
        &     &     &     &     & \,                               \\
        &     &     &     & Z_{i+1}                                     \\
        & Z_{i-2}\\
     \,\\
	\arrow[dotted,from=2-3,to=1-4]
	\arrow[from=5-2,to=4-5]
	\arrow[from=4-5,to=3-6]
	\arrow[dotted,from=4-5,to=1-4]
	\arrow[dotted,from=5-2,to=2-3]
	\arrow[from=6-1,to=5-2]
\end{tikzcd}
\end{equation}

The algebro-geometric procedure in the right column of \eqref{eq:summary_comm_dia} and the topological procedure in \eqref{eq:defect_peak1},\eqref{eq:defect_peak2} represent our attempt to find mathematical shadows of the process of cut-elimination process for multiplicative proof nets.

\subsection{Towards Geometry}\label{section:towards_geometry}

Let $\pi$ be a proof net with ideal $I_\pi \subseteq P_\pi$. The algebraic set
\[
\mathbb{V}(I_\pi) = \{ a \in \mathbb{A}^\pi \l f(a) = 0 \text{ for all } f \in I_\pi \}
\]
is an intersection of pairwise diagonals, which is not geometrically very interesting. More geometry can be introduced into the picture by considering pairs $\mathbb{V}(I_\pi) \subseteq V$ where $V$ is an an algebraic set. For example, given any $N \ge 2$ the polynomial $W = \sum_{f \in G_{\pi}} f^N$ is contained in $I_\pi$ and so we have an inclusion
\[
\mathbb{V}(I_\pi) \subseteq \mathbb{V}(W)\,.
\]
The subscheme $\mathbb{V}(I_\pi)$ naturally determines an object in the bounded derived category of coherent sheaves of the singular hypersurface $\mathbb{V}(W)$, and in closely related categories like the category of matrix factorisations of $W$. In sequels to the present paper we will explore this idea as a way of associating objects in triangulated categories to proofs; the role of the Buchberger algorithm in the present paper is played by constructions in homological algebra such as the perturbation lemma.

\appendix

\section{Geometry of Interaction: Multiplicatives}\label{section:mult_goi0}
	In \cite{Troiani} a permutation is associated to each proof net, as was originally done by Girard \cite{Girard}. In short, these permutations act on the disjoint union of all the unoriented atoms of all the conclusions to axiom links of $\pi$. The permutation of interest here is $\alpha_\pi$ (Definition \cite[4.1.3]{Troiani}) which permutes two unoriented atoms $X,Y$ if and only if $X,Y$ are respectively unoriented atoms of formulas $A,\neg A$ where $A,\neg A$ are the conclusions of some axiom link in $\pi$, and $X,Y$ appear at the same index of the \emph{sequences} of unoriented atoms of $A,\neg A$. This permutation is equal to $\sigma$ which is introduced in Proposition \ref{prop:permutation}.
	\begin{proposition}\label{prop:GoI0_relation}
		Let $\pi$ be a proof net with single conclusion $A$ with sequence of oriented atoms given by: $\big((U_{1},u_1),...,(U_n,u_n)\big)$ and let $\sigma$ be as given in Proposition \ref{prop:permutation}. Let $\delta_\pi$ be the permutation defined in \cite[Definition 4.2.10]{Troiani}. Then for all $i= 1,...,n$ we have:
		\begin{equation}
			\delta_{\pi}(U_i) = U_{\sigma(i)}
		\end{equation}
	\end{proposition}
	\begin{proof}
		First we consider the special case when there are no $m$-redexes in $\pi$, and all conclusions of all axiom links are atomic. We will use some structure of such proof nets proved in \cite{Troiani}. First, we observe that all reductions are of the following form (see \cite[Lemma 4.2.5]{Troiani}).
		\begin{equation}\label{eq:redex_form}
			\begin{tikzcd}[column sep = small, row sep = small]
				& \ax &&&& \ax &&&& \ax \\
				X && {\neg X} && X && {\neg X} && X && {\neg X} \\
				\vdots &&& \cut &&& \vdots && \vdots && \vdots
				\arrow[curve={height=12pt}, no head, from=1-2, to=2-1]
				\arrow[curve={height=-12pt}, no head, from=1-2, to=2-3]
				\arrow[from=2-1, to=3-1]
				\arrow[curve={height=12pt}, from=2-3, to=3-4]
				\arrow[curve={height=-12pt}, from=2-5, to=3-4]
				\arrow[curve={height=12pt}, no head, from=1-6, to=2-5]
				\arrow[curve={height=-12pt}, no head, from=1-6, to=2-7]
				\arrow[from=2-7, to=3-7]
				\arrow[curve={height=12pt}, no head, from=1-10, to=2-9]
				\arrow[curve={height=-12pt}, no head, from=1-10, to=2-11]
				\arrow[from=2-11, to=3-11]
				\arrow[from=2-9, to=3-9]
			\end{tikzcd}
		\end{equation}
		Hence, all the cut links of $\pi$ appear inside ``chains" of axiom and cut links, as in the following Diagram.
		
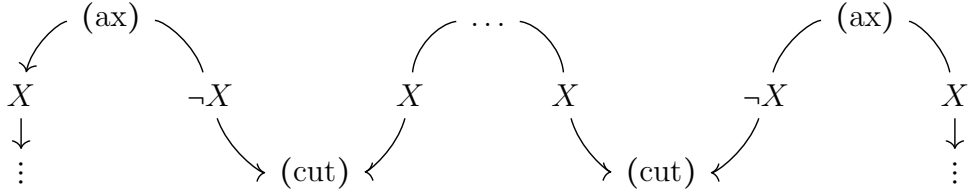
\[\begin{tikzcd}[column sep = tiny, row sep = small]
			& \ax &&&& \ldots & {} &&& \ax \\
			X && {\neg X} && X && X && {\neg X} && X \\
			\vdots &&& \cut &&&& \cut &&& \vdots
			\arrow[curve={height=-12pt}, no head, from=1-10, to=2-11]
			\arrow[curve={height=12pt}, no head, from=1-10, to=2-9]
			\arrow[curve={height=-12pt}, from=2-9, to=3-8]
			\arrow[curve={height=12pt}, from=2-7, to=3-8]
			\arrow[curve={height=-12pt}, no head, from=1-6, to=2-7]
			\arrow[curve={height=12pt}, no head, from=1-6, to=2-5]
			\arrow[curve={height=-12pt}, from=2-5, to=3-4]
			\arrow[curve={height=12pt}, from=2-3, to=3-4]
			\arrow[curve={height=-12pt}, no head, from=1-2, to=2-3]
			\arrow[curve={height=12pt}, from=1-2, to=2-1]
			\arrow[from=2-1, to=3-1]
			\arrow[from=2-11, to=3-11]
		\end{tikzcd}\]
	Also, by Proposition \cite[4.1.7]{Troiani}, a cut-free proof net with sole conclusion $B$ satisfying the property that every conclusion of every axiom link is atomic, is unique up to axiom links. Since each chain of axiom and cut links can be replaced by a single axiom link, it follows from this Proposition that $\pi$ is uniquely determined up to these chains. The proof net $\pi$ is thus a root labelled $\operatorname{c}$ with an incoming edge labelled $A$, the body of $\pi$ is a finite binary tree uniquely determined by $A$, with some choice of chains connecting the leaves of the body of $\pi$. See Example \cite[4.1.9]{Troiani} for an example. Consider the subgraph of $\pi$ given by removing the chains, we notice that the order of the atomic propositions at the leaves (read from left to right) is the same as the order of the sequence of oriented atoms of $A$. Hence, $U_{\sigma(i)}$ is the unoriented atom corresponding to an atomic proposition connected via a chain to the formula whose corresponding unoriented atom is $U_{i}$.
	
	On the other hand, as described in the proof of Theorem \cite[4.2.11]{Troiani}, the permutation $\delta_\pi$ is a product of transpositions, where two unoriented formulas $X,Y$ are swapped if and only if they are at the extreme ends of a common chain. Hence, $\delta_{\pi}(U_i) = U_{\sigma(i)}$.
	
	Now we consider the case where $\pi$ may admit $m$-redexes, we still assume that every conclusion of every axiom link is atomic though. Let $\pi'$ denote the proof net obtained by reducing all $m$-redexes, by confluence of cut-reduction (Proposition \cite[3.0.3]{Troiani}), $\pi'$ is independent of the order in which these $m$-redexes are reduced. We claim that $\delta_{\pi} = \delta_{\pi'}$ and also that the permutation $\sigma$ obtained from $\pi$ is equal to the permutation $\sigma'$ obtained from $\pi'$. Both of these facts follow from Lemma \cite[4.2.4]{Troiani}.
	
	Finally, we consider the general case. In \cite[Definition 3.0.9]{Troiani} we defined an \emph{$\eta$-increx} to be a subgraph of a proof net of the following form.
	\begin{equation}
		\begin{tikzcd}[column sep = small, row sep = small]
			& \ax \\
			{A \otimes B} && {\neg B \parr \neg A} \\
			\vdots && \vdots
			\arrow[curve={height=-12pt}, no head, from=1-2, to=2-3]
			\arrow[curve={height=12pt}, no head, from=1-2, to=2-1]
			\arrow[from=2-1, to=3-1]
			\arrow[from=2-3, to=3-3]
		\end{tikzcd}
	\end{equation}
We can replace this by the subgraph \eqref{eq:eta_expansion} below, and then repeat that process to eventually obtain a proof net $\pi'$ such that every every conclusion of every axiom link of $\pi'$ is atomic.
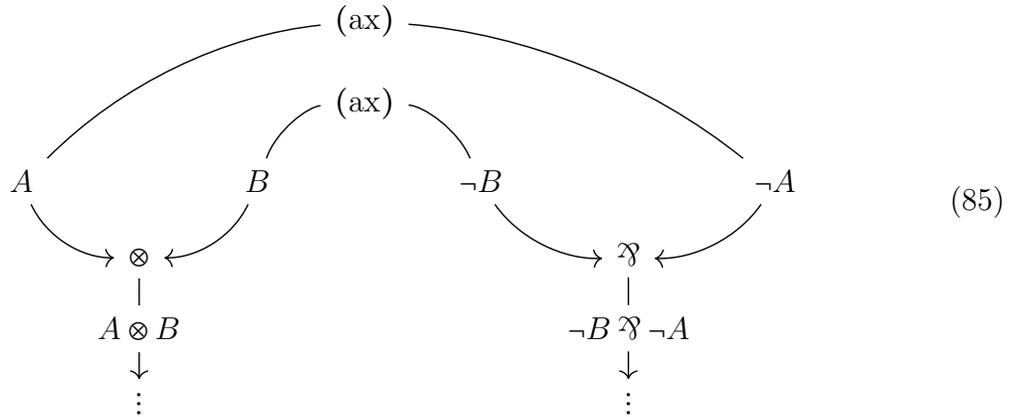
\begin{equation}\label{eq:eta_expansion}
\begin{tikzcd}[column sep = small, row sep = small]
	&&& \ax \\
	&&& \ax \\
	A && B & & \neg B & & \neg A\\
	& \otimes &&&& \parr\\
	& A \otimes B &&&& \neg B \parr \neg A\\
	& \vdots &&&& \vdots
	\arrow[curve={height=18pt}, no head, from=1-4, to=3-1]
	\arrow[curve={height=-18pt}, no head, from=1-4, to=3-7]
	\arrow[curve={height=12pt}, no head, from=2-4, to=3-3]
	\arrow[curve={height=-12pt}, no head, from=2-4, to=3-5]
	\arrow[curve={height=12pt}, from=3-1, to=4-2]
	\arrow[curve={height=-12pt}, from=3-3, to=4-2]
	\arrow[curve={height=12pt}, from=3-5, to=4-6]
	\arrow[curve={height=-12pt}, from=3-7, to=4-6]
	\arrow[no head, from=4-2, to=5-2]
	\arrow[no head, from=4-6, to=5-6]
	\arrow[from=5-2, to=6-2]
	\arrow[from=5-6, to=6-6]
\end{tikzcd}
\end{equation}
One must check that $\pi'$ is independent of the choice of $\eta$-expansion procedure, and also that this process terminates. This is done just after Lemma 3.0.10 in \cite{Troiani}. It now remains to show that $\delta_{\pi} = \delta_{\pi'}$ and that the permutation $\sigma$ obtained from $\pi$ is equal to the permutation $\sigma'$ obtained from $\pi'$. Both of these facts follow from inspection of \eqref{eq:eta_expansion}.
		\end{proof}

\end{document}